\documentclass[10pt]{amsart}
\usepackage[utf8]{inputenc}
\usepackage[utf8]{inputenc}
\usepackage{amsmath}
\usepackage{amsthm}
\usepackage{mathtools}
\usepackage{amsfonts}
\usepackage{caption}
\usepackage{graphicx, subfig}
\usepackage{amssymb}
\usepackage[utf8]{inputenc}
\usepackage{amsthm}
\usepackage{graphicx} 
\usepackage{float} 
\usepackage{tikz}
\usepackage{diagram}
\usepackage[all]{xy}
\usepackage{cite}

\newtheorem{dummy1}{dummy1}[section]
\newtheorem{dummy}{dummy}[subsection]
\newtheorem{lemma}[dummy]{Lemma}
\newtheorem{theorem}[dummy]{Theorem}

\newtheorem{corollary}[dummy]{Corollary}
\newtheorem{proposition}[dummy]{Proposition}
\newtheorem*{proposition*}{Proposition}
\newtheorem*{corollary*}{Corollary}
\newtheorem*{theorem*}{Theorem}
\theoremstyle{definition}
\newtheorem{definition}[dummy]{Definition}

\newtheorem{aremark}[dummy]{Remark}
\newtheorem{theorem1}[dummy1]{Theorem}
\newtheorem{lemma1}[dummy1]{Lemma}
\newtheorem{proposition1}[dummy1]{Proposition}
\newtheorem{remark1}[dummy1]{Remark}
\newtheorem{definition1}[dummy1]{Definition}

\begin{document}
\bibliographystyle{plain}
\title{Tropicalization of the Mirror Curve near the Large Radius Limit}
\author{Zhaoxing Gu}
\date{December 2022}
\maketitle
\begin{center}
\section*{Abstract}    
\end{center} 
The mirror of a toric orbifold is an affine curve called the mirror curve. In this paper, firstly, we recall the basic tools in tropical geometry and give a definition of the mirror curve. Then we calculate the tropical spine of the mirror curve for a smooth toric Calabi-Yau 3-fold near the large radius limit. Finally, we recall a special kind of real algebraic curves called the cyclic-M curve and show that under some special choices of parameters near the large radius limit, the mirror curve is a cyclic-M curve. Applying Mikhalkin's result\cite{mikhalkin2000real} of cyclic-M curves, we show that the mirror curve is glued from tubes and pairs of pants. 

\tableofcontents

\begin{center}
\item \section{Introduction}    
\end{center}
\subsection{Backgrounds of the mirror curve}
The mirror of a toric orbifold is an affine curve called the mirror curve. The mirror curve was studied in mirror symmetry by Hori and Vafa for toric Calabi-Yau 3-folds. Let $(X,\omega)$ be a symplectic toric Calabi-Yau 3-fold. There exist three types of mirrors for $(X,\omega)$: Landau-Ginzburg (Givental) mirror, Hori-Vafa mirror, and the mirror curve \cite{hori2000mirror}\cite{givental1995homological}.

For the toric Calabi-Yau 3-fold $(X,\omega)$, the mirror B-model is a 3-dimensional Landau-Ginzburg model on $(\mathbb{C^*})^3$ given by the superpotential $$W=H(x,y,q)z.$$
Here $H(x,y,q)=0$ is an affine curve with the parameter $q$ related to the Kähler parameter of $X$ given by the mirror map. Another mirror, Hori-Vafa mirror of $(X,\omega)$ is a non-compact Calabi-Yau 3-fold $\Tilde{X}$ with a holomorphic 3-form $\Omega$ on $\Tilde{X}$, where $\Tilde{X}$ is a hypersurface in $\mathbb{C}^2\times(\mathbb{C^*})^2$, given by $$\Tilde{X}=\{(u,v,x,y)\in\mathbb{C}^2\times(\mathbb{C^*})^2|uv=H(x,y)\},$$ and $$\Omega=Res_{\Tilde{X}}(\frac{1}{uv-H(x,y)}du\wedge dv\wedge \frac{dx}{x}\wedge\frac{dy}{y}).$$ These two types of mirrors of a toric Calabi-Yau 3-fold $(X,\omega)$ could be reduced to an affine curve $$C_q=\{(x,y)\in (\mathbb{C^*})^2|H(x,y,q)=0\}$$ with the Kähler parameter $q$. We define the curve by the mirror curve. 

The mirror curve could be used for the prediction of open-closed Gromov-Witten invariants. For example, in \cite{eynard2007invariants}, the Eynard-Orantin topological recursion provides an algorithm for calculating higher genus invariants for a spectral curve. We can relate the Eynard-Orantin invariants $\omega_{g,n}$ of the mirror curve to all genus open-closed Gromov-Witten invariants of the symplectic toric Calabi-Yau 3-fold $(X,\omega)$ by Bouchard-Klemm-Mari\~no-Pasquetti (BKMP) Remodeling Conjecture \cite{bouchard2009remodeling}\cite{bouchard2010topological}\cite{hori2000mirror}\cite{fang2020remodeling}\cite{eynard2007invariants}.

Besides, the statement of a version of homological mirror symmetry for a toric Calabi-Yau 3-fold is formulated via the mirror curve. Specifically, for a toric Calabi-Yau 3-fold, the B-model is the matrix factorization, while the A-model is a Fukaya-type category on its mirror curve \cite{pascaleff2019topological} \cite{abouzaid2013homological}\cite{lee2016homological}\cite{bai2022topological}.

\subsection{Main results} For a simplicial toric Calabi-Yau 3-fold $X_{\Sigma}$ defined by a 3-dim fan $\Sigma\subset N\cong \mathbb{Z}^3$, all the generators of 1-dim cones lie in a hyperplane $N'\subset N$, so a generator corresponds to a lattice point in $N'$. Thus $\Sigma$ gives a triangulation of $\Delta$, where $\Delta$ is the convex hull of these lattice points given by 1-dim cones. We define such triangulation by $T_{\Sigma}$. The first result is that the tropicalization of the mirror curve near the large radius limit gives a subdivision of $\mathbb{R}^2$. This subdivision is exactly dual to the triangulation $T_{\Sigma}$.

\begin{theorem*}[Theorem \ref{thm:subdivision-triangulation}]
The subdivision $T$ given by the tropical spine is dual to the triangulation $T_{\Sigma}$ of $\Sigma$ near the large radius limit.
\end{theorem*}

Besides, we recall propositions of a special kind of real algebraic curves determined by some Laurent polynomials namely cyclic-M curves in a real toric surface. This kind of curve has the maximal number of connected components in the corresponding real toric surface and intersects the axes of the toric surface in a cyclic order. Here the maximal number of connected components equals one plus the number of lattice points inside the Newton polytope of Laurent polynomial by Harnack's inequality\cite{harnack1876ueber}. We give examples of a non-cyclic-M curve and a cyclic-M curve in Figure \ref{f1} and Figure \ref{f2}.

\begin{figure}[htbp]
\centering
\begin{minipage}[t]{0.48\textwidth}
\centering
\includegraphics[width=6cm]{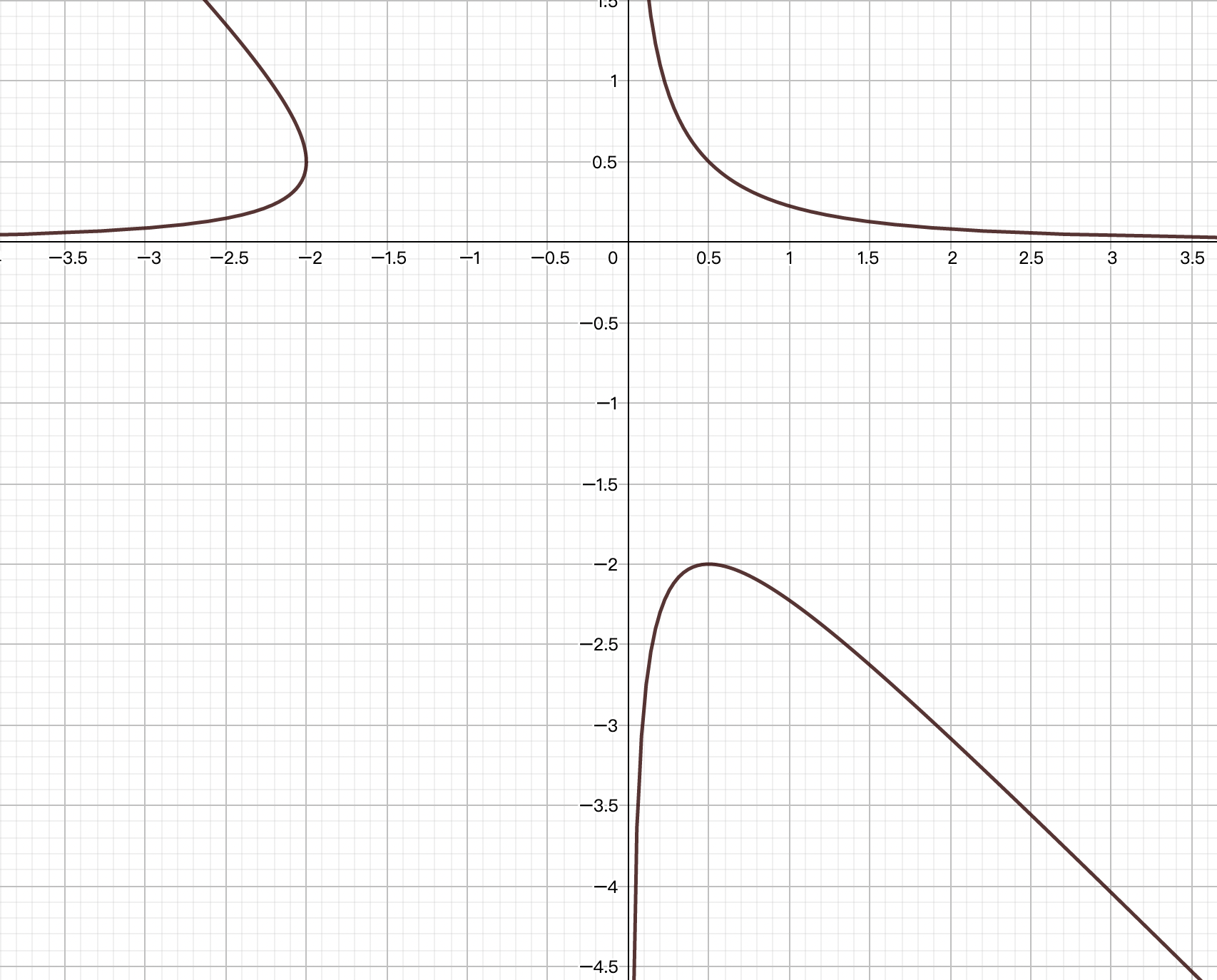}
\caption{Non-cyclic-M curve: $1+x+y-0.5x^{-1}y^{-1}=0$}
\label{f1}
\end{minipage}
\begin{minipage}[t]{0.48\textwidth}
\centering
\includegraphics[width=6cm]{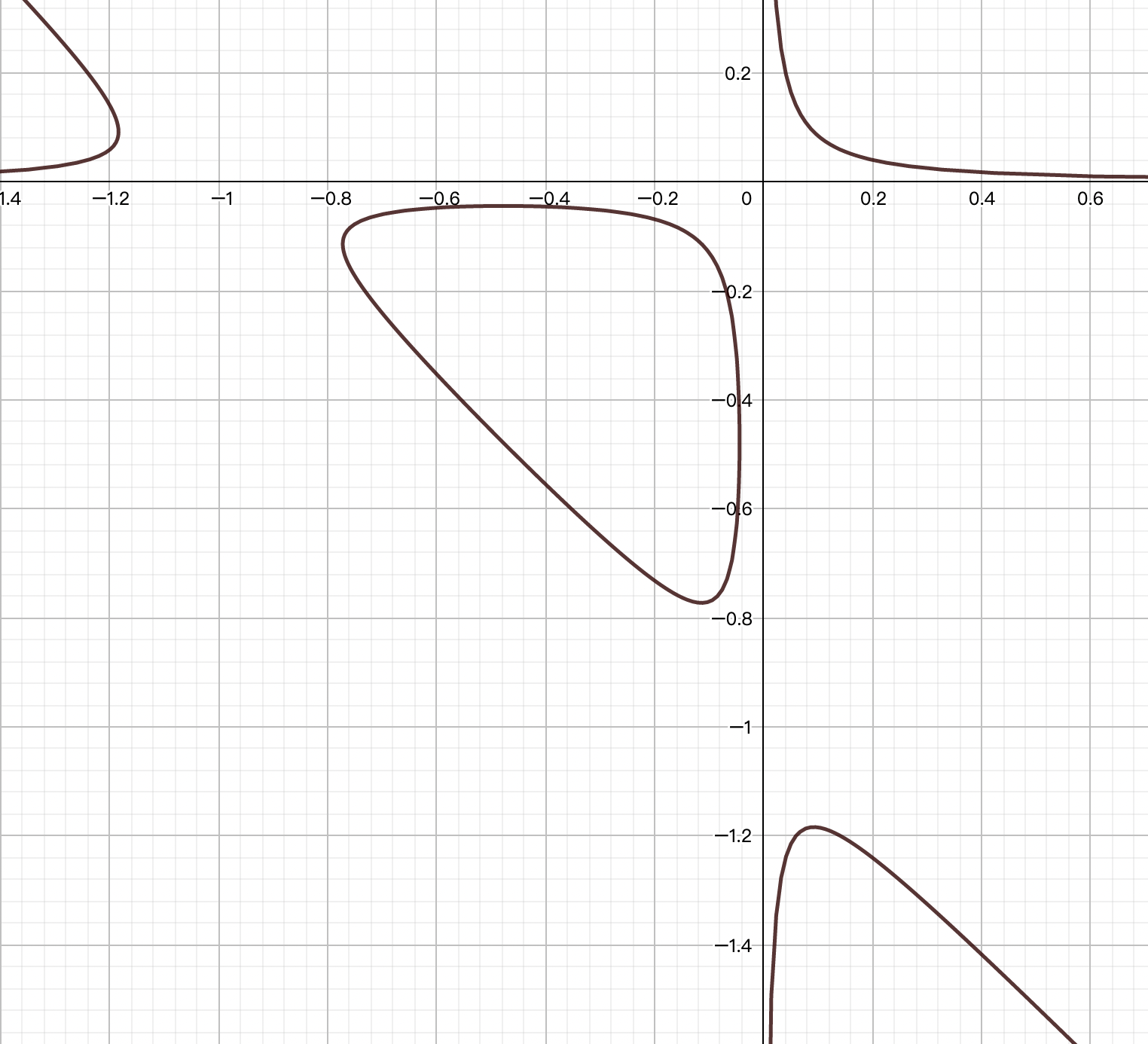}
\caption{cyclic-M curve: $1+x+y-0.01x^{-1}y^{-1}=0$}
\label{f2}
\end{minipage}
\end{figure}

Mikhalkin has studied cyclic-M curves in manners of tropical geometry. In \cite{mikhalkin2000real}, he has proved the following statements about cyclic-M curves:

\begin{proposition*}[Proposition \ref{prop:F=RA}]
Let $p$ be a Laurent polynomial with real coefficients and $A=\{(x,y)\in(\mathbb{C^*})^2|p(x,y)=0\}$. If $\mathbb{R}A=A\cap(\mathbb{R^*})^2$ is a cyclic-M curve, we define $F$ to be the locus of critical points of $\mu|_A$, where $\mu$ is the amoeba map. Then we have $F=\mathbb{R}A$.   
\end{proposition*}

\begin{proposition*}[Proposition \ref{prop: boundary embedding}]
Let $p$ be a Laurent polynomial with real coefficients and $A=\{(x,y)\in(\mathbb{C^*})^2|p(x,y)=0\}$. If $\mathbb{R}A$ is a cyclic-M curve, then $\mu(\mathbb{R}A)=\partial\mu(A)$. Besides, the amoeba map $\mu$ is an embedding on $\mathbb{R}A$.
\end{proposition*}
\begin{theorem*}[Theorem \ref{tp type}]
Let $p$ be a Laurent polynomial with real coefficients and $A=\{(x,y)\in(\mathbb{C^*})^2|p(x,y)=0\}$. If $\mathbb{R}A=A\cap(\mathbb{R^*})^2$ is a cyclic-M curve, the topological type of $(\mathbb{R}X_{\Delta_p};\mathbb{R}A, l_1\cup...\cup l_n)$ is uniquely determined by $\Delta_p$.
\end{theorem*}
The mirror curve is defined by a Laurent polynomial written as $$H(x,y,q)=1+x+y+\sum\limits_{i=4}^{p+3}a_i(q)x^{m_i}y^{n_i}.$$The second result about a mirror curve of a smooth toric Calabi-Yau 3-fold is that under special choices of real parameters, the mirror curve is a cyclic-M curve.
\begin{corollary*}[Corollary \ref{cr5.4}]
Near the large radius limit, with $a_i(q)<0$ when $m_i,n_i$ are both odd and $a_i(q)>0$ when $m_i$ or $n_i$ is even, the mirror curve is a cyclic-M curve.
\end{corollary*}
Finally, we apply Mikhalkin's result to show the mirror curve of smooth toric Calabi-Yau 3-fold is glued from tubes and pairs of pants, which any vertex in the dual graph of $T_{\Sigma}$ corresponds to a pair of pants and any edge corresponds to a tube. This gluing determines the topology of the mirror curve.
\subsection{Outline of the proof}
We prove the results in several steps. Beginning in Section 2, we recall tropical geometry and give a definition of the mirror curve. For the first result, we mainly finish the proof by applying the method of tropical geometry. Specifically, in Section \ref{2.6}, we will see that generally for a Laurent polynomial of $F(z_1,...,z_n)$, we could construct an injective map $\psi_F$ from the connected components of $\mathbb{R}^n-A_F$ to $\Delta_F\cap \mathbb{Z}^n$. Then in Section \ref{3}, we prove that near the large radius limit, for the mirror curve $H(x,y,q)=0$ of a smooth toric Calabi-Yau 3-fold, the previous injective map $\psi_H$ is also surjective by constructing the preimage of any lattice point in $\Delta_H$. Also in Section \ref{3}, by introducing the notation of Ronkin function, we show the amoeba of a Laurent polynomial has a canonical deformation retract named tropical spine. Then near the large radius limit, we could calculate the tropical spine of a mirror curve concretely. Further, such a tropical spine gives a dual subdivision of $T_{\Sigma}$. These steps lead to our first Theorem \ref{thm:subdivision-triangulation}. 

For the second result, we mainly finish the proof with some symmetry propositions of the distribution rules of coefficients for the mirror curve as given in Section \ref{2.5}. In Section \ref{4}, we recall Mikhalkin's result about cyclic-M curves that if $\mathbb{R}A$ is a cyclic-M curve and $A$ is the zeroes of the Laurent polynomial $p$ in $(\mathbb{C^*})^2$, then the amoeba map $\mu_A$ is an embedding on the boundary of amoeba and 2-1 inside the amoeba. This fact explains why we hope the mirror curve is a cyclic-M curve. Because we have already known the deformation retract of a mirror curve's amoeba, if the mirror curve is a cyclic-M curve, we could locally figure out the topology of the mirror curve. However, generally for the real parameters near the large radius limit, the mirror curve is not a cyclic-M curve. In Section \ref{5}, we first give the choice of signs for the parameters in Section \ref{5.1} and explain the motivation for such a choice. Then in Section \ref{5.2}, we prove under such a choice of coefficients, any interior lattice point in Newton polytope $\Delta_H$ corresponds to a unique bounded connected component of mirror curve in $(\mathbb{R}^*)^2$. Thus the mirror curve is an M-curve. Finally, in Section \ref{5.3}, we construct arcs connecting any two adjacent points of intersection of the unbounded component and axes of toric surface $\mathbb{R}X_{\Delta_p}$ by the intermediate value theorem. Then we get Corollary \ref{cr5.4}.
\subsection{Acknowledgement}
I would sincerely appreciate my advisor for the undergraduate research Prof. Bohan Fang for his patient guidance and useful suggestions both on the project itself and further math research. I would also thank Prof. Shuai Guo for his frequent discussions with me about the project, his helpful advice, and for inviting me to give the report on such a research project. Besides, I would wish to thank Ce Ji, Kaitai He, Zhiyuan Zhang, Jialiang Lin, and Zhengnan Chen for their valuable discussions about the project. This project is partially supported by an undergraduate research grant at School of Mathematical Sciences, Peking University.
\begin{center}
\section{Preliminaries on toric varieties, mirror curves, and tropical geometry}
\end{center}
\label{2}
In this section, we recall some results in toric variety and tropical geometry, then give the definition of mirror curve. We mainly follow the notation in \cite{fang2020remodeling}, and we refer to \cite{cox2011toric}\cite{fulton1993introduction} for general toric language.
\subsection{Subdivision, dual subdivision} In this chapter, we mainly introduce some combinatorial languages which are used for the description of the mirror curve and its tropicalization. We follow the symbol of \cite{passare2000amoebas}.

\begin{definition}
For a convex set $K$ in $\mathbb{R}^n$. A collection $T$ of nonempty closed convex subsets of $K$
is called a convex subdivision if it satisfies the following conditions:
\begin{enumerate}
    \item The union of all sets in $T$ is equal to $K$.
    \item If $\sigma$, $\tau\in T$ and $\sigma\cap\tau$ is nonempty, then $\sigma\cap\tau\in T$
    \item If $\sigma\in T$ and $\tau$ is any subset of $\sigma$, then $\tau\in T$ if and only if $\tau$ is a face of $\sigma$.
\end{enumerate} Here a face of a convex set $\sigma$ means a set of the form $\{x\in\sigma|\langle\xi,x\rangle=sup_{y\in\sigma}\langle\xi,y\rangle\}$ for some $\xi\in\mathbb{R}^n$. Since then, we write a subdivision short for a convex subdivision.
\end{definition}

In a subdivision $T$ of $K$, we call $k$-dimensional subsets $k$-cells. If $K$ is a convex polytope, and all cells of $T$ are convex polytopes, we call $T$ a polytopal subdivision. Besides, if $T$ is a subdivision of a 2-dimensional polytope $\Delta$, and all 2-dimensional cells of $T$ are triangles, we call $T$ a triangulation.

If $\tau\subset\sigma$, and $\tau,\sigma$ are both convex sets, we define the cone generated from $\tau$ and $\sigma$ by
$$cone(\tau,\sigma)=\{t(x-y)|x\in\sigma,y\in\tau,t\ge0\}.$$
Clearly, this set is a convex cone. If C is a convex cone, its dual is defined to be the
cone $C^{v}=\{\xi\in\mathbb{R}^n|\langle\xi,x\rangle\le 0,\forall x\in C\}$. Then we will define dual subdivision.

\begin{definition} Let $K$, $K_0$ be convex sets in $\mathbb{R}^n$, and let $T$, $T_0$ be convex subdivisions of $K$, $K_0$. We say that $T$ and $T_0$ are dual (to each other) if there exists a bijective map $T\to T_0$, denoted $\sigma\to\sigma^*$, satisfying the following conditions:
\begin{enumerate}
\item For $\sigma,\tau\in T$, $\tau\subset\sigma$ if and only if $\sigma^*\subset\tau^*$
\item If $\tau\subset\sigma$, then $cone(\tau,\sigma)$ is dual to $cone(\sigma^*,\tau^*)$
\end{enumerate}
\end{definition}
Because $cone(\sigma,\sigma)$ is the affine space spanned by $\sigma$, we get $\sigma$ and $\sigma^*$ are orthogonal. Moreover, dim $\sigma$+dim $\sigma^*=n$. Figure \ref{f3} and Figure \ref{f4} show a subdivision of $\mathbb{R}^2$ and its dual subdivision. The dual subdivision is an example of triangulation. 
\begin{figure}[htbp]
\centering
\begin{minipage}[t]{0.48\textwidth}
\centering
\includegraphics[width=6cm]{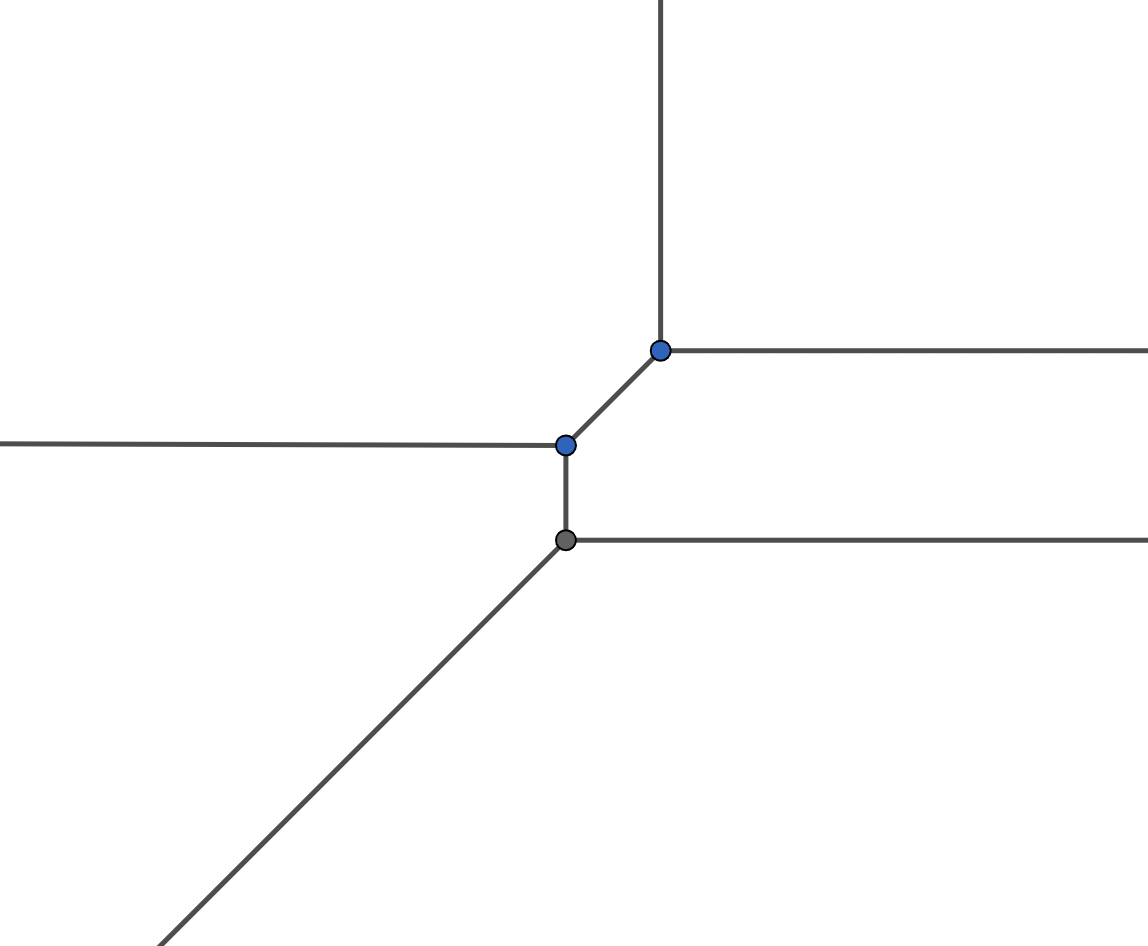}
\caption{Division of $\mathbb{R}^2$}
\label{f3}
\end{minipage}
\begin{minipage}[t]{0.48\textwidth}
\centering
\includegraphics[width=6cm]{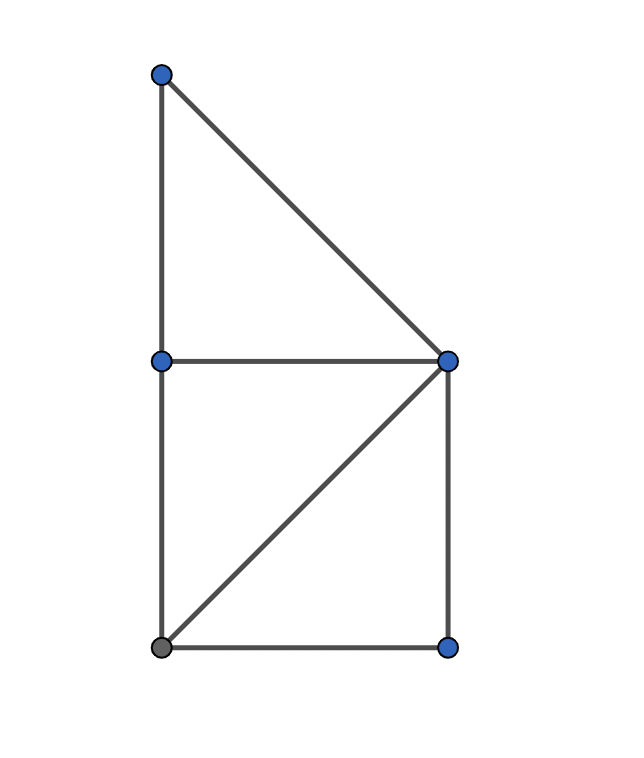}
\caption{Dual Triangulation}
\label{f4}
\end{minipage}
\end{figure}

\subsection{Simplicial toric Calabi-Yau 3-fold and its fan}
In this paper, we only consider the simplicial toric Calabi-Yau 3-fold $X_{\Sigma}$ given by a simplicial fan $\Sigma$ in a lattice $N\cong \mathbb{Z}^3$. For the convention, we denote $Hom(N,\mathbb{Z})$ by $M$, then $M\cong \mathbb{Z}^3$ is also a lattice. We call $M$ the dual lattice of $N$. Besides, by the definition of the toric Calabi-Yau 3-fold, the canonical divisor of $X_{\Sigma}$ is trivial.

It is known that there are two ways to define a general 3-dim toric variety. The first is to use homogeneous coordinates\cite{cox2011toric}. For any toric 3-fold $X$, we could write $$X=\frac{\mathbb{C}^{k+3}-S}{(\mathbb{C^*})^k}.$$ Here $S$ is a fixed subset under the action of open dense complex torus $(\mathbb{C}^*)^3\subset X$, and k $\mathbb{C^*}$-actions are given by $$\mathbb{C^*}_i: \lambda* (z_1,...,z_{k+3})=(\lambda^{T_{i,1}}z_1,...,\lambda^{T_{i,k+3}}z_{k+3}).$$ The coefficient matrix $$T=\left(
\begin{array}{ccc}
T_{1,1} & ... & T_{1,k+3}\\
... & ...& ...\\
T_{k,1}& ...& T_{k,k+3}\\
\end{array}
\right)$$ is called the toric charge for $X$. Under such definition, the equivalent description of Calabi-Yau condition is that the sum of elements in any fixed line equals zero, i.e. $$X\quad \text{Calabi-Yau}\Longleftrightarrow \sum\limits_{a=1}^{k+3}T_{i,a}=0 \quad \forall i=1,2,...,k.$$

The second way to define a toric 3-fold is just the original definition given by a 3-dim simplicial fan $\Sigma\subset N\cong \mathbb{Z}^3$ \cite{fulton1993introduction}. For a toric 3-fold $X_{\Sigma}$ which corresponds to the fan $\Sigma$, we define the d-dimensional cones in $\Sigma$ by $\Sigma(d)$. For the convention, we always write $\Sigma(1)=\{\rho_1,...,\rho_{k+3}\}$, where $k+3$ is the number of 1-dimensional cones in $\Sigma$. Besides, we denote the generator of $\rho_i$ by $b_i$.  

By definition, $X_{\Sigma}$ contains $\mathbb{T}=N\bigotimes\mathbb{C^*}\cong (\mathbb{C^*})^3$ as an open dense subset. The natural action of $\mathbb{T}$ on itself could be extended to $X_{\Sigma}$. The lattice $N$ could also be canonically identified with $Hom(\mathbb{C}^*,\mathbb{T})$. $N$ is so called cocharacter lattice of $\mathbb{T}$. $M=Hom(\mathbb{T},\mathbb{C}^*)$ is called the character lattice. 

If we define $\Tilde{N}=\bigoplus\limits_{i=1}^{k+3}\mathbb{Z}\Tilde{b_i}$, then we have a natural group homomorphism $\psi: \Tilde{N}\to N$, which sends $\Tilde{b_i}$ to $b_i$. We denote the kernel of $\psi$ by $L$, then we get the exact sequence $$0\to L\stackrel{\psi}\to \Tilde{N}\stackrel{\phi}\to N\to 0.$$ Tensoring $\mathbb{C}^*$ in the exact sequence, we have the exact sequence $$1\to G_{\Sigma}\to \Tilde{\mathbb{T}}\to \mathbb{T}\to 1,$$ where $\Tilde{\mathbb{T}}\cong (\mathbb{C}^*)^{k+3}$ has an action on itself. We extend the action to $\mathbb{C}^{k+3}=$Spec $\mathbb{C}[Z_1,$ $...,Z_{3+k}]$. For any $\sigma\in \Sigma$, we consider $Z_{\sigma}=\prod\limits_{\rho_i\notin\sigma}Z_i$, define $Z(\Sigma)$ to be the closed subvariety defined by the ideal generated by $\{Z_{\sigma}|\sigma\in\Sigma\}$. Then $X_{\Sigma}$ could be identified with the geometric quotient $\mathbb{C}^{3+k}-Z(\Sigma)/G_{\Sigma}$. This is the relation between the two definitions.

Calabi-Yau condition also has a great description under the second definition. The description is that $X$ Calabi-Yau iff all $b_i$ lie in a hyperplane of $N$, or equivalently, there exists a vector $e_3^*\in M$ which makes $\langle e_3^*,b_i\rangle=1$ for any $i=1,...,k+3$. Now we choose $e_1^*,e_2^*$ which make $\{e_1^*,e_2^*,e_3^*\}$ a basis of $M$. Then under the dual basis $\{e_1,e_2,e_3\}$, $b_i$ has the coordinate $(m_i,n_i,1)$ in the lattice $N$. Since then, we always identify the generator of any 1-dimensional cone $b_i$ with a lattice point $(m_i,n_i)$ in $N'$.

For the convention, we define the Calabi-Yau subtorus $\mathbb{T'}$ of $\mathbb{T}$ that $$\mathbb{T'}=ker(e_3^*:\mathbb{T}\to \mathbb{C})\cong (\mathbb{C^*})^2.$$ Then $N'=ker(e_3^*:N\to Z)$ could be identified with $Hom(\mathbb{C^*},\mathbb{T'})$. We define $P_{\Sigma}\subset N'$ the convex hull of $(m_i,n_i)$. Then $P_{\Sigma}$ is a 2-dim convex polytope. The simplicial fan $\Sigma$ also gives a triangulation of $P_{\Sigma}$. For any $\sigma\in \Sigma(3)$, $\sigma\cap N'$ is a triangle. Then a subdivision $T_{\Sigma}$ is given by all these triangles and their faces. Since then, we always use a triangulation $T_{\Sigma}$ to represent a simplicial toric Calabi-Yau 3-fold.      

For a general triangulation $T_{\Sigma}$, $X_{\Sigma}$ is a toric orbifold. In \cite{jiang2008orbifold}, the author introduced the extended stacky fan to deal with the such orbifold case. Mirror curves are also generally defined for simplicial toric Calabi-Yau 3-folds. However, in this paper, we only consider the smooth toric manifold case. Because $X_{\Sigma}$ is smooth iff any cone of $\Sigma$ forms a $\mathbb{Z}$-basis of $N$, since then we assume any triangle in $T_{\Sigma}$ has area $\frac{1}{2}$. In other words, triangulation $T_{\Sigma}$ is the finest. Figure \ref{f5} shows a finest triangulation. 
\begin{figure}
    \centering
    \includegraphics[scale=0.3]{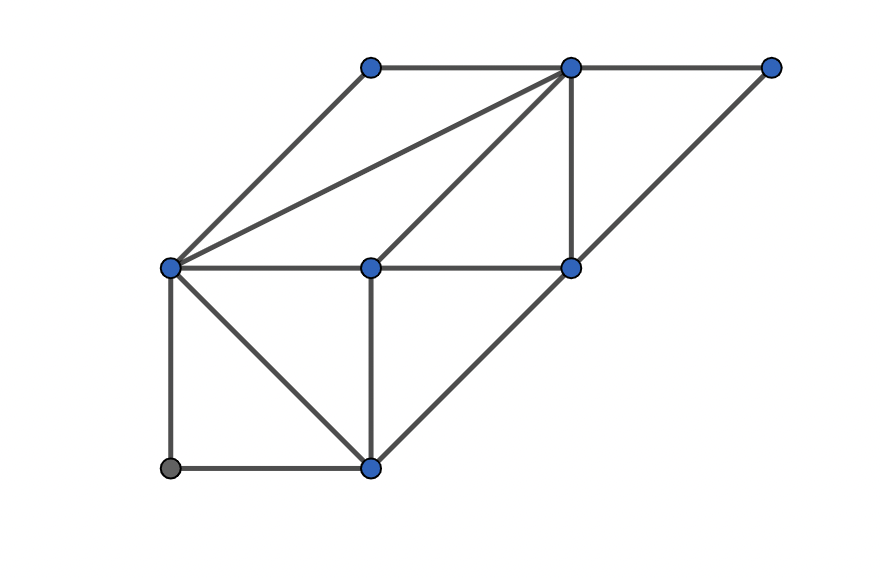}
    \caption{Finest Division}
    \label{f5}
\end{figure}

\subsection{Toric surface associated to a 2-dimension convex polytope}
The mirror curve which we will give a detailed definition later is a complex curve in $(\mathbb{C^*})^2$ as well as a punctured Riemann surface, which is given by a complex Laurent polynomial in two variables with the parameter $q$. Sometimes, we need to compactify it on the toric surface defined by Newton polytope of the polynomial to get a smooth Riemann surface and consider genus or real connected components on the toric surface instead of $(\mathbb{C^*})^2$ or $(\mathbb{R^*})^2$. So in this chapter, we review the definition and some basic facts about the toric surface.

Firstly, we follow the definition in Fulton's book\cite{fulton1993introduction}. Let $\Delta$ be a convex polytope in $M_{\mathbb{R}}$ with integer vertices $A_1,A_2,...,A_k$ in counterclockwise order and $\{\Delta_i=A_iA_{i+1},i=1,2,...,k\}$ be k sides of $\Delta$, where $A_1=A_{k+1}$. For any face of $\Delta$, specifically, all segments, vertices, and the polytope itself (the first two are called proper faces), we consider their dual cones. If $\Delta_P$ is a face of $\Delta$, the dual cone $\sigma_P=\{v\in N_\mathbb{R}|\langle v,u_1\rangle\ge\langle v,u_2\rangle$, for any $u_1\in \Delta_P,u_2\in \Delta\}$. It is easy to see that all dual cones form a fan in $N_\mathbb{R}$. Then the fan defines a $2$-dimensional toric variety $X_{\Delta}$, we called this variety the toric surface associated with convex polytope $\Delta$. $X_{\Delta}$ contains $(\mathbb{C^*})^2$ as an open dense subset.    

Next we move on to the real toric surface $\mathbb{R}{X_{\Delta}}$, which is the closure of $(\mathbb{R^*})^2$ in $X_{\Delta}$. For a non-singular algebraic curve on a toric surface $X_{\Delta}$ whose Newton polytope coincides with the polytope $\Delta$, we could show the genus g of such a curve equals the number of interior lattice points of the Newton polytope by calculating the dimension of global differential forms on the curve\cite{fang2020remodeling}. By Harnack's inequality \cite{harnack1876ueber}, there are at most g+1 connected components of the curve on the real toric surface.

A more concrete and useful description of the toric surface is given in \cite{fang2020remodeling}. Let $\Delta_1$, ..., $\Delta_n$ be the sides of the polytope $\Delta$. We take the dual 1-dimensional cone $\rho_i$ of each $\Delta_i$. For $\rho_i=\mathbb{Z}_{\ge0}b_i$, considering $\Tilde{N}=\bigoplus\limits_{i=1}^n \mathbb{Z}\Tilde{b_i}$ and the natural map $\psi:\Tilde{N}\to N$, which sends $\Tilde{b_i}$ to $b_i$, we have the exact sequence:
\begin{align*}
    0\to L\to \Tilde{N}\stackrel{\psi}\to N\to 0.
\end{align*}

Tensoring $\mathbb{C}^*$ on the exact sequence, we have
\begin{align*}
    1\to G_{\Delta}\to (\mathbb{C}^*)^n\to (\mathbb{C}^*)^2\to 1.
\end{align*}

$G_{\Delta}$ has a natural action on $(\mathbb{C}^*)^n$. We could extend the action to $\mathbb{C}^n$. Then we define $X_{\Delta}=\mathbb{C}^n/G_{\Delta}$ by the toric surface,  with the open dense subset $(\mathbb{C}^*)^n/G_{\Delta}\cong (\mathbb{C}^*)^2$. Moreover, we define $l_i=\{(z_1,...,z_n)\in \mathbb{R}X_{\Delta}|z_i=0\}$ under the homogeneous coordinate $(z_1,...,z_n)$ to be the axis of the real toric surface $\mathbb{R}X_{\Delta}$ related to $\Delta_i$. For any $i=1,...,n$, the axis $l_i$ is a boundary divisor of $\mathbb{R}X_{\Delta}$.
\subsection{Nef cone}
To give the exact definition of the mirror curve, we must firstly introduce the Nef cone with respect to a simplicial fan. One could find the general definition of Nef cone in \cite{cox2011toric}. But in this paper, still for the convention, we only consider the case of smooth toric Calabi-Yau 3-fold. Recall the exact sequence $$0\to L\stackrel{\psi}\to \Tilde{N}\stackrel{\phi}\to N\to 0,$$ and take $Hom(-,Z)$ on the exact sequence of free $Z$-module. We get $$0\to M\stackrel{\phi^v}\to \Tilde{M}\stackrel{\psi^v}\to L^{v}\to 0.$$

We assume $\#\Sigma(1)=p+3$ and then consider the $\mathbb{T}$-equivalent Poincare dual $D_i^{\mathbb{T}}$ of the divisor $\{Z_i=0\}$ in $\mathbb{C}^{p+3}$. Here $D_i^{\mathbb{T}}$ is also the dual basis of $\Tilde{b_i}$ for $i=1,2,..,p+3$. Next we define $D_i=\psi^v(D_i^{\mathbb{T}})$ in $L^v$. For any $\sigma\in \Sigma(3)$, let $I_{\sigma}$ be $\{i|b_i\notin\sigma\}$ and $I_{\sigma}'$ be $\{i|b_i\in\sigma\}$. Then $\{D_i|i\in I_{\sigma}\}$ form a basis of $L^{v}$. 
\begin{definition}
For a $\sigma\in \Sigma(3)$, we define the Nef cone of $\sigma$ to be $$\Tilde{Nef}_{\sigma}=\{\mathbb{R}_{\ge 0}D_i|i\in I_{\sigma}\}.$$ Besides, we define the Nef cone of $\Sigma$ to be $$\Tilde{Nef}(\Sigma)=\bigcap\limits_{\sigma\in \Sigma(3)}\Tilde{Nef}_{\sigma}.$$
\end{definition} 
\subsection{Mirror curve}
\label{2.5}
In this chapter, we define the mirror curve of a smooth toric Calabi-Yau 3-manifold. For $\Sigma$, we give any lattice point of $P_{\Sigma}$ a monomial with parameter $q=(q_1,...,q_p)$. For the convention, we always assume $b_1=(1,0)$, $b_2=(0,1)$, $b_3=(0,0)$. We choose a basis $\{H_i|i=1,2,...,p\}$ of $L_{\mathbb{Q}}^v$ in the Nef cone of $\Sigma$. Then for any $\sigma\in \Sigma(3)$, because $\{D_i|i\in I_{\sigma}\}$ form a basis of $L_{\mathbb{Q}}^v$, we could write $H_i=\sum\limits_{j\in I_{\sigma}}s_{i,j}^{\sigma}D_j$ for $i=1,2,..,p$. Here $s_{i,j}^{\sigma}$ is a non-negative rational number for any $i,j,\sigma$. We now take $H_1,...,H_p$ which make all $s_{i,j}^{\sigma}$ non-negative integers and $S_{\sigma}=[s_{i,j}^{\sigma}]_{p*p}$ non-degenerate. Then for the parameter $q$, we define 
    
$$ a_i^{\sigma}(q)=\left\{
\begin{array}{lcl}
1       &      & {i\in I_{\sigma}'}\\
\prod\limits_{j=1}^{p}(q_j)^{s_{i,j}^{\sigma}}     &      & {i\in I_{\sigma}}
\end{array} \right. $$

A flag $(\tau,\sigma)$ of $\Sigma$ consists of two cones $\sigma\in\Sigma(3)$ and $\tau\in\Sigma(2)$ with $\tau\subset\sigma$. Then for any flag $(\tau,\sigma)$, there exist the unique $i_1,i_2,i_3$ which make $I_{\sigma}'=\{i_1,i_2,i_3\}$, $I_{\tau}'=\{i_2,i_3\}$, $I_{\sigma}=\{1,2,...,p+3\}-I_{\sigma}'$, and $I_{\tau}=\{1,2,...,p+3\}-I_{\tau}'$, where $b_{i_1},b_{i_2},b_{i_3}$ satisfy the counterclockwise order in $\sigma$. For the convention, we call $b_{i_3}$ the origin of flag $(\tau,\sigma)$.

Given a flag $(\tau,\sigma)$, if $\boldsymbol{b_{i_1}}-\boldsymbol{b_{i_3}}=(m_{i_1},n_{i_1})$, $\boldsymbol{b_{i_2}}-\boldsymbol{b_{i_3}}=(m_{i_2},n_{i_2})$, we define \begin{align*}
    e_{1,(\tau,\sigma)}&=m_{i_1}e_1+n_{i_1}e_2\\
    e_{2,(\tau,\sigma)}&=m_{i_2}e_1+n_{i_2}e_2.
\end{align*}

Then $m_{i_1}n_{i_2}-m_{i_2}n_{i_1}=1$. Any lattice point $b_{j}$ of $P_{\Sigma}$ could be written $m_j'e_{1,(\tau,\sigma)}+n_j'e_{2,(\tau,\sigma)}+b_{i_3}$. We call $(m_j',n_j')$ the coordinate of $b_j$ under the flag $(\tau,\sigma)$.  

\begin{definition}The mirror curve with respect to a given flag $(\tau,\sigma)$ is the zeroes in $(\mathbb{C^*})^2$ of following Laurent polynomial:  $$H_{(\tau,\sigma)}(X_{(\tau,\sigma)},Y_{(\tau,\sigma)},q)=\sum\limits_{i=1}^{p+3}a_i^{\sigma}(q)X_{(\tau,\sigma)}^{m_i'}Y_{(\tau,\sigma)}^{n_i'}.$$
\end{definition}
Specially, when $(\tau_1,\sigma_1)$ satisfies the condition $I_{\sigma_1}'=\{1,2,3\},I_{\tau_1}'=\{2,3\}$, for the convention, we write $H_{(\tau_1,\sigma_1)}(X_{(\tau_1,\sigma_1)},Y_{(\tau_1,\sigma_1)},q)$ as $H(X,Y,q)$. From now on, if we talk about a mirror curve without a given flag, then we mean the mirror curve of flag $(\tau_1,\sigma_1)$. 

There are many good propositions of the mirror curve. The most elementary and important proposition of the mirror curve is the affine equivalence. 

\begin{proposition}
\label{p2.5.2}
The mirror curves of two different flags $(\tau_1,\sigma_1),(\tau_2,\sigma_2)$ of $\Sigma$ are affine equivalent. Besides, if $I_{\sigma_1}'=\{1,2,3\}$, $I_{\tau_1}'=\{2,3\}$, $I_{\sigma_2}'=\{i_1,i_2,i_3\}$, and $I_{\tau_2}'=\{i_2,i_3\}$, we could write down the coordinate change concretely.
\begin{align*}
    X_{(\tau_2,\sigma_2)}&=X^aY^b\frac{a_{i_1}(q)}{a_{i_3}(q)}\\
    Y_{(\tau_2,\sigma_2)}&=X^cY^d\frac{a_{i_2}(q)}{a_{i_3}(q)}.
\end{align*} Under such coordinate transformation, we have the following equivalence $$H(X,Y,q)=a_{i_3}(q)X^{m_{i_3}}Y^{n_{i_3}}H_{(\tau_2,\sigma_2)}(X_{(\tau_2,\sigma_2)},Y_{(\tau_2,\sigma_2)},q).$$

Here $a,b,c,d$ is given by $$\boldsymbol{b_{i_1}}-\boldsymbol{b_{i_3}}=(a,b)$$ $$\boldsymbol{b_{i_2}}-\boldsymbol{b_{i_3}}=(c,d).$$
\end{proposition}
By the affine equivalence, when we want to study propositions of one fixed flag or we need a fixed flag, we may do the coordinate change and send the three vertices of the flag to $(0,0),(0,1),(1,0)$.

Two deeper and more useful propositions about mirror curves show the distribution rules of $a_i(q)$ at each lattice point.

\begin{proposition}
\label{p2.5.3}
In $\Sigma$, if $b_{i_1}$ directly connects $b_{i_2}$ by a 1-cell in $T_{\Sigma}$ and $2b_{i_2}=b_{i_1}+b_{i_3}$, then $\frac{a_{i_1}(q)a_{i_3}(q)}{a_{i_2}(q)^2}$ is a non-constant monomial of $q$.
\end{proposition} 
Figure \ref{f6} shows this case.

\begin{figure}
    \centering
    \includegraphics[scale=0.4]{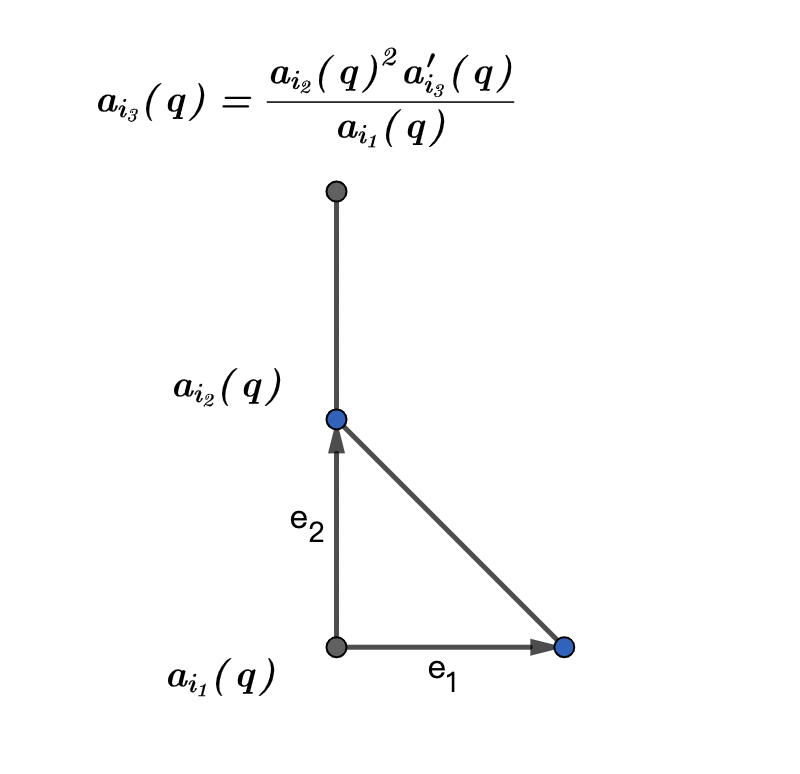}
    \caption{}
    \label{f6}
\end{figure}

\begin{proof}
When $i_2=3,i_1=2$, the proof is done because with $\boldsymbol{b_{i_3}}=(-1,0)$, $a_{i_3}(q)$ is a non-constant monomial of $q$. For the general case, we hope to simplify the case to $i_2=3,i_1=2$.

Choose a flag $(\tau,\sigma)$ which makes $I_{\tau}=\{i_1,i_2\}$, $I_{\sigma}=\{i_1,i_2,j\}$. Then the coordinate of $b_{i_3}$ under the flag $(\tau,\sigma)$ is $(0,-1)$, and $a_{i_3}'(q)$ is a non-constant monomial of $q$. Under the affine coordinate change $$X_{(\tau,\sigma)}=X^aY^b\frac{a_{i_1}(q)}{a_{i_2}(q)}$$
$$Y_{(\tau,\sigma)}=X^cY^d\frac{a_{j}(q)}{a_{i_2}(q)},$$ we have $$a_{i_2}(q)X^{m_{i_2}}Y^{n_{i_2}}a_{i_3}'(q)X_{(\tau,\sigma)}^{-1}=a_{i_3}(q)X^{m_{i_3}}Y^{n_{i_3}}.$$ Therefore $\frac{a_{i_1}(q)a_{i_3}(q)}{a_{i_2}(q)^2}=a_{i_3}'(q)$
\end{proof}
\begin{proposition}
\label{p2.5.4}
For a flag $(\tau,\sigma)$ with $I_{\sigma}'=\{i_1,i_2,i_3\}$ and $I_{\tau}'=\{i_2,i_3\}$, and another vertex $b_{i_4}$ which makes $i_4\in I_{\sigma}$,  $\frac{a_{i_4}(q)(a_{i_3}(q))^{m_i'+n_i'-1}}{a_{i_1}(q)^{m_i'}a_{i_2}(q)^{n_i'}}$ is a non-constant monomial of $q$.
\end{proposition} 
Figure \ref{f7} shows this case.
\begin{figure}
    \centering
    \includegraphics[scale=0.4]{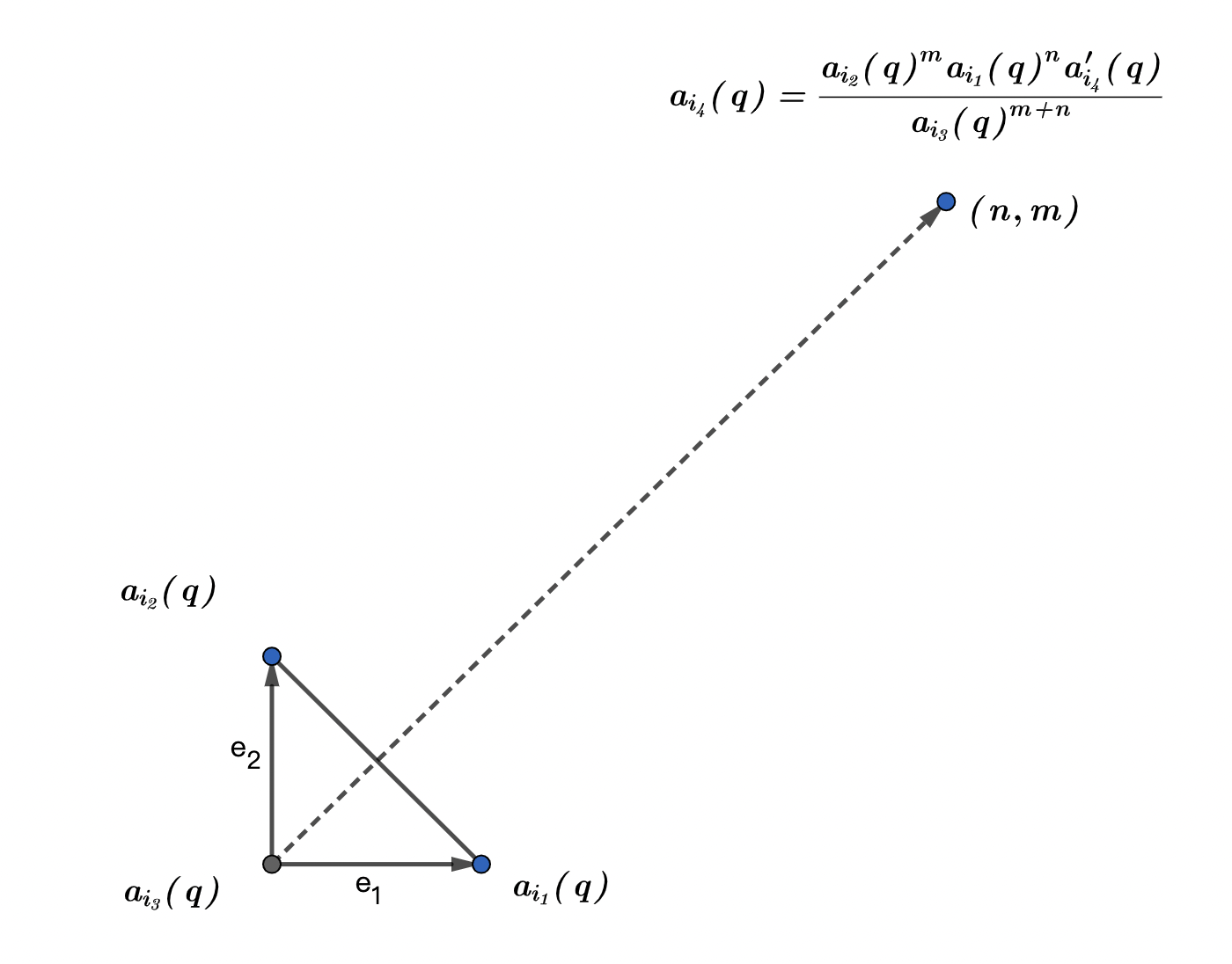}
    \caption{}
    \label{f7}
\end{figure}
\begin{proof}
As the proof of Proposition \ref{p2.5.3}, firstly, the case with $I_{\sigma}=\{1,2,3\}$, $I_{\tau}=\{2,3\}$ is easy. Then we consider the affine coordinate change under flag $(\tau,\sigma)$. Because $b_{i_4}$ has coordinate $(m_{i_4}',n_{i_4}')$, we have $$a_{i_4}'(q)=\frac{a_{i_4}(q)(a_{i_3}(q))^{m_i'+n_i'}}{a_{i_1}(q)^{m_i'}a_{i_2}(q)^{n_i'}}.$$
\end{proof}

\begin{aremark}
Proposition \ref{p2.5.3} could be seen as a special case of Proposition \ref{p2.5.4}. However, Proposition \ref{p2.5.3} does not require a fixed flag, so sometimes it is more convenient to apply Proposition \ref{p2.5.3} instead of Proposition \ref{p2.5.4}.
\end{aremark}

\begin{aremark}
One could also give the definition of the mirror curve with parameter $q=(q_1,...,q_p)\in (\mathbb{C^*})^p$ if we define the smooth toric Calabi-Yau 3-fold $X_{\Sigma}$ with the homogeneous coordinate. Another formal definition of the mirror curve of
$X_{\Sigma}$ is given by $$C_q=\{1+x_2+...+x_{3+p}=0|q_i=\prod\limits_{m=2}^{3+k}x_m^{T_{i_m}},\quad i=1,...,p\}.$$ We call $\frac{1}{|q|}$ the radius of parameter $q$. Then people always study propositions of mirror curves with a large radius parameter. For the convention, if we say a statement for the mirror curve holds near the large radius limit, we mean $\exists \epsilon>0$, for the mirror curve with any $q$ has a radius larger than $\frac{1}{\epsilon}$, or equivalently, $|q|<\epsilon$, the statement holds.
\end{aremark}

\begin{aremark}The importance of such propositions is due to when $(m_{i_4},n_{i_4})\ne(0,0)$, $(0,1)$, $(1,0)$, especially with an additional condition that $m_{i_4}$ and $n_{i_4}$ are both non-negative, near the large radius limit, we could easily compare the norm of $a_{i_4}(q)$ to the norms of $a_{i_1}(q)$, $a_{i_2}(q)$, $a_{i_3}(q)$.
\end{aremark}
\subsection{Tropical geometry and amoeba}
\label{2.6}
In this chapter, we recall some basic notations and tools of tropical geometry. First, let us recall the notation of amoeba. Consider the map given below.
$$ 
\begin{array}{lcl}
     \mu:(\mathbb{C}^*)^n\to \mathbb{R}^n\\
\mu(z_1,...,z_n)=(log|z_1|,...,log|z_n|)
\end{array}  $$

We define such a map by the amoeba map.

\begin{definition}Let $K_F$ be the zeroes in $(\mathbb{C^*})^n$ of a holomorphic function $F(z_1,...,z_n)$ with n variables. We define the amoeba of $F$ to be the image $\mu(K_F)$. We write the amoeba of $F$ by $A_F$.
\end{definition}
Amoeba is the core object of tropical geometry. There are many great propositions about the amoeba. Next, we list one and give a brief idea of the proof. If anyone is interested in the details, he could refer to \cite{yger2012tropical}.

\begin{proposition}If $F$ is a Laurent polynomial with the Newton polytope $\Delta_F$. We could define an injective map $\psi$ from the connected components of $\mathbb{R}^n-A_F$ to $\Delta_F\bigcap \mathbb{Z}^n$. 
\end{proposition}
\begin{proof}: We skip some details and write down the proof briefly. For any point $x_0=(x_1,...,x_n)$ in $\mathbb{R}^n-A_F$, we consider $\mu^{-1}(x_0)\subset (\mathbb{C}^*)^n$ which does not intersect $K_F$. Then the degree of loop $\gamma_i:\theta\to F(e^{x_1},...,e^{x_i+i\theta},..,e^{x_n})$ gives an integer $v_{F,x_0,i}$. Thus we have constructed a map $\psi$ from $\mathbb{R}^n-A_F$ to a point of $\mathbb{Z}^n$ given by $$\psi(x_0)=(v_{F,x_0,1},...,v_{F,x_0,n}).$$ We write $v_{F,x_0}$ for short of $(v_{F,x_0,1},...,v_{F,x_0,n})$. Also, we could write the degree as Cauchy integral \begin{align*}
    v_{F,x_0,i}=\frac{1}{(2i\pi)^n}\int_{\mu^{-1}(x_0)}\frac{\partial F}{\partial z_i}(z)\frac{1}{F(z)}\bigwedge\limits_{j=1}^n\frac{dz_i}{z_i}.
\end{align*}

It is easy to see that the integer $v_{F,x_0,i}$ we just gave is constant on any connected component of $\mathbb{R}^n-A_F$ because the Cauchy integral changes continuously. Besides, by the integral, we take any preimage point $z=(z_1,...,z_n)\in (\mathbb{C}^*)^n$ of $x_0$ instead of  $(e^{x_1},...,e^{x_i+i\theta},..,e^{x_n})$. Denote such the degree of loop $F(z_1,...,z_ie^{i\theta},..,z_n)$ by $v_{F,x_0,z,i}$. Calculating $v_{F,x_0,z,i}$ by the Cauchy integral, we get the same result with $v_{F,x_0,i}$ because the integral change continuously in torus $\mu^{-1}(x_0)$. Then we define $\psi(C_{x_0})=v_{F,x_0}=v_{F,C_{x_0}}$, where $C_{x_0}$ is the connected component which contains $x_0$.

For any lattice point $a=(a_1,..,a_n)$ $\in \mathbb{Z}^n$, we know the degree of the loop \begin{align*}
    \gamma_F:\theta\in[0,2\pi]\to F(z_1e^{ia_1\theta},...,z_ne^{ia_n\theta})
    \end{align*} is $\langle a,v_{F,C_{x_0}}\rangle$. However, the degree of the loop must be no greater than the degree of $\gamma_F$ under $e^{i\theta}$. Thus $\langle a,v_{F,C_{x_0}}\rangle\le max\langle a,s\rangle_{s\in \Delta_F}$ for any $a,x_0$. Then we have $\psi(C_{x_0})=v_{F,C_{x_0}}\in \Delta_F$.

Next, we will prove the injectivity of $\psi$. If there exist two components $C_1\ne C_2$, but $\psi(C_1)=\psi(C_2)$, because there exists a line $l_i$ with a rational slope passing $x_0\in C_1$ and $x_1\in C_2$. We parametrize such line with the parameter $t$, and integer vector $a=(a_1,...,a_n)$. Then $l_i$ could be written $(x_{0,1}+a_1t,...,x_{0,n}+a_nt)$. We assume $x_1$ corresponds to a positive $t$ in the line. Because the degree of the loop which corresponding $(a_1,...,a_n)$ at $x_0$ could also be seen as the link number of the loop of $\mu^{-1}(x_0)$ with degree $(a_1,...,a_n)$ at each coordinate and $K_F$, and $l_i$ has passed the amoeba, the degree increases from $x_0$ to $x_1$. We have shown $\langle a,v_{F,C_1}\rangle<\langle a,v_{F,C_2}\rangle$. This is a contradiction. 
\end{proof}
Since then, for short, if $\psi(C)=\sigma$, we say that the component $C$ has the degree $\sigma$. At the end of this chapter, I will introduce the notation of the tropical Laurent polynomial. 

\begin{definition}The tropical Laurent polynomial is the Laurent polynomial under two operations: $\boxtimes$ and $\boxplus$.
For a finite index set $J$ and $w_i\in \mathbb{R}\cup \{-\infty\}$, $\boxplus$ and $\boxtimes$ are defined as follows: \begin{align*}\boxplus_{i\in J} w_i=max\{w_i\}_{i\in J},\end{align*}  \begin{align*}\boxtimes_{i\in J}w_i=\sum\limits_{i\in J}w_i.\end{align*}\end{definition} 

Specifically, the tropical Laurent polynomial with $n$ variables could be written as \begin{align*}\boxplus_{i\in J}a_i\boxtimes w_1^{c_{i_1}}\boxtimes...\boxtimes w_n^{c_{i_n}},\end{align*} where $c_{i_r}$ is an integer, and $w_r^{c_{i_r}}$ is short for $w_i\boxplus...\boxplus w_i$ which has total $c_{i_r}$ $w_i$.

Because the tropical Laurent polynomial could be written as $max\{ f_i\}_{i\in J}$, where $f_j=a_j+c_{j_1}w_1+...+c_{j_n}w_n$, we define the tropical hypersurface of a tropical Laurent polynomial to be the union of all the points in $(\mathbb{R}\cup \{-\infty\})^n$ which make at least two $i\in J$, denoted by $i_1,i_2$, satisfy $f_{i_1}=f_{i_2}=max\{ f_i\}_{i\in J}$. It is obvious that the tropical hypersurface consists of some segments, rays, and lines. Actually, the tropical hypersurface is exactly the set of critical points of $max\{ f_i\}_{i\in J}$. Since then, we write the tropical polynomial short for the tropical Laurent polynomial.

\section{Tropical deformation}
\label{3}
In this section, we mainly prove that near the large radius limit, the amoeba of the mirror curve $H=0$ for smooth toric Calabi-Yau 3-fold $X_{\Sigma}$ has a canonical deformation retract which gives the dual subdivision of $T_{\Sigma}$. We finish the proof by figuring out all connected components of $\mathbb{R}^2-A_H$ and directly calculating the tropical spine which we will later give a definition of. 

\begin{lemma1}
\label{l3.1}
For any non-zero Laurent polynomial $F:(\mathbb{C^*})^n\to\mathbb{C}$ with its amoeba $A_F$, $A_F$ has a tropical hypersurface of $\boxplus\tau_{F,C}\boxtimes w_1^{\sigma_{F,C,1}}\boxtimes w_2^{\sigma_{F,C,2}}...\boxtimes w_n^{\sigma_{F,C,n}}$ as a deformation retract. Here the boxplus is done over all connected components of $\mathbb{R}^n-A_F$ and $\sigma_{F,C}$ is the degree of $C$, $\tau_{F,C}$ is some constant related to $F$ and $C$. We define this tropical hypersurface by the tropical spine.
\end{lemma1}

\begin{proof} Firstly, we define the Ronkin function $R_F:\mathbb{R}^n\to \mathbb{R}$ as in \cite{yger2012tropical}: \begin{align*}R_F(x)&=\int_{(S^1)^n}log|F(e^{x+i\theta}))|d\sigma_{(S^1)^n}(\theta)\\&=\frac{1}{(2\pi)^n}\int_{[0,2\pi]^n}log|F(e^{x+i\theta})|d\theta_1...d\theta_n.\end{align*} The symbol $d\sigma_{(S^1)^n}(\theta)$ here is the normalized Haar measure of group $(S^1)^n$. Then $R_F$ is $C^{\infty}$ in any  connected component of $\mathbb{R}^n-A_F$ because the integral has no flaw points outside $A_F$. The integral is well-defined on $\mathbb{R}^n$ because $log|F|$ is plurisubharmonic\cite{yger2012tropical} on $\mathbb{R}^n$, i.e. subharmonic on every complex line. Then $R_F$ is a convex function. Thus the improper integral at any point gives a finite real value on $\mathbb{R}^n$.

We calculate the gradient of $R_F(x)$ \begin{align*}\frac{dR_F(x)}{dx_i}&=Re(\frac{1}{(2\pi)^n}\int_{[0,2\pi]^n}\frac{\frac{\partial F}{\partial x_j}}{F}(e^{x+i\theta})d\theta_1...d\theta_n)\\&=Re(\frac{1}{(2i\pi)^n}\int_{\mu^{-1}(x)}\frac{\partial F}{\partial z_i}(z)\frac{1}{F(z)}\bigwedge\limits_{j=1}^n\frac{dz_i}{z_i})\\&=Re(\sigma_{F,C,j})\\&=\sigma_{F,C,j}.\end{align*} Then $R_F(x)=\tau_{F,C}+\langle\sigma_{F,C},x\rangle$ on $C$, where $\tau_{F,C}$ is some constant related to $F$ and $C$.

Because $R_F$ is convex, if we extend $\tau_{F,C}+\langle\sigma_{F,C},x\rangle$ to another component $C'$ of $\mathbb{R}^n-A_F$ linearly, we have $$\tau_{F,C}+\langle\sigma_{F,C},x\rangle<\tau_{F,C'}+\langle\sigma_{F,C'},x\rangle \quad \forall x\in C'.$$ We have shown that the tropical spine lies in the amoeba $A_F$. By intuition, the amoeba retracts into the spine. The concrete retracting process is given in \cite{yger2012tropical}, one who has interest could find it.
\end{proof}
Generally speaking, even we have known that the amoeba of a Laurent polynomial $F$ has its tropical spine  $\boxplus_C\tau_{F,C}\boxtimes w_1^{\sigma_{F,C,1}}\boxtimes...\boxtimes w_n^{\sigma_{F,C,n}}$ as the deformation retract, it is still hard to calculate the tropical spine because it is hard to know exactly which lattice point in $\Delta_F$ corresponds to a connected component of $\mathbb{R}^n-A_F$. Besides, the calculation of the coefficient $\tau_{F,C}$ is difficult as well. However, the case is easy when one monomial in $F$ has the largest norm and is larger than the sum of norms of other monomials.

\begin{lemma1}
\label{l3.2}
In a Laurent polynomial \begin{align*}F(z)=\sum\limits_{\sigma\in\Delta_F}c_{\sigma}z^{\sigma},\end{align*} if there exist one monomial $c_{\sigma_0}z^{\sigma_0}$ and a point $z_0$ in $(\mathbb{C}^*)^n$ making \begin{align*}|c_{\sigma_0}z_0^{\sigma_0}|>\sum\limits_{\sigma\in\Delta_F,\sigma\ne \sigma_0}|c_{\sigma}z_0^{\sigma}|,\end{align*} obviously $x_0=log|z_0|\in \mathbb{R}^n$ lies in $\mathbb{R}^n-A_F$. Let $C_{x_0}$ be the connected component of $\mathbb{R}^n-A_F$ which contains $x_0$, Then the degree of $C_{x_0}$ equals $\sigma_0$.
\end{lemma1}

\begin{proof}: Define $$g(z)=\frac{\sum\limits_{\sigma\in\Delta_F,\sigma\ne \sigma_0}(c_{\sigma}z^{\sigma})}{c_{\sigma_0}z^{\sigma_0}}=\sum\limits_{\sigma\in\Delta_F,\sigma\ne \sigma_0}\frac{c_{\sigma}}{c_{\sigma_0}}z^{\sigma-\sigma_0}.$$ We do the direct calculation. 
\begin{align*}
    R_F(x_0)&=\frac{1}{(2\pi)^n}\int_{[0,2\pi]^n}log|F(z_0*e^{i\theta})|d\theta_1...d\theta_n
    \\&=\frac{1}{(2\pi)^n}\int_{[0,2\pi]^n}log|c_{\sigma_0}z_0^{\sigma_0}(1+\frac{\sum\limits_{\sigma\in\Delta_F,\sigma\ne \sigma_0}(c_{\sigma}z_0^{\sigma}e^{i<\sigma,\theta>})}{c_{\sigma_0}z_0^{\sigma_0}e^{i<\sigma_0,\theta>}})|d\theta_1...d\theta_n
    \\&=log|c_{\sigma_0}|+\langle x_0,\sigma_0\rangle+\int_{Log^{-1}(x_0)}log|1+\frac{\sum\limits_{\sigma\in\Delta_F,\sigma\ne \sigma_0}(c_{\sigma}z^{\sigma})}{c_{\sigma_0}z^{\sigma_0}}|d\sigma_{(S^1)^n}
    \\&=log|c_{\sigma_0}|+\langle x_0,\sigma_0\rangle+\frac{1}{(2\pi i)^n}\int_{Log^{-1}(x_0)}log|1+g(z)|\frac{dz_1...dz_n}{z_1...z_n}.
\end{align*}

Because $|g(z)|<1,\forall z\in Log^{-1}(x_0)$, we have \begin{align*}
    &\frac{1}{(2\pi i)^n}\int_{Log^{-1}(x_0)}log|1+g(z)|\frac{dz_1...dz_n}{z_1...z_n}\\&=Re(\frac{1}{(2\pi i)^n}\int_{Log^{-1}(x_0)}log(1+g(z))\frac{dz_1...dz_n}{z_1...z_n})\\&=Re(\frac{1}{(2\pi i)^n}\int_{Log^{-1}(x_0)}\sum\limits_{k=1}^{+\infty}(-1)^{k-1}\frac{g(z)^k}{k}\frac{dz_1...dz_n}{z_1...z_n})\\&=Re(\frac{1}{(2\pi i)^n}\int_{Log^{-1}(x_0)}\sum\limits_{k=1}^{+\infty}(-1)^{k-1}\frac{(\sum\limits_{\sigma\in\Delta_F,\sigma\ne \sigma_0}\frac{c_{\sigma}}{c_{\sigma_0}}z^{\sigma-\sigma_0})^k}{k}\frac{dz_1...dz_n}{z_1...z_n})\\&=Re(\frac{1}{(2\pi i)^n}\int_{Log^{-1}(x_0)}\sum\limits_{k=1}^{+\infty}(-1)^{k-1}(\sum\limits_{\sigma_j\in\Delta_F-\sigma_0,j=1,2,..,k}\frac{c_{\sigma_1}...c_{\sigma_k}}{kc_{\sigma_0}^k}*\\&z^{\sigma_1+...+\sigma_k-k\sigma_0})\frac{dz_1...dz_n}{z_1...z_n}).\end{align*}
It is known that when $\sigma\ne (-1,...,-1)$,\begin{align*}
    \int_{Log^{-1}(x_0)}z^{\sigma}dz_1...dz_n=0,
\end{align*} and\begin{align*}
    \int_{Log^{-1}(x_0)}\frac{dz_1...dz_n}{z_1...z_n}=(2\pi i)^n.
\end{align*} Thus \begin{align*}
    &\frac{1}{(2\pi i)^n}\int_{Log^{-1}(x_0)}log|1+g(z)|\frac{dz_1...dz_n}{z_1...z_n}\\&=Re(\frac{1}{(2\pi i)^n}\int_{Log^{-1}(x_0)}\sum\limits_{k=1}^{+\infty}(-1)^{k-1}(\sum\limits_{\sigma_j\in\Delta_F-\sigma_0,j=1,2,..,k}\frac{c_{\sigma_1}...c_{\sigma_k}}{kc_{\sigma_0}^k}*\\&z^{\sigma_1+...+\sigma_k-k\sigma_0})\frac{dz_1...dz_n}{z_1...z_n})\\&=Re(\frac{1}{(2\pi i)^n}\int_{Log^{-1}(x_0)}(\sum\limits_{k,\sigma_1+...+\sigma_k=k\sigma_0}\frac{c_{\sigma_1}...c_{\sigma_k}}{kc_{\sigma_0}^k})\frac{dz_1...dz_n}{z_1...z_n})\\&=Re(\sum\limits_{k,\sigma_1+...+\sigma_k=k\sigma_0}\frac{c_{\sigma_1}...c_{\sigma_k}}{kc_{\sigma_0}^k}).
\end{align*}
Because $$\sum\limits_{\sigma\in\Delta_F,\sigma\ne \sigma_0}|\frac{c_{\sigma}}{c_{\sigma_0}}z^{\sigma-\sigma_0}|<1,\forall z\in Log^{-1}(x_0),$$ then \begin{align*}
&\sum\limits_{k,\sigma_1+...+\sigma_k=k\sigma_0}|\frac{c_{\sigma_1}...c_{\sigma_k}}{kc_{\sigma_0}^k}|\\       
&<\sum\limits_{k=1}^{+\infty}\sum\limits_{\sigma_j\in\Delta_F-\sigma_0,j=1,2,..,k}|\frac{c_{\sigma_1}...c_{\sigma_k}}{kc_{\sigma_0}^k}z^{\sigma_1+...+\sigma_k-k\sigma_0}|\\&=\sum\limits_{k=1}^{+\infty}(-1)^{k-1}\frac{(\sum\limits_{\sigma\in\Delta_F,\sigma\ne \sigma_0}|\frac{c_{\sigma}}{c_{\sigma_0}}z^{\sigma-\sigma_0}|)^k}{k}<+\infty.\end{align*} Therefore, the series $\sum\limits_{k,\sigma_1+...+\sigma_k=k\sigma_0}\frac{c_{\sigma_1}...c_{\sigma_k}}{kc_{\sigma_0}^k}$ is absolutely convergent. 
Then we get $$R_F(x_0)=log|c_{\sigma_0}|+\langle x_0,\sigma_0\rangle+Re(\sum\limits_{k,\sigma_1+...+\sigma_k=k\sigma_0}\frac{c_{\sigma_1}...c_{\sigma_k}}{kc_{\sigma_0}^k}).$$
When $x_0$ makes one point $z_0$ in $\mu^{-1}(x_0)$ and $\sigma_0$ satisfy the condition of Lemma \ref{l3.2}, which also guarantees all point $z$ in  $\mu^{-1}(x_0)$ satisfy the same condition, there exist a neighborhood of $x_0$ in $\mathbb{R}^n$, denoted by $U_{x_0}$ here, making all $z\in \mu^{-1}(U_{x_0})$ and $\sigma_0$ satisfying the condition of Lemma \ref{l3.2}. Equivalently, the condition of Lemma \ref{l3.2} is an open condition on $\mathbb{R}^n$ for $x_0$. Then for any $x\in U_{x_0}$, we have $$R_F(x)=log|c_{\sigma_0}|+\langle x,\sigma_0\rangle+Re(\sum\limits_{k,\sigma_1+...+\sigma_k=k\sigma_0}\frac{c_{\sigma_1}...c_{\sigma_k}}{kc_{\sigma_0}^k}).$$ Finally, we know that \begin{align*}\sigma_{F,C_{x_0}}&=\sigma_0 \\ \tau_{F,C_{x_0}}&=log|c_{\sigma_0}|+Re(\sum\limits_{k,\sigma_1+...+\sigma_k=k\sigma_0}\frac{c_{\sigma_1}...c_{\sigma_k}}{kc_{\sigma_0}^k}). \end{align*} 
\end{proof}
Applying Lemma \ref{l3.2} to the case of the mirror curve near the large radius limit, we have the following theorem:

\begin{theorem1}
\label{t3.3}
Near the large radius limit, for the mirror curve $H(X,Y,q)=0$ of a smooth toric Calabi-Yau 3-fold $X_{\Sigma}$, the injective map $\psi$ from the connected components of $\mathbb{R}^2-A_H$ to $\Delta_H\cap Z^2$ is also surjective. 
\end{theorem1}

\begin{proof} When $\sigma=(m_i,n_i)\in \Delta_H\cap Z^2$, by applying Lemma \ref{l3.2}, we only need to find a point $z_{\sigma}=(x_{\sigma},y_{\sigma})$ which makes the norm of $a_i(q)x_{\sigma}^{m_i}y_{\sigma}^{n_i}$ larger than the sum of norms of all other monomials. When $\sigma=(0,0)=(m_3,n_3)$, $z_{\sigma}=(\frac{1}{3},\frac{1}{3})$ satisfies the condition near the large radius limit. That is because when $z_{\sigma}=(\frac{1}{3},\frac{1}{3})$, the sum of all norms of monomials $a_i(q)x_{\sigma}^{m_i}y_{\sigma}^{n_i}$, $i\ne3$ is a polynomial of $|q|$ with a constant term $\frac{2}{3}$. Near the large radius limit, we have $$1>\frac{1}{3}+\frac{1}{3}+\sum\limits_{i=4}^{p+3}\frac{|a_i(q)|}{3^{m_i+n_i}}.$$ Applying Lemma \ref{l3.2}, we know that $(0,0)$ is in the image of $\psi$.

Next, for any lattice point $b_i$ of $\Delta_H\cap Z^2$, we see it as the origin of a flag $(\sigma,\rho)$. Because $H_{(\sigma,\rho)}$ is affine equivalent to $H=H_{(\sigma_0,\rho_0)}$, the corresponding coordinate change preserves the norm ratio, and such coordinate change gives an isomorphism from $(\mathbb{C}^*)^2$ to $(\mathbb{C}^*)^2$, the lattice point $b_i$ has the non-empty preimage of $\psi$.
\end{proof}
Now we have found all the connected components of $\mathbb{R}^2-A_H$. Thus the tropical spine of the mirror curve could be written as \begin{align*}\boxplus_{(m_i,n_i)\in \Delta_H}log|a_i(q)|\boxtimes Re(\sum\limits_{k,b_{j_1}+...+b_{j_k}=kb_{j_i}}\frac{a_{j_1}(q)...a_{j_k}(q)}{ka_{i}(q)^k})\boxtimes w_1^{m_i}\boxtimes w_2^{n_i}.\end{align*}

The tropical hypersurface gives a convex subdivision\cite{passare2000amoebas} of $\mathbb{R}^2$ denoted by $T$. Here the 2-cells of the subdivision are all linear parts of the tropical polynomials, the 1-cells are all the critical segments, rays, and lines. For this subdivision, we have the following theorem.

\begin{theorem1}
\label{thm:subdivision-triangulation}   
The subdivision $T$ given by the tropical hypersurface is dual to the triangulation $T_{\Sigma}$ near the large radius limit.  
\end{theorem1}

\begin{proof}
We could write the tropical polynomial as $$S(w)=max_{i=1,2,...,p+3}\{\gamma_i+m_iw_1+n_iw_2\},$$
 where $$\gamma_i=log|a_i(q)|+Re(\sum\limits_{k,b_{j_1}+...+b_{j_k}=kb_{j_i}}\frac{a_{j_1}(q)...a_{j_k}(q)}{ka_{i}(q)^k}).$$

First, we construct the duality of all 1-cells of $T_{\Sigma}$ in the subdivision $T$. For any two fixed vertices $b_{i_1}$ and $b_{i_2}$, we define \begin{align*}
    \sigma_{i_1,i_2}^*=\{(w_1,w_2)|S(w)=\gamma_{i_1}+m_{i_1}w_1+n_{i_1}w_2=\gamma_{i_2}+m_{i_2}w_1+n_{i_2}w_2\}.
\end{align*} Obviously, $\sigma_{i_1,i_2}^*$ is either a 1-cell including segment and ray, or an empty set. Thus, we only need to show near the large radius limit, $\sigma_{i_1,i_2}^*$ is non-empty iff there's a 1-cell in $T_{\Sigma}$ which connects $b_{i_1}$ and $b_{i_2}$. 

In order to calculate the spine, we will show $$\frac{a_{j_1}(q)...a_{j_k}(q)}{ka_{i}(q)^k}$$ is a non-constant monomial of $q$ for any $\boldsymbol{b_{j_1}}+...+\boldsymbol{b_{j_k}}=k\boldsymbol{b_i},j_r\ne i.$ When $i=3$, because $\boldsymbol{b_{j_1}}+...+\boldsymbol{b_{j_k}}=k\boldsymbol{b_i},$ there must be some $j_r\ne 1,2$. Then the case is easy because for any $i\ne 1,2,3$, $a_i(q)$ is a non-constant monomial of $q$. For a general $i$, we consider a flag $(\sigma,\tau)$ which makes $I_{\sigma}'=\{t_1,t_2,i\}$ and contains $b_i$ as the origin $b_3'$. We assume $b_{j_s}$ corresponds to $b_{r_s}'$ under the flag $(\sigma,\tau)$. Then because $\boldsymbol{b_{j_1}}+...+\boldsymbol{b_{j_k}}=k\boldsymbol{b_i},$ we have $\boldsymbol{b_{r_1}'}+...+\boldsymbol{b_{r_k}'}=0.$ Because  \begin{align*}
    H_{(\sigma,\tau)}(X_{(\sigma,\tau)},Y_{(\sigma,\tau)},q)=1+X_{(\sigma,\tau)}+Y_{(\sigma,\tau)}+\sum\limits_{w=4}^{p+3}a_w'(q)X_{(\sigma,\tau)}^{m_w'}Y_{(\sigma,\tau)}^{n_w'},
\end{align*} and the monomial at any lattice point coincides with that in $H(X,Y,q)$ under the coordinate change, we have $$a_{r_s}'(q)=\frac{a_{j_s}(q)a_{t_1}(q)^{m_{r_s}'}a_{t_2}(q)^{n_{r_s}'}}{a_{i}(q)^{m_i'+n_i'+1}}.$$ Then because $\boldsymbol{b_{r_1}'}+...+\boldsymbol{b_{r_k}'}=0,$ we get $$\frac{a_{j_1}(q)...a_{j_k}(q)}{a_{i}(q)^k}=a_{r_1}'(q)...a_{r_k}'(q).$$ Now we have reduced the case to $i=3$. Thus $$\frac{a_{j_1}(q)...a_{j_k}(q)}{ka_{i}(q)^k}$$ is a non-constant polynomial of $q$. Near the large radius limit, for any $l$, we know that $$\sum\limits_{k<l,b_{j_1}+...+b_{j_k}=kb_{i}}\frac{a_{j_1}(q)...a_{j_k}(q)}{ka_{i}(q)^k}$$ is a polynomial of $q$ with the constant 0. Besides, near the large radius limit, there exists $z_0=(\frac{1}{3},\frac{1}{3})=(x_{0,(\sigma,\tau)},y_{0,(\sigma,\tau)})$ making $$\sum\limits_{i\ne 3}|a_i'(q)||x_{0,(\sigma,\tau)}|^{m_i'}|y_{0,(\sigma,\tau)}|^{n_i'}<\frac{3}{4}.$$

Then \begin{align*}
    &\sum\limits_{k\ge l,b_{j_1}+...+b_{j_k}=kb_{i}}|\frac{a_{j_1}(q)...a_{j_k}(q)}{ka_{i}(q)^k}|\\&<\sum\limits_{k\ge l,}|\frac{a_{j_1}(q)...a_{j_k}(q)}{ka_{i}(q)^k}||x_0^{m_{j_1}+...+m_{j_k}-km_i}||y_0^{n_{j_1}+...+n_{j_k}-kn_i}|\\&=(\sum\limits_{i\ne 3}|a_i'(q)||x_{0,(\sigma,\tau)}|^{m_i'}|y_{0,(\sigma,\tau)}|^{n_i})^l\sum\limits_{i=0}^{+\infty}(-1)^{i+l-1}\frac{(\sum\limits_{i\ne 3}|a_i'(q)||x_{0,(\sigma,\tau)}|^{m_i'}|y_{0,(\sigma,\tau)}|^{n_i})^{i}}{i+l}\\&<(\frac{3}{4})^l(\sum\limits_{i=0}^{+\infty}(\frac{3}{4})^i)=4(\frac{3}{4})^l.
\end{align*} We have proved that near the large radius limit, $|\gamma_i-log|a_i(q)||<1,\forall i$.

Note that if $i_1=3,i_2=2$, then $-1<\gamma_{i_1}<1$, $-1<\gamma_{i_2}<1$. We consider the point $w_0=(w_{1,0},w_{2,0})=(-3,\gamma_{i_1}-\gamma_{i_2})$. Near the large radius limit, any tropical monomial of the tropical polynomial $S(w)=max_{i=1,..,p+3}\{\gamma_i+m_iw_1+n_iw_2\}$ is either one of $\gamma_{i_1},\gamma_{i_2}+w_2,\gamma_{1}+w_1$ or one of $\gamma_i+m_iw_1+n_iw_2,i>3$. But we have shown that $\gamma_i<log|a_i(q)|+1$ for all $i>3$. Thus near the large radius limit, we have $\gamma_i+m_iw_{1,0}+n_iw_{2,0}<-2$. When $i=1$, $\gamma_1+w_{1,0}<-3+1=-2$. Besides, $\gamma_{i_1}=\gamma_{i_2}+w_{2,0}>-1$. Now we have proved that $w_0\in \sigma_{i_1,i_2}^*$ for $i_1=3$ and $i_2=2$. Similarly, with the symmetry Proposition \ref{p2.5.2}, we know that for any $(i_1,i_2)$ satisfying the condition that $b_{i_1}$ connects $b_{i_2}$ directly by a 1-cell in $T_{\Sigma}$, $\sigma_{i_1,i_2}^*$ is non-empty near the large radius limit.

On the other hand, if $i_1=3$ but $b_{i_1}$ doesn't connect $b_{i_2}$ directly in $T_{\Sigma}$, we choose a flag $(\tau,\sigma)$ contains $b_{i_1}$ as the origin, and $b_{j_1}$, $b_{j_2}$ as another two vertices making $b_{i_2}$ lies in the cone spanned by the flag, i.e. $\boldsymbol{b_{i_2}}-\boldsymbol{b_{i_1}}=\lambda_1 e_{1,(\tau,\sigma)}+\lambda_2e_{2,(\tau,\sigma)}$ with non-negative $\lambda_1,\lambda_2$. Because $b_{i_1}$ does not connect $b_{i_2}$ directly, we know that $\lambda_1+\lambda_2\ge 2$. Figure \ref{f8} shows the case.     

\begin{figure}
    \centering
    \includegraphics[scale=0.5]{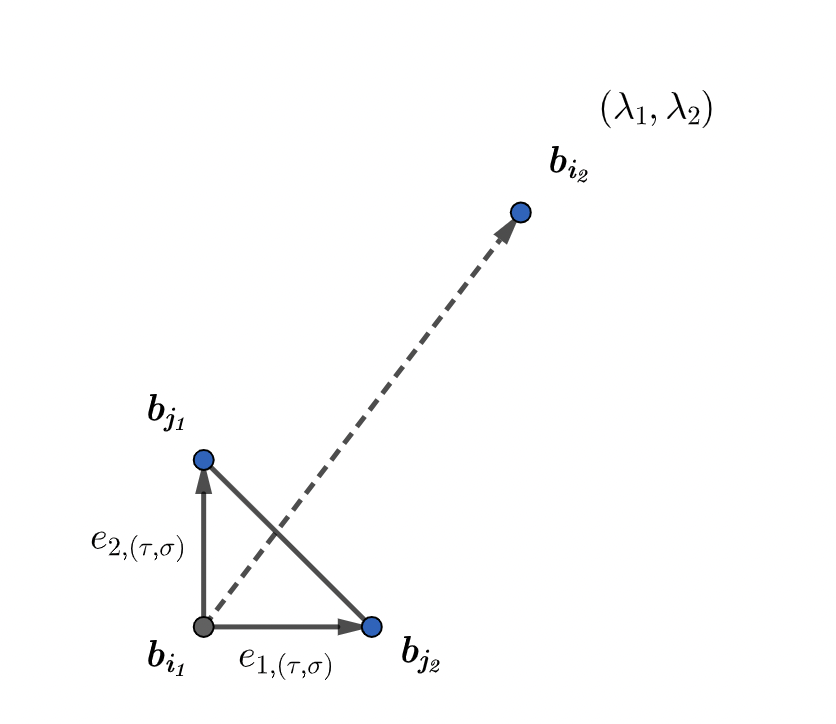}
    \caption{}
    \label{f8}
\end{figure}
By Proposition \ref{p2.5.4}, we know that $\frac{a_{i_2}(q)a_{i_1}(q)^{\lambda_1+\lambda_2-1}}{a_{j_1}(q)^{\lambda_1}a_{j_2}(q)^{\lambda_2}}$ is a non-constant monomial of $q$. Then near the large radius limit, we have $$|\frac{a_{i_2}(q)a_{i_1}(q)^{\lambda_1+\lambda_2-1}}{a_{j_1}(q)^{\lambda_1}a_{j_2}(q)^{\lambda_2}}|<\frac{1}{e^{2\lambda_1+2\lambda_2}}.$$ Taking the logarithm, we get $$log|a_{i_2}(q)|+(\lambda_1+\lambda_2-1)log|a_{i_1}(q)|-\lambda_1log|a_{j_1}(q)|-\lambda_2log|a_{j_2}(q)|<-2\lambda_1-2\lambda_2.$$ However, for any $(w_1,w_2)\in \mathbb{R}^2$, \begin{align*}
    &\gamma_{i_1}+m_{i_1}w_1+n_{i_1}w_2+(\lambda_1+\lambda_2-1)(\gamma_{i_2}+m_{i_2}w_1+n_{i_2}w_2)\\&-\lambda_1(\gamma_{j_1}+m_{j_1}w_1+n_{j_1}w_2)-\lambda_2(\gamma_{j_2}+m_{j_2}w_1+n_{j_2}w_2)\\&=\gamma_{i_1}+(\lambda_1+\lambda_2-1)\gamma_{i_2}-\lambda_1\gamma_{j_1}-\lambda_2\gamma_{j_2}\\&=log|a_{i_2}(q)|+(\lambda_1+\lambda_2-1)log|a_{i_1}(q)|-\lambda_1log|a_{j_1}(q)|-\lambda_2log|a_{j_2}(q)|\\&+(log|a_{i_1}(q)|-\gamma_{i_1})+(\lambda_1+\lambda_2-1)(log|a_{i_2}(q)|-\gamma_{i_2})-\lambda_1(log|a_{j_1}(q)|-\gamma_{j_1})\\&-\lambda_2(log|a_{j_2}(q)|-\gamma_{j_2})\\&<-2(\lambda_1+\lambda_2)+2(\lambda_1+\lambda_2)=0
\end{align*} 
Because $\lambda_1+\lambda_2-1\ge 1$, $\lambda_1,\lambda_2\ge 0$, the previous inequality contradicts the fact that  $\gamma_{i_1}+m_{i_1}w_1+n_{i_1}w_2=\gamma_{i_2}+m_{i_2}w_1+n_{i_2}w_2=max\{\gamma_i+m_iw_1+n_iw_2\}_{i=1,...,p+3}$. Therefore, we know that $\sigma_{i_1,i_2}^*$ is empty near the large radius limit. Now we have constructed all dual 1-cells in $T$, then these 1-cells uniquely determine the dual 0-cells and 2-cells. Dual 0-cells are the faces of these 1-cells. Dual 2-cells are those areas bounded by the 1-cells. Figure \ref{f9} and Figure \ref{f10} show an example. Here Figure \ref{f9} gives a toric Calabi-Yau 3-fold with the mirror curve $1+x+y+q_1y^2+q_1q_2xy=0$. Figure \ref{f10} shows the amoeba (red lines) and the tropical spine (black lines) for such a mirror curve.
\begin{figure}[htbp]
\centering
\begin{minipage}[t]{0.48\textwidth}
\centering
\includegraphics[width=6cm]{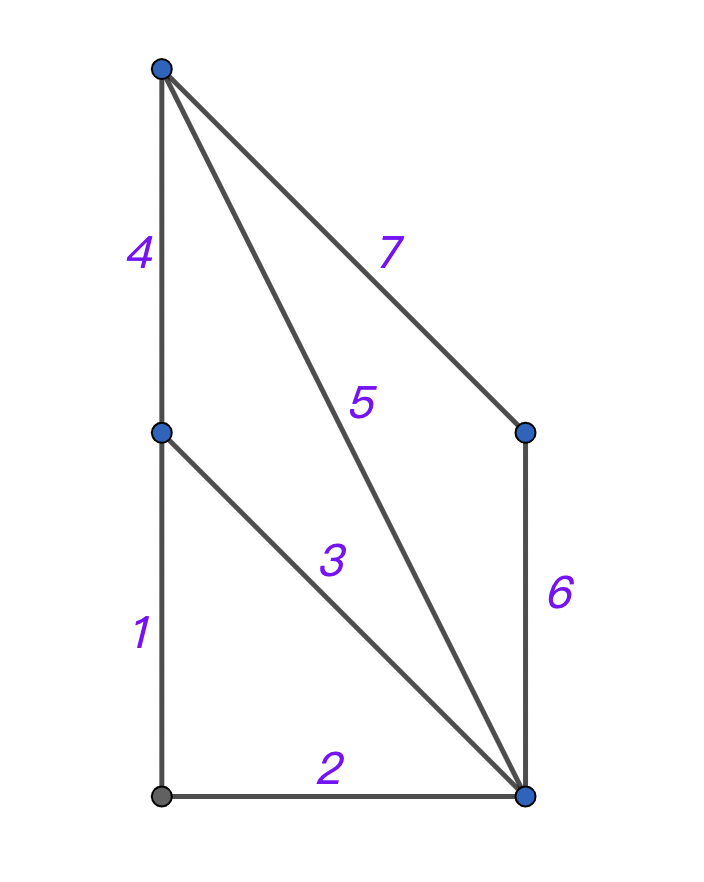}
\caption{}
\label{f9}
\end{minipage}
\begin{minipage}[t]{0.48\textwidth}
\centering
\includegraphics[width=6cm]{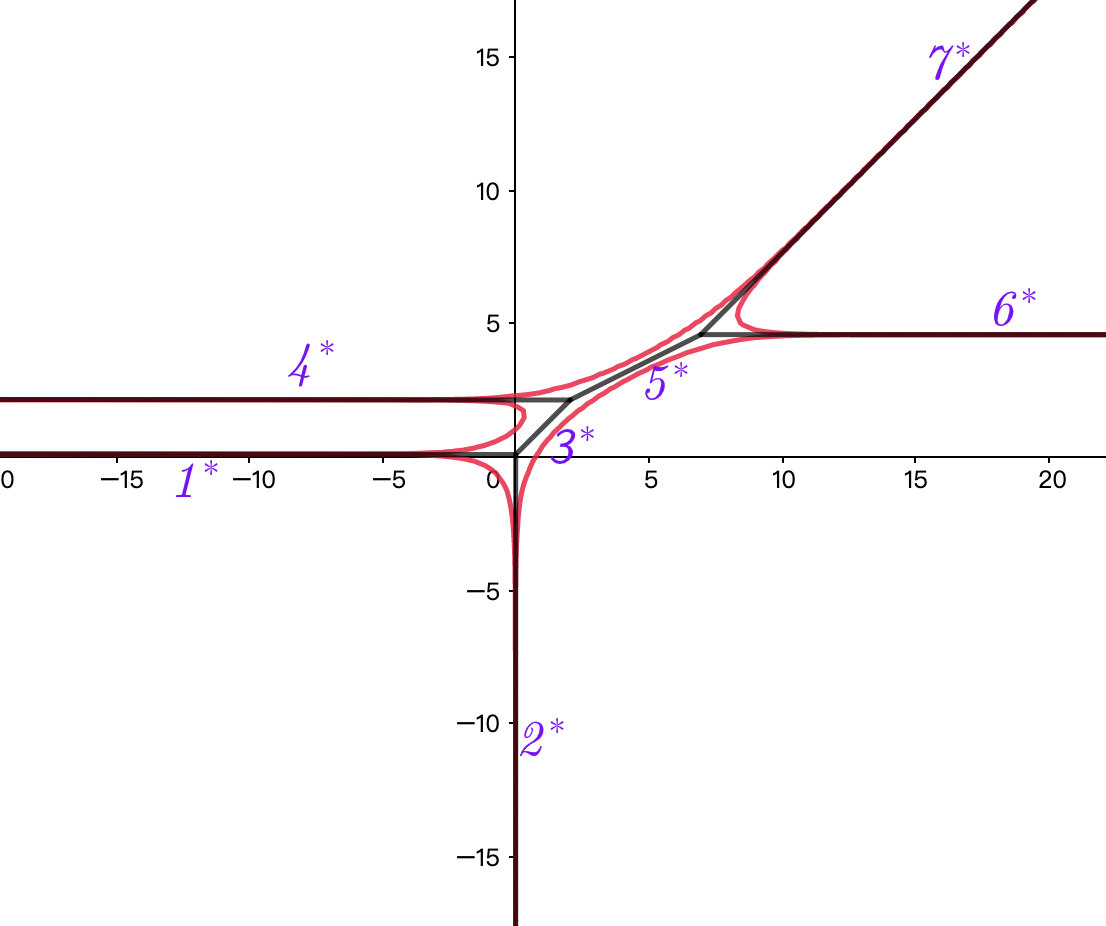}
\caption{}
\label{f10}
\end{minipage}
\end{figure}
\end{proof}

\begin{remark1}One could also give a dual subdivision of $T_{\Sigma}$ by the methods in symplectic geometry\cite{fang2020remodeling}. For a symplectic toric Calabi-Yau 3-fold $X$, we consider the maximal torus action of $\mathbb{T}^3\subset (\mathbb{C}^*)^3$ and see the action as a Hamiltonian action. Here one could also see the symplectic metric as the induced symplectic metric given by the K$\ddot{a}$hler quotient of the canonical symplectic metric on $\mathbb{C}^{p+3}$. We denote the moment map by $\mu_{\mathbb{T}_{\mathbb{R}}}$, where $\mu_{\mathbb{T}_{\mathbb{R}}}: X\to M_{\mathbb{R}}$ after we identify the Lie algebra of $\mathbb{T}^3$ with $M_{\mathbb{R}}$. Then we take the projection $\pi$ from $M_{\mathbb{R}}$ to $M'_{\mathbb{R}}$ and consider the composition $\mu_{\mathbb{T}_{\mathbb{R}}}'=\pi*\mu_{\mathbb{T}_{\mathbb{R}}}$. Finally, we define the union of all 1-dimensional and 0-dimensional $\mathbb{T}^3$-orbits of $X$ by $X_1$. Consider the image $\mu_{\mathbb{T}_{\mathbb{R}}}'(X_1)$ of $X_1$ under $\mu_{\mathbb{T}_{\mathbb{R}}}'$. The image is called the toric graph. Actually, the two subdivisions given by the tropical spine and the toric graph of $\mathbb{R}^2$ are both dual subdivisions of the triangulation $T_{\Sigma}$.
\end{remark1}
\section{Cyclic-M curve and its amoeba}
\label{4}
In Section \ref{2}, we have introduced the amoeba map $\mu$. Then in Section \ref{3}, we have calculated the tropical spine of the amoeba of the mirror curve for a smooth toric Calabi-Yau 3-fold $X_{\Sigma}$ directly. Such the spine gives a dual subdivision of $T_{\Sigma}$. If we want to figure out deeper propositions of the mirror curve itself, not only the amoeba of the mirror curve, we also wish the amoeba map $\mu$ to have some good propositions. For example, we hope $\mu$ is locally a two-fold covering map in the interior of the amoeba and an embedding on the boundary. However, such good propositions do not always hold, even for the mirror curve with real coefficients near the large radius limit. Therefore, we must suggest some new requirements for the mirror curve. 

In this section, we mainly recall the notation of the cyclic-M curve and the propositions about the cyclic-M curve introduced by Mikhalkin\cite{mikhalkin2000real}. These propositions also explain why we hope the mirror curve is a cyclic-M curve. 

For a non-singular algebraic curve $A=\{(x,y)\in (\mathbb{C}^*)^2|p(x,y)=0\}$ with the Newton polytope $\Delta_p$, we could do the compactification on the toric surface $X_{\Delta_p}$ and get $\bar{A}$. For a non-singular real algebraic curve $\mathbb{R}A=\{(x,y)\in (\mathbb{R}^*)^2|p(x,y)=0\}$, we could also do the compactification on the real toric surface $\mathbb{R}X_{\Delta_p}$. \cite{fang2020remodeling} shows that if $g$ is the number of interior lattice points of $\Delta_p$, there exist $g$ linear-independent differential 1-forms on $\bar{A}$, and $g$ is the maximal number. Then the genus of $\bar{A}$ on $X_{\Delta_p}$ equals $g$. By Harnack's inequality\cite{harnack1876ueber}, $\mathbb{R}\bar{A}=\bar{A}\cap \mathbb{R}X_{\Delta_p}$ has at most $g+1$ connected components on the real toric surface $\mathbb{R}X_{\Delta_p}$ which is the closure of $(\mathbb{R}^*)^2$ in $X_{\Delta_p}$. Now we could give an exact definition of the M-curve. 
\begin{definition1}
\label{def:M-curve}
Let $p$ be a polynomial with real coefficients and $g$ be the number of interior lattice points of $\Delta_p$. A non-singular real algebraic curve $\mathbb{R}A=\{(z_1,z_2)\in (\mathbb{R^*})^2|p(z_1,z_2)=0\}$ is an M-curve iff $\mathbb{R}\bar{A}$ has g+1 connected components on the real toric surface $\mathbb{R}X_{\Delta_p}$.
\end{definition1}
 
Then we explain what is a cyclic-M curve. Here the `cyclic' means the order of an M-curve intersecting the axes of the corresponding real toric surface in a cyclic order. If $\mathbb{R}A$ is an M-curve and the real toric surface $\mathbb{R}X_{\Delta_p}$ has axes $l_1,...,l_n$, any connected component of $\mathbb{R}\bar{A}$ is homeomorphic to a circle. Thus, if we want to talk about the cyclic order, we must first assume there is only one component $W$ of $\mathbb{R}\bar{A}$ intersecting the axes of $\mathbb{R}X_{\Delta_p}$. Besides, we also require that there exist arcs $c_1,...,c_n$ of $W$, which they do not intersect with each other, and $c_j$ only intersects $l_j$ with $d_j$ points, where $d_j$ is the number of lattice points on the side of $\Delta_p$ corresponding to $l_j$ minus 1. We define $d_j$ as the integer length of $\Delta_p$. Then $d_j$ is also the degree of $p$ restricted on $l_j$. Finally, we must require the order of these arcs to coincide with the order of sides of $\Delta_p$. We write the definition formally right now.
\begin{definition1}
\label{def:cyclic-M curve}
Let $\mathbb{R}A$ be an M-curve defined by a polynomial $p(x,y)=0$ with real coefficients and $\Delta_p$ be the Newton polytope of $p$. We call $\mathbb{R}A$ an cyclic-M curve iff 
\begin{enumerate}
    \item There is only one connected component $W_p$ of $\mathbb{R}\bar{A}$ intersecting axes of $X_{\Delta_p}$.
    \item There exist non-intersecting arcs $c_i$ of $W_p$ intersecting axis $l_i$ at $d_i$ points, and not intersecting the other axes, where $l_i$ is the axis defined by the side $\Delta_i$ of $\Delta_p$, and $d_i$ is the integer length of $\Delta_i$.
    \item The order of $c_i$ on $W_p$ coincides with the order of $\Delta_i$ on $\Delta_p$. 
\end{enumerate}
\end{definition1} 

\begin{remark1}
An M-curve is not necessarily a cyclic-M curve. For example, Figure \ref{f11} gives a non-singular M-curve given by $p(z_1,z_2)=1+z_1+z_2+qz_1z_2$, $0<q<1$. However, this curve is not a cyclic-M curve because the order of $\Delta_i$ does not coincide with the order of arc $c_i$ on $W_p$. 
\end{remark1}

\begin{figure}
    \centering
    \includegraphics[scale=0.3]{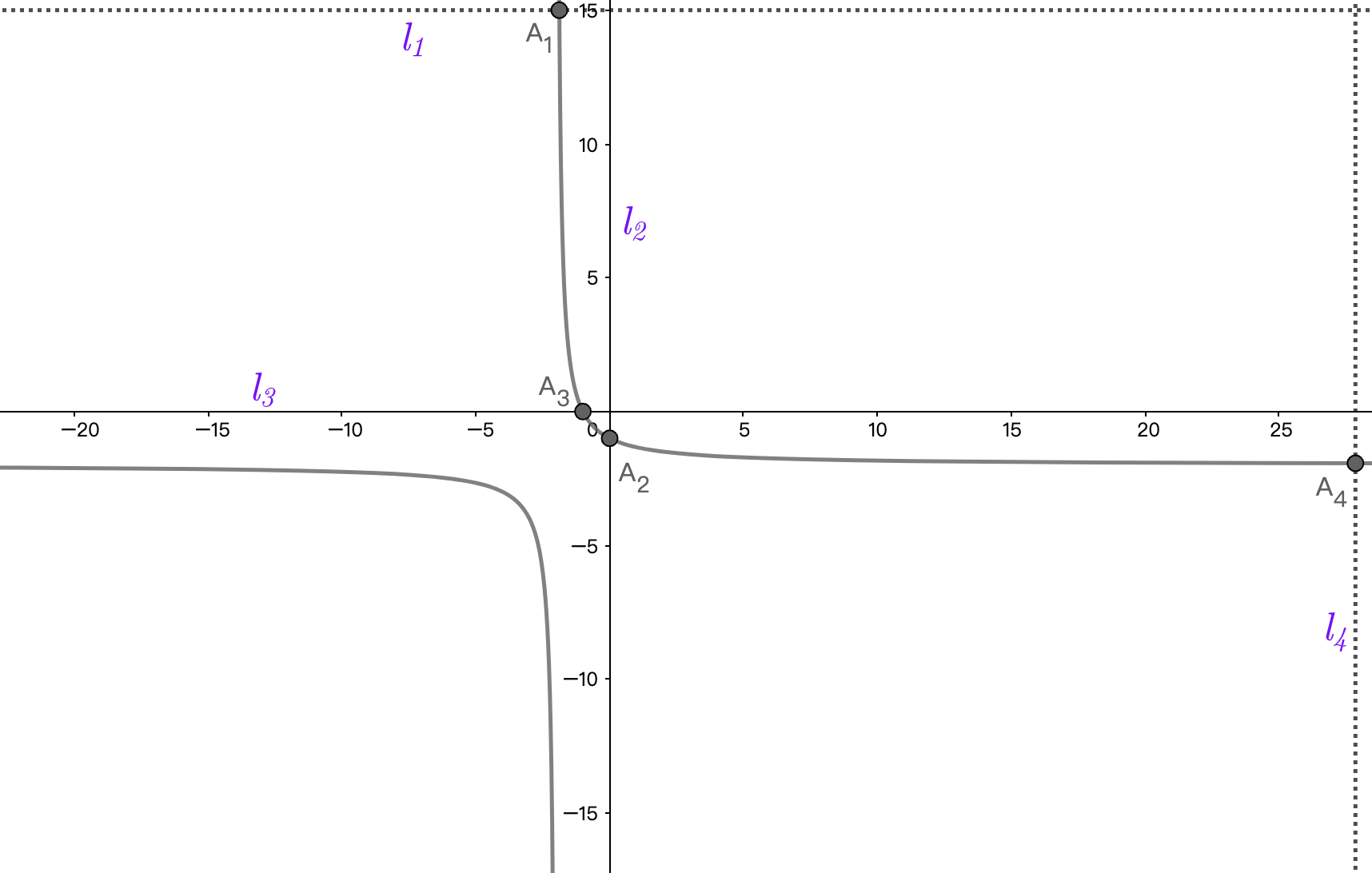}
    \caption{$1+z_1+z_2+0.5z_1z_2=0$}
    \label{f11}
\end{figure}
The cyclic condition is important here because, without this condition, there may exist some inflection points on the boundary of the amoeba of $p$, which may destroy the good propositions of the amoeba map. 

In 2000, Mikhalkin proved the following propositions and a theorem about cyclic-M curves\cite{mikhalkin2000real}. For short, we only list the statements here. In the paper\cite{mikhalkin2000real}, one can read the complete proof.

\begin{proposition1}
\label{prop:F=RA}
Let $p$ be a Laurent polynomial with real coefficients and $A=\{(x,y)\in(\mathbb{C^*})^2|p(x,y)=0\}$. If $\mathbb{R}A=A\cap(\mathbb{R^*})^2$ is a cyclic-M curve, we define $F$ to be the locus of critical points of $\mu|_A$, where $\mu$ is the amoeba map. Then we have $F=\mathbb{R}A$.  
\end{proposition1}

\begin{proposition1}
\label{prop: boundary embedding}
Let $p$ be a Laurent polynomial with real coefficients and $A=\{(x,y)\in(\mathbb{C^*})^2|p(x,y)=0\}$. If $\mathbb{R}A$ is a cyclic-M curve, then $\mu(\mathbb{R}A)=\partial\mu(A)$. Besides, the amoeba map $\mu$ is an embedding on $\mathbb{R}A$.
\end{proposition1}

\begin{theorem1}
\label{tp type}
Let $p$ be a Laurent polynomial with real coefficients and $A=\{(x,y)\in(\mathbb{C^*})^2|p(x,y)=0\}$. If $\mathbb{R}A=A\cap(\mathbb{R^*})^2$ is a cyclic-M curve, the topological type of $(\mathbb{R}X_{\Delta_p};\mathbb{R}A, l_1\cup...\cup l_n)$ is uniquely determined by $\Delta_p$.
\end{theorem1}
With such propositions, we know that if the mirror curve $H(x,y,q)=0$ is a cyclic M-curve, there don't exist inflection points on the boundary of the amoeba of the mirror curve. Thus every point inside the amoeba has exactly two preimage points, and on the boundary, the amoeba map is an embedding. Then we could figure out the local topology of the mirror curve such that the mirror curve is glued from some pairs of pants and tubes. However, for general choices of $q$, the mirror curve $H(x,y,q)=0$ isn't always a cyclic-M curve, in the next section, we will see that under a special family of $q$, the mirror curve satisfies the cyclic-M condition.
\section{Cyclic M-condition of the mirror curve and its topology}
\label{5}
In Section \ref{3}, we have calculated the tropical spine of the amoeba of mirror curve for a smooth toric Calabi-Yau 3-fold $X_{\Sigma}$ directly, and such a spine gives a dual subdivision of $T_{\Sigma}$. Then in Section \ref{4}, we recall the propositions introduced by Mikhalkin. The most important one in this paper is that if $F=0$ is a cyclic-M curve on the real toric surface, then the boundary of the amoeba is exactly the image of zeroes in $(\mathbb{R}^*)^2$ under the amoeba map Besides, the amoeba map is 2-1 in the interior of amoeba and an embedding on the boundary.

However, the mirror curve with general real coefficients does not always satisfy either M-condition or cyclic condition. For example, the mirror curve of $\Sigma_1$, $$H_1(x,y,q)=1+x+y+qx^{-1}y^{-1}=0,$$ when $q>0$, is not an M-curve as we see in Figure \ref{f12} and Figure \ref{f13}. The mirror curve of $\Sigma_2$, $$H_2(x,y,q)=1+x+y+qxy=0,$$ when $q>0$ doesn't satisfy the cyclic condition as we see in Figure \ref{f14} and Figure \ref{f15} as well.

\begin{figure}[htbp]
\centering
\begin{minipage}[t]{0.48\textwidth}
\centering
\includegraphics[width=6cm]{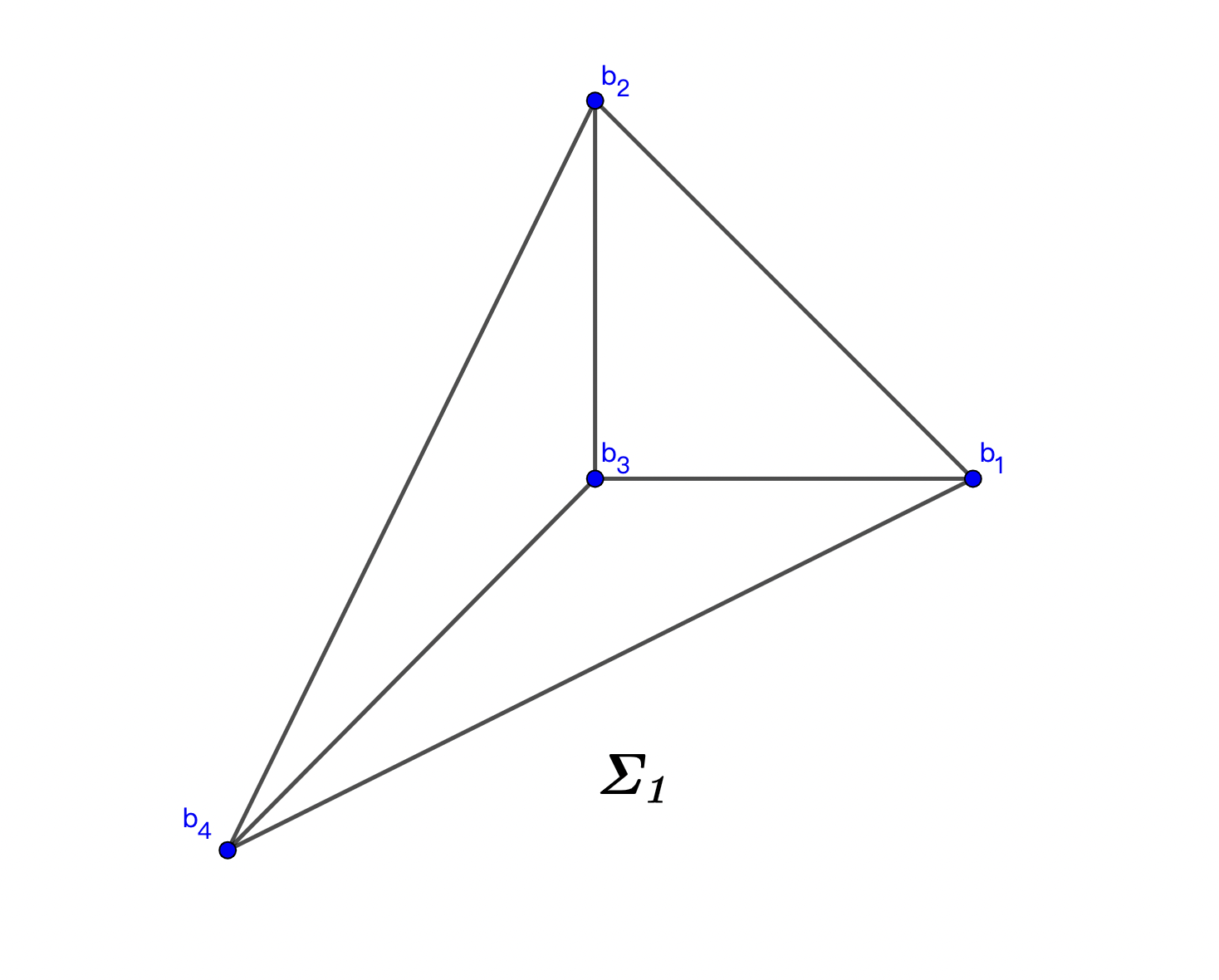}
\caption{$\Sigma_1$}
\label{f12}
\end{minipage}
\begin{minipage}[t]{0.48\textwidth}
\centering
\includegraphics[width=6cm]{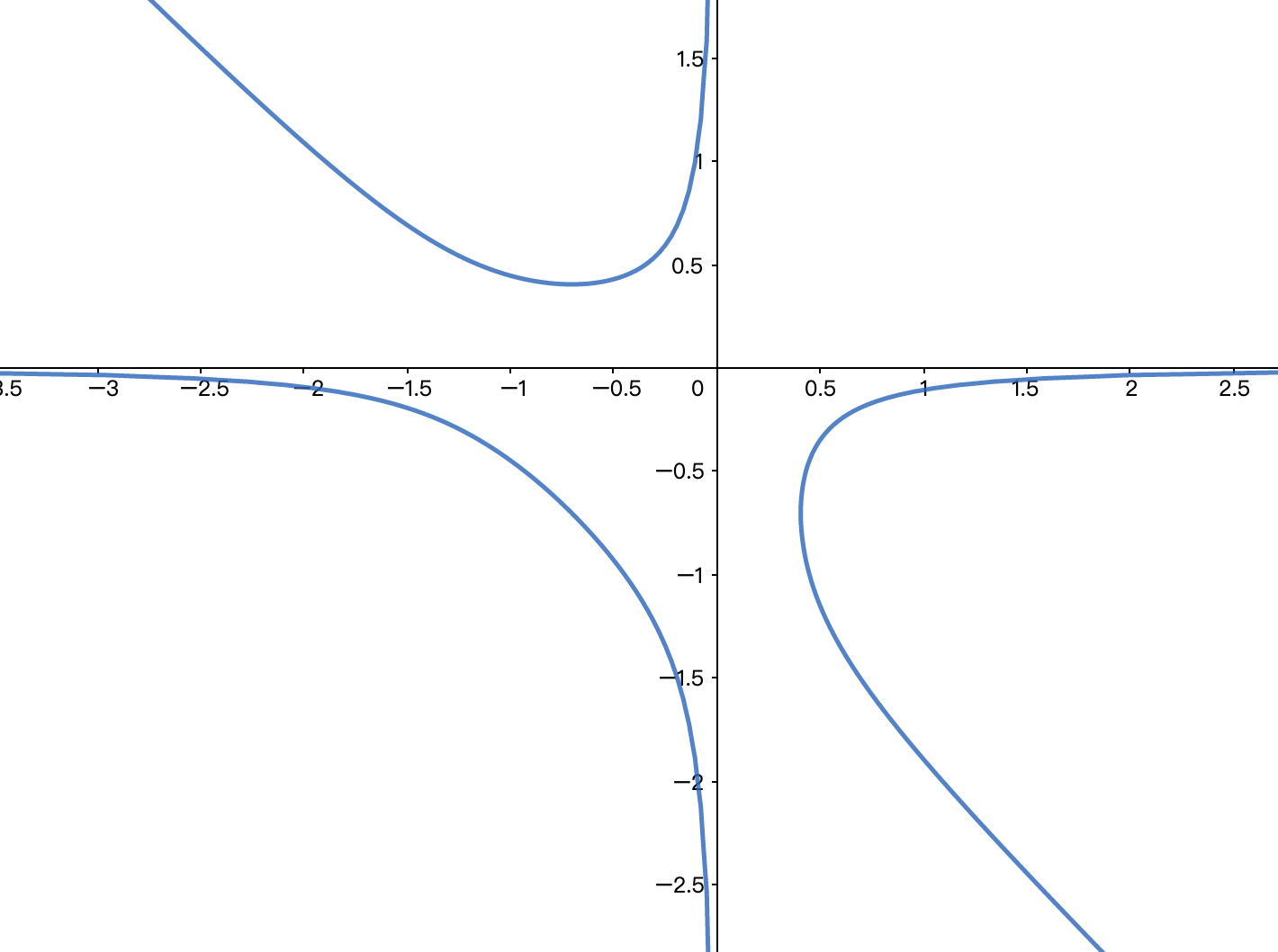}
\caption{$1+x+y+qx^{-1}y^{-1}=0$}
\label{f13}
\end{minipage}
\end{figure}

\begin{figure}[htbp]
\centering
\begin{minipage}[t]{0.48\textwidth}
\centering
\includegraphics[width=6cm]{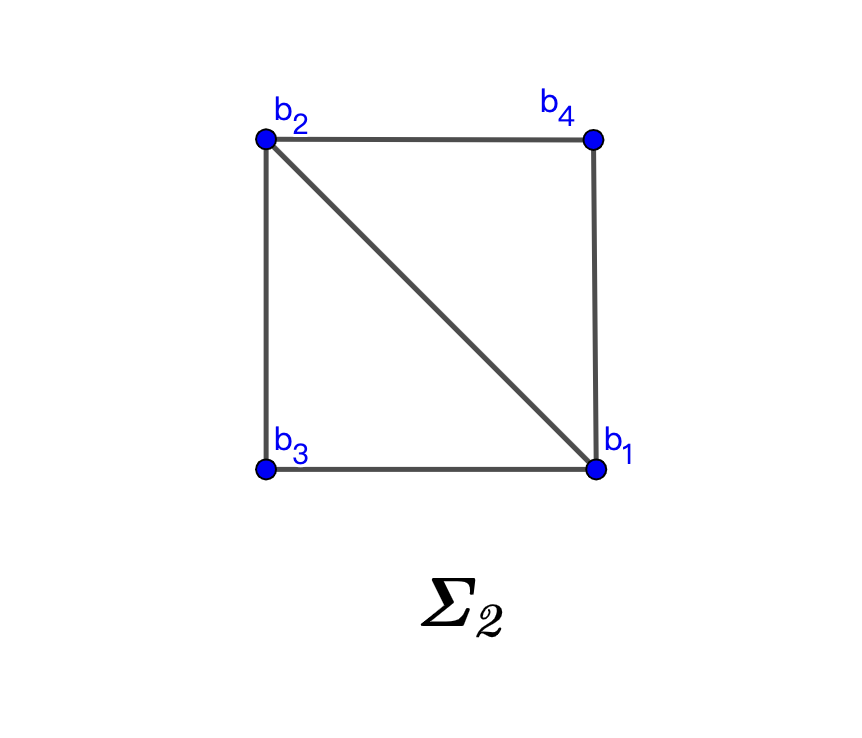}
\caption{$\Sigma_2$}
\label{f14}
\end{minipage}
\begin{minipage}[t]{0.48\textwidth}
\centering
\includegraphics[width=6cm]{Figure9.png}
\caption{$1+x+y+qxy=0$}
\label{f15}
\end{minipage}
\end{figure}
The main goals of this section are to find a special family of coefficients and to prove that these coefficients make the mirror curve a cyclic-M curve. We finish the construction by choosing the positive or negative sign at every lattice point in the fan.
\subsection{Expected choices of coefficients in the mirror curve}
\label{5.1}
In this chapter, we mainly show how we choose the positive or negative sign for each $a_i(q)$ when we expect the mirror curve satisfies the cyclic-M condition. Besides, we also explain the motivations for such a choice. 

Firstly, we consider the case $T_{\Sigma}$ only contains four 0-cells: $b_3(0,0)$, $b_2(0,1)$, $b_1(1,0)$, $b_4(m_4,n_4)$. Because $T_{\Sigma}$ only contains triangles with area $\frac{1}{2}$ as 2-cells, $b_4$ could only possibly lie in the union of three lines in $N'$: $x=-1$, $y=-1$, $x+y=2$. As shown in the Figure \ref{f16}, there are only 12 cases.  
\begin{figure}
    \centering
    \includegraphics[scale=0.35]{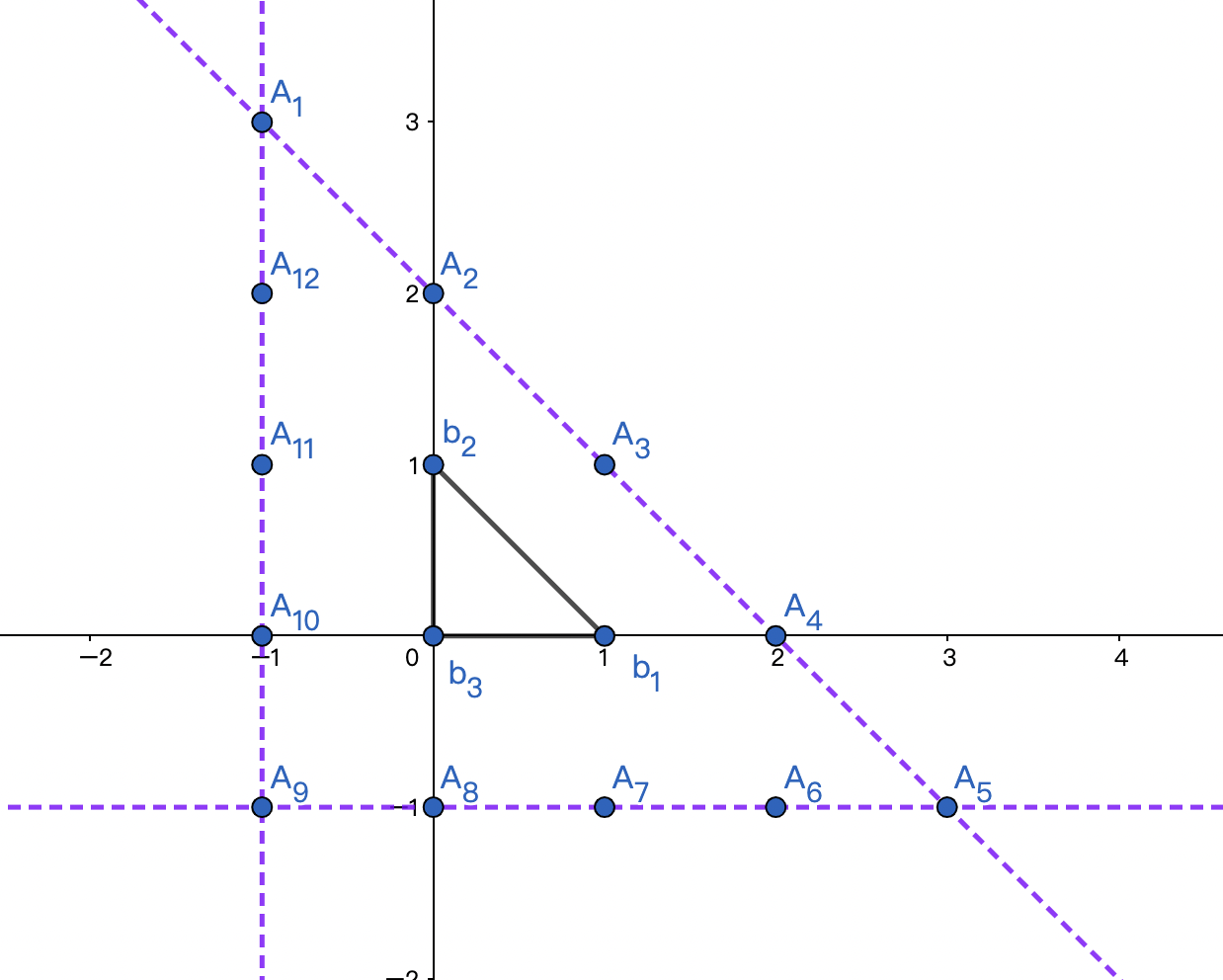}
    \caption{}
    \label{f16}
\end{figure}
Then $S_{124}=S_{\Delta_{b_1b_2b_4}}=\frac{1}{2}$ or $S_{134}=\frac{1}{2}$ or $S_{234}=\frac{1}{2}$. By drawing the mirror curve (we draw two families of cyclic-M curves here in Figure \ref{f17} and Figure \ref{f18}), our expectation must be that $a_i(q)>0$ when $(m_4,n_4)$ satisfies the condition $m_4$ or $n_4$ is even, and $a_i(q)<0$ when $(m_4,n_4)$ satisfies the condition $m_4$ and $n_4$ are both odd.
\begin{figure}[htbp]
\centering
\begin{minipage}[t]{0.48\textwidth}
\centering
\includegraphics[width=6cm]{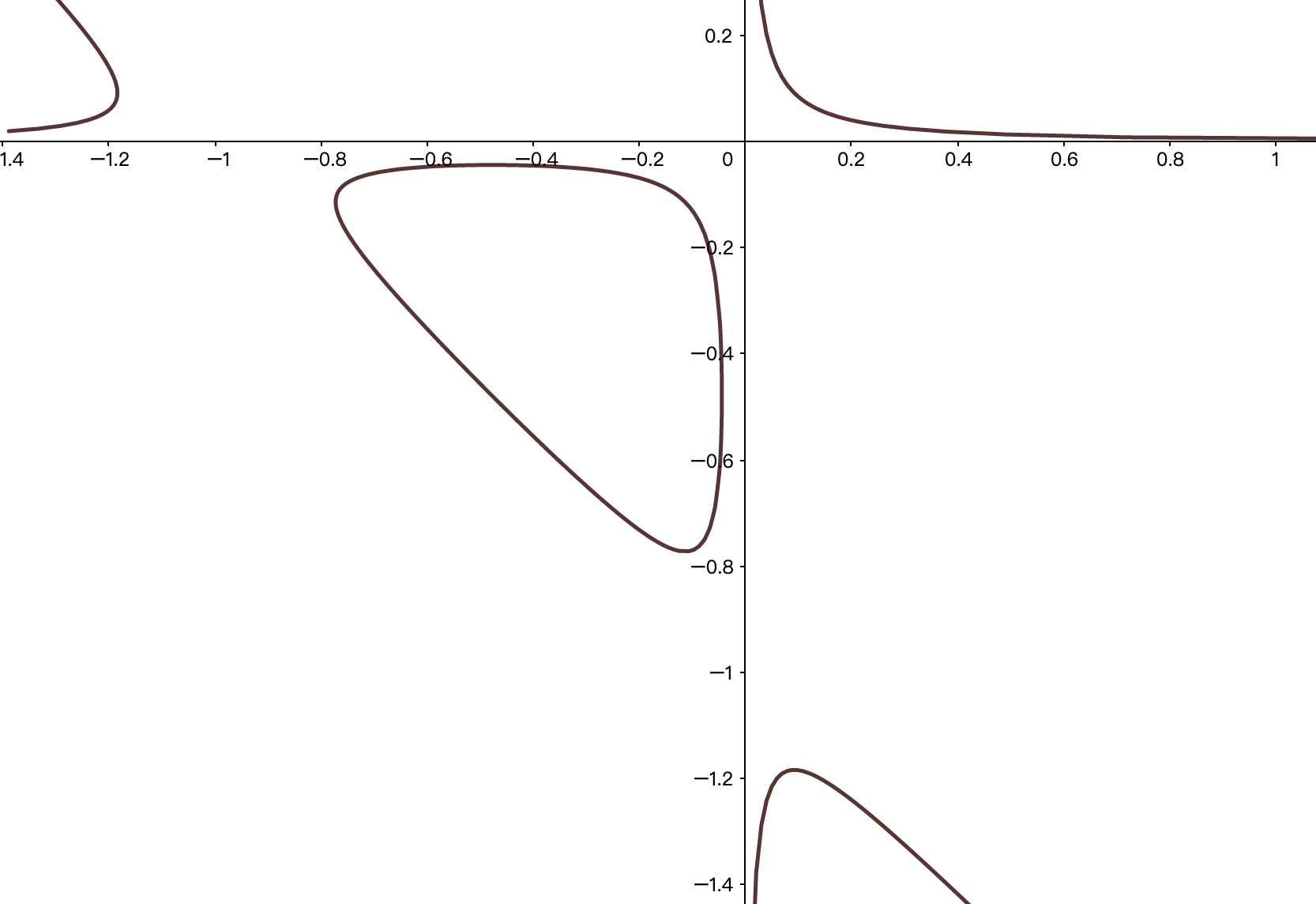}
\caption{$1+x+y-0.01x^{-1}y^{-1}=0$}
\label{f17}
\end{minipage}
\begin{minipage}[t]{0.48\textwidth}
\centering
\includegraphics[width=6cm]{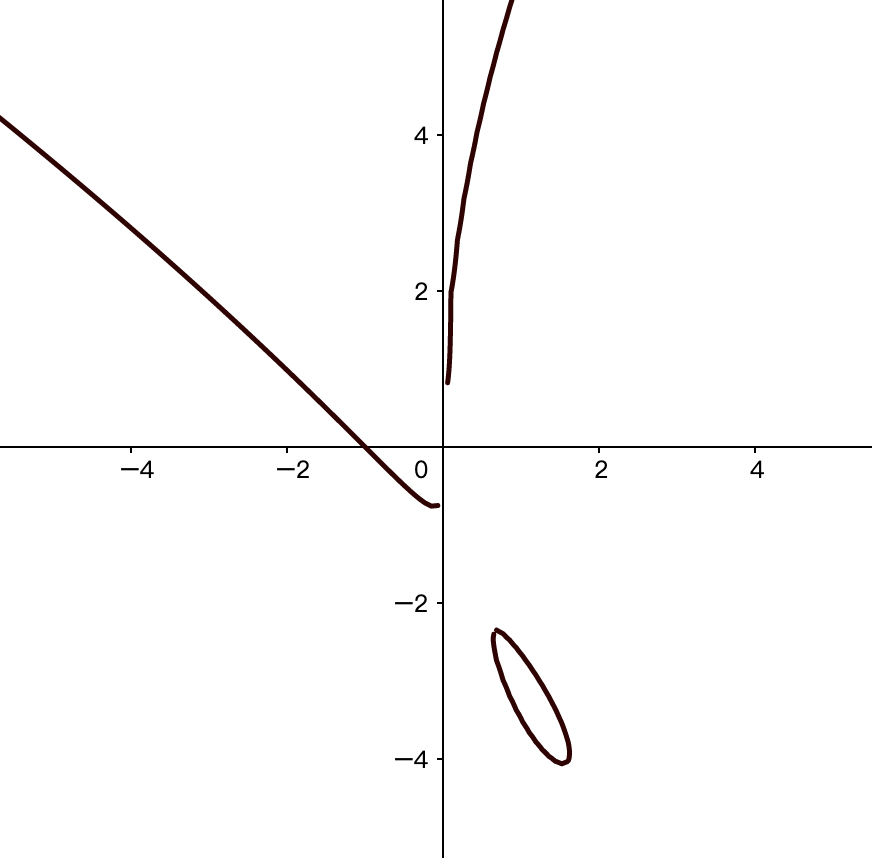}
\caption{$1+x+y-0.02x^{-1}y^3=0$}
\label{f18}
\end{minipage}
\end{figure}

Now we consider a general finest triangulation $T_{\Sigma}$. For any lattice point $b_{i}=(m_i,n_i)$, there exists a chain of triangles $\Delta_1\to\Delta_2\to...\to\Delta_{k_i}$, where $\Delta_1$=$\Delta_{b_1b_2b_3}$, $\Delta_j$ have exact two vertices the same as $\Delta_{j+1}$, and $\Delta_{k_i}$ is the first triangle which contains $b_{i}$ as a vertex in the chain. Assume the unique vertex which belongs to $\Delta_{j+1}$ but does not belong to $\Delta_{j}$ is $b_{r_{j+1}}$. Then we have defined a chain of flag $...\to(\tau_j,\sigma_j)\to...$ in $\Sigma$, where $I_{\tau_j}'=\{s_1,s_2\}$, $I_{\sigma_j}'=\{s_1,s_2,r_j\}$, and $b_{s_1},b_{s_2}$ are other two vertices in $\Delta_i$. We could give the expected sign of $a_{r_{s+1}}(q)$ by taking the affine coordinate change which sends $\Delta_s$ to $\Delta_{(0,0),(1,0),(0,1)}$. Under the affine coordinates change, $b_{r_{s+1}}$ is sent to a lattice point in the three given lines of the previous Figure \ref{f16}. Thus we could give the expected sign of $a_{r_{s+1}}'(q)$. Because we already know the expected sign of $a_e(q)$, where $b_e$ is any vertex of $\Delta_{s}$, we could calculate the expected sign of $a_{r_{s+1}}(q)$ by Proposition \ref{p2.5.4}.  

The signs we give by different chains agree on the same point. Besides, the sign of $a_i(q)$ even only depends on the coordinate of $b_i$ in $N'$. Actually our expectation is that $a_i(q)<0$ when $m_i,n_i$ are both odd, and $a_i(q)>0$ when $m_i$ or $n_i$ is even for all $i=1,2,...,p+3$. We show the fact by direct computation. Given any flag $(\tau,\sigma)$, we consider the coordinate change.
\begin{align*}
X_{(\tau,\sigma)}=X^{a(\tau,\sigma)}Y^{b(\tau,\sigma)}\frac{a_{i_1}(q)}{a_{i_3}(q)} 
\\
Y_{(\tau,\sigma)}=X^{c(\tau,\sigma)}Y^{d(\tau,\sigma)}\frac{a_{i_2}(q)}{a_{i_3}(q)}
\end{align*}

If $\Delta_{s}$ contains $b_{i_1}$, $b_{i_2}$, $b_{i_3}$ as vertices, $\Delta_{s+1}$ contains $b_{i_1}$, $b_{i_2}$, $b_{i_4}$ as vertices, and $b_{i_4}$ has the coordinate $(\lambda_1,\lambda_2)$ under the flag $(\tau_s,\sigma_s)$, because $\Delta_{s+1}$ contains $b_{i_1},b_{i_2}$, we get $\lambda_1+\lambda_2=2$.

Next, we calculate the expected sign of $a_{i_4}(q)$ by $a_{i_1}(q)$, $a_{i_2}(q)$, $a_{i_3}(q)$. In the mirror curve $H_{(\tau,\sigma)}(X_{(\tau,\sigma)},Y_{(\tau,\sigma)},q)=0$, the lattice point $b_{i_4}$ corresponds to the monomial $a'_{4}(q)X_{(\tau,\sigma)}^{\lambda_1}Y_{(\tau,\sigma)}^{\lambda_2}$. By our construction, $a'_{4}(q)<0$ if $\lambda_1$ is odd, and $a'_{4}(q)>0$ if $\lambda_1$ is even. But we know that $$a_{i_4}(q)=a_{i_3}(q)(\frac{a_{i_1}(q)}{a_{i_3}(q)})^{\lambda_1}(\frac{a_{i_2}(q)}{a_{i_3}(q)})^{\lambda_2}a'_4(q).$$ When $\lambda_1$ is even, the sign of $a_{i_4}(q)$ is the same as $a_{i_3}(q)$, and two coordinate components of $b_{i_4}-b_{i_3}$ are both even. When $\lambda_1$ is odd, we have $a_{i_4}(q)a_{i_3}(q)$ have the opposite sign of  $a_{i_1}(q)a_{i_2}(q)$. Then we consider $(a,b)=e_{1,(\tau_s,\sigma_s)}=\boldsymbol{b_{i_1}}-\boldsymbol{b_{i_3}}$, and $(c,d)=e_{2,(\tau_s,\sigma_s)}=\boldsymbol{b_{i_2}}-\boldsymbol{b_{i_3}}$. Because $ad-bc=1$, without loss of generality, we assume $ad$ is odd and $b$ is even. No matter $c$ is even or odd, we could prove that $b_{i_1},b_{i_2},b_{i_3},b_{i_4}$ have different parities by enumeration. Here $(s_1,s_2)$ and $(s_3,s_4)$ have different parities means that $s_1-s_3$ and $s_2-s_4$ have at least one odd integer. Thus by induction, we know $a_{i_4}(q)<0$ when $m_{i_4},n_{i_4}$ are both odd, and $a_{i_4}(q)>0$ when $m_{i_4}$ or $n_{i_4}$ is even.

\begin{aremark} These choices of signs are preserved under any coordinate change. Specifically, if we fix any flag $(\tau,\sigma)$ and consider the mirror curve of such a flag given by $H_{(\tau,\sigma)}(X_{(\tau,\sigma)},Y_{(\tau,\sigma)},q)=0$, where $$H_{(\tau,\sigma)}(X_{(\tau,\sigma)},Y_{(\tau,\sigma)},q)=\sum\limits_{i=1}^{p+3}a'_i(q)X_{(\tau,\sigma)}^{m'_i}Y_{(\tau,\sigma)}^{n'_i}.$$ Then the sign of $a'_i(q)$ given by $a_i(q)$ under the coordinate change agrees with the original sign we give, that $a'_i(q)>0$ when $m'_i$ or $n'_i$ is even, and $a'_i(q)<0$ when $m'_i$ and $n'_i$ are both odd. This could be another explanation for such choices of coefficients.
\end{aremark}
\begin{aremark}In this chapter, we only explain why we expect that when the sign of $a_i(q)$ is given above, near the large radius limit, the mirror curve is a cyclic-M curve by some basic examples and bold calculations. However, we have not yet given any strict proof. 
\end{aremark}
\subsection{M-curve condition}
\label{5.2}
In this chapter, we prove that under the right coefficient signs we give near the large radius limit, the mirror curve is an M-curve. As we defined before, the mirror curve given by $H=0$ is an M-curve iff $H=0$ have $g+1$ connected components on the corresponding real toric surface. Here $g$ is the number of interior lattice points in the Newton polytope of $H$.  

We find these $g+1$ components now. Firstly, there must exist at least one connected component (and exactly one because all the intersection points at infinity would be patched together) intersecting with the axes of the real toric surface, that is because when you do the coordinate change which sends a certain axis to the y-axis, $x=0$, the intersections of the mirror curve and the axis $x=0$ is the zeroes of a Laurent polynomial \begin{align*}
    f(y)=1+y+\sum\limits_{i}a_i(q)y^{m_i}=0.
    \end{align*} Near the large radius limit, we have $f(-\frac{2}{3})>0$ and $f(-\frac{4}{3})<0$. By the intermediate value theorem, $f$ has at least one zero on the axes. Thus the mirror curve intersects with the axes of the real toric surface.

Then we directly find the other $g$ connected components in $(\mathbb{R}^*)^2$, which each has a `domain term' corresponding to an interior lattice point of $P_{\Sigma}$. Later we will explain what is a `domain term'. Because $\Sigma$ only contains finite flags and lattice points, the norm of the coordinate of any lattice point under each flag has a maximal value. For the convention, we use $D$ to stand for an upper bound of such the maximal value. That is to say for any flag $(\tau,\sigma)$ and a lattice point $b_j$ in $P_{\Sigma}$, if the coordinate of $b_j$ under the flag $(\tau,\sigma)$ is $(m_i',n_i')$, then $D\ge \sqrt{m_i'^2+n_i'^2}$.

\begin{theorem}
\label{t5.2.1}
For the mirror curve $H(X,Y,q)=0$ under the choice of signs given by $a_i(q)<0$ when $m_i$ and $n_i$ are both odd, and $a_i(q)>0$ when $m_i$ or $n_i$ is even, near the large radius limit, fixing any interior lattice point $b_i(m_i,n_i)$ of $\Sigma$, there exists a unique connected component $W_{b_i}$ of the mirror curve on the real toric surface $X_{P_{\Sigma}}$ making all points $z_0=(x_0,y_0)\in W_{b_i}$ satisfying that \begin{align*}(1+c)|a_i(q)x_0^{m_i}y_0^{n_i}|>|a_j(q)x_0^{m_j}y_0^{n_j}|\quad\forall j\ne i,\end{align*} where $c$ is a positive constant (irrelevant with $q$) making \begin{align*}(1+c)^{2D}<\frac{4}{3}.\end{align*}
\end{theorem}
For the convention, if $W_{b_i}$ and $b_i$ satisfy the condition of Theorem \ref{t5.2.1}, we call $W_{b_i}$ a domain component of $b_i$, and $b_i$ the domain term of $W_{b_i}$. 

\begin{proof}

Without loss of generality, we assume that $b_3(0,0)$ is an interior lattice point of polytope $P_{\Sigma}$ in $N'$. Then we hope there exists a connected component which is a domain component of $b_3$. 

Because $b_3$ is an interior lattice point, first, we claim that \begin{align*}C_{b_3}=\{(x,y)\in (\mathbb{R}^*)^2|1+c\ge|a_i(q)x^{m_i}y^{n_i}|,\forall i\ne3 \}\end{align*} is a bounded compact subset in $(\mathbb{R}^*)^2$ with the boundary \begin{align*}\partial C_{b_3}=\{(x,y)\in (\mathbb{R}^*)^2|1+c=|a_{i_0}(q)x^{m_{i_0}}y^{n_{i_0}}|\ge|a_{i}(q)x^{m_{i}}y^{n_{i}}|,\exists i_0,\forall i\}.\end{align*} In $(\mathbb{R}^*)^2$, $C_{b_3}$ have 4 parts which are \begin{align*}C_{b_3,++}&=\{(x,y)\in C_{b_3},x>0,y>0\}, \\C_{b_3,+-}&=\{(x,y)\in C_{b_3},x>0,y<0\}, \\C_{b_3,-+}&=\{(x,y)\in C_{b_3},x<0,y>0\}, \\C_{b_3,--}&=\{(x,y)\in C_{b_3},x<0,y<0\}.\end{align*}

We prove our claim now. Obviously, $C_{b_3}=\bigsqcup C_{b_3,\pm\pm}$. We only need to prove $C_{b_3,++}$ is compacted, bounded with the boundary $$\{(x,y)\in (\mathbb{R}^+)^2|1+c=|a_{i_0}(q)x^{m_{i_0}}y^{n_{i_0}}|\ge|a_{i}(q)x^{m_{i}}y^{n_{i}}|, \exists i_0,\forall i\}.$$

Here `bounded' is obvious because $|x|\le 1+c,|y|\le 1+c$.

To show $C_{b_3,++}$ is compact, we only need to show its embedding in $\mathbb{R}^2$ is compact. It is to say $C_{b_3,++}$ is closed in $\mathbb{R}^2$. We only need to prove the closure of ${C_{b_3,++}}$ in $\mathbb{R}^2$ doesn't intersect $x=0$ or $y=0$. It suffices to show $\exists \epsilon_0>0,\epsilon_1>0$ making any $(x_0,y_0)\in C_{b_3,++}$ satisfy $|x_0|>\epsilon_0,|x_1|>\epsilon_1$.

\begin{figure}
    \centering
    \includegraphics[scale=0.4]{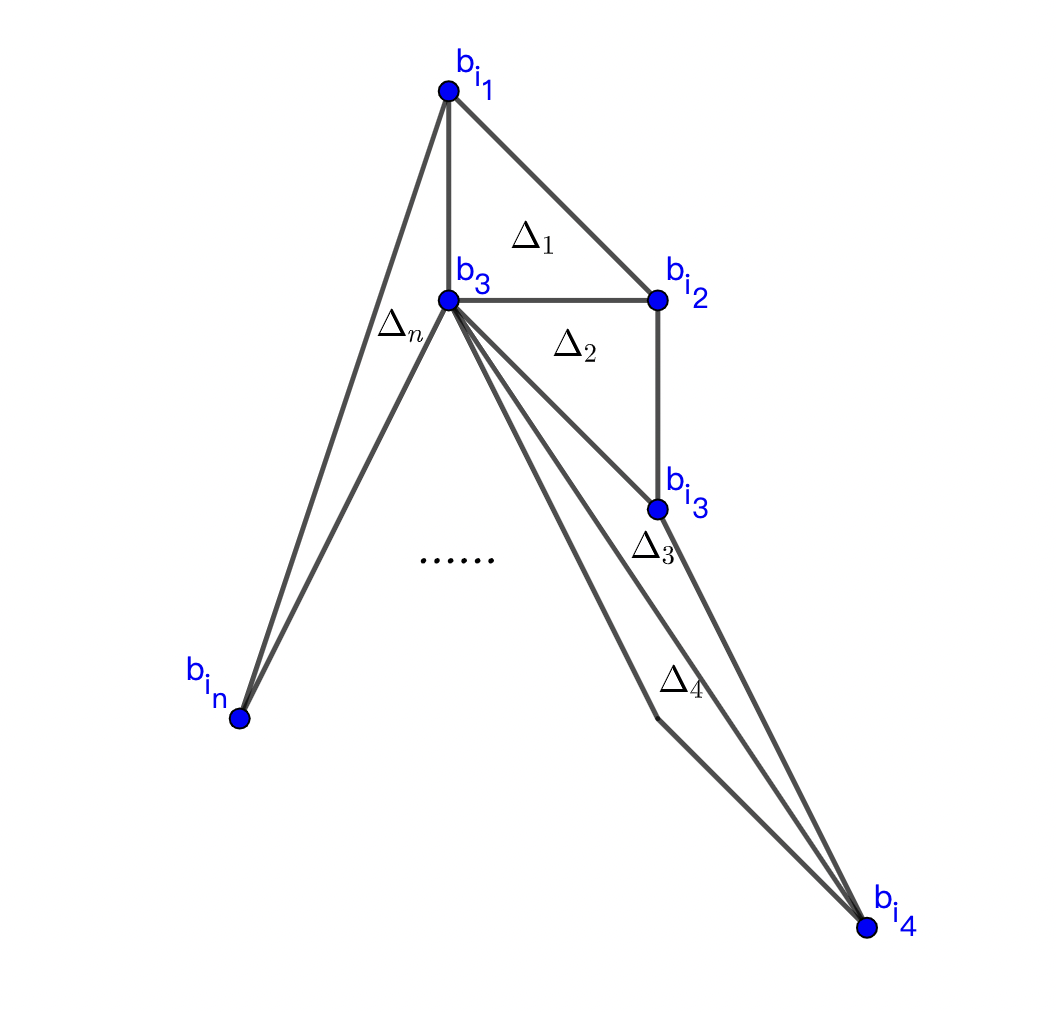}
    \caption{}
    \label{f19}
\end{figure}
Because $b_3(0,0)$ is the interior point of $\Sigma$, we assume that $b_3$ is the vertex of triangle $\Delta_1$, $\Delta_2$,..., $\Delta_n$ in the clockwise order as we show on Figure \ref{f19}, $\Delta_j=\Delta_{b_3b_{i_j}b_{i_{j+1}}}$, where $i_{n+1}=i_{1}$, $i_1=1$ , and $i_2=2$. Then there are n segments in $T_{\Sigma}$ connecting $b_3$ and $n$ other lattice points $b_{i_1}=b_1,b_{i_2}=b_2,...,b_{i_n}$. There must be two adjacent lattice points of $\{b_{i_r}\}_{r=1...n}$, which are denoted by $b_{i_{k}},b_{i_{k+1}}$, satisfying that $\Vec{a}=(-1,-1)$ has a non-negative integer linear coefficient combination under $\boldsymbol{b_{i_k}}$ and $\boldsymbol{b_{i_{k+1}}}$, as we show in Figure \ref{f20}. That is to say that there exists $\lambda_1,\lambda_2\in \mathbb{N}$ and $1\le k\le n$ making \begin{align*}\Vec{a}=\lambda_1\boldsymbol{b_{i_k}}+\lambda_2\boldsymbol{b_{i_{k+1}}}.\end{align*} Then we consider the following inequality \begin{align*}\frac{4}{3}&>(1+c)^{2D}\\&\ge(1+c)^{\lambda_1+\lambda_2}\\&\ge|a_{i_k}(q)x^{m_{i_k}}y^{n_{i_k}}|^{\lambda_1}|a_{i_{k+1}}(q)x^{m_{i_{k+1}}}y^{n_{i_{k+1}}}|^{\lambda_2}\\&=|a_{i_k}(q)|^{\lambda_1}|a_{i_{k+1}}(q)|^{\lambda_2}|x|^{-1}|y|^{-1}.\end{align*} Because $|y|\le1+c$, we then calculate the lower boundary of $|x|$ that \begin{align*}|x|>\frac{3}{4(1+c)}|a_{i_k}(q)|^{\lambda_1}|a_{i_{k+1}}(q)|^{\lambda_2}.\end{align*} Now we have shown that $C_{b_3,++}$ is compact with the boundary $$\partial C_{b_3,++}=\{(x,y)\in (\mathbb{R}^+)^2|1+c=|a_{i_0}(q)x^{m_{i_0}}y^{n_{i_0}}|\ge|a_{i}(q)x^{m_{i}}y^{n_{i}}|, \exists i_0,\forall i\}.$$
\begin{figure}
    \centering
    \includegraphics[scale=0.4]{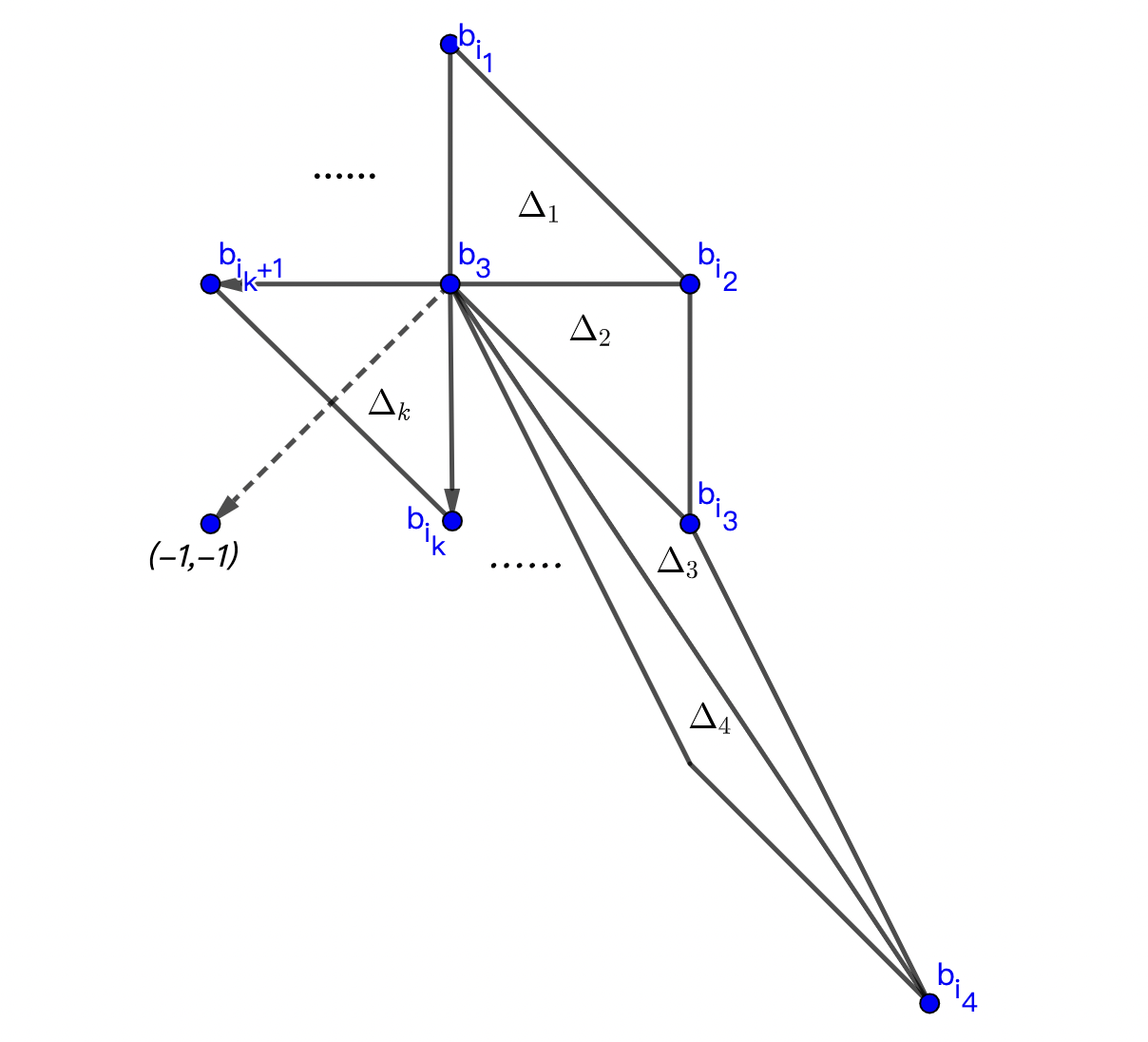}
    \caption{}
    \label{f20}
\end{figure}

Then we want to find a connected component of the mirror curve inside $C_{b_3,--}$. Near the large radius limit, it is easy to find an interior point $(x_0,y_0)=(-\frac{1}{3},-\frac{1}{3})$ of $C_{b_3,--}$ which makes $H(x_0,y_0,q)>0$, because $$H(-\frac{1}{3},-\frac{1}{3},q)=1-\frac{1}{3}-\frac{1}{3}+\sum\limits_{i=4}^{p+3}a_i(q)(-\frac{1}{3})^{m_i+n_i}=\frac{1}{3}+\sum\limits_{i=4}^{p+3}a_i(q)(-\frac{1}{3})^{m_i+n_i}$$, where $(-\frac{1}{3},-\frac{1}{3})$ is an interior point of $C_{b_3,--}$. Except $1$, all monomials in $H(x,y,q)$, taking values at $(-\frac{1}{3},-\frac{1}{3})$, are either $-\frac{1}{3}$ or non-constant monomials of $q$. Thus near the large radius limit, $H(-\frac{1}{3},-\frac{1}{3},q)>0$. 

Next, we show that for all $(x,y)$ on the boundary $\partial C_{b_3,--}$ of $C_{b_3,--}$, near the large radius limit, $H(x,y,q)<0$. With such a result, by the intermediate value theorem, there exists at least one connected component to be the domain component of $b_3(0,0)$.

Now we consider the mirror curve given by 
$$H(X,Y,q)=1+X+Y+\sum\limits_{i=4}^{p+3}a_i(q)X^{m_i}Y^{n_i}.$$
If there exist $(x_0,y_0)$ and $i_0$ which make $$1+c=|a_{i_0}(q)x_0^{m_{i_0}}y_0^{n_{i_0}}|\ge|a_i(q)x_0^{m_i}y_0^{n_i}|, \quad i=1,2,..,p+3,$$ we could prove $b_{i_0}$ must connect $b_3$ directly by a 1-cell in $T_{\Sigma}$. If not, there exists $\Delta_j$ with vertices $b_3,b_{i_j},b_{i_{j+1}}$, which makes $\boldsymbol{b_{i_0}}=\lambda_1\boldsymbol{b_{i_j}}+\lambda_2\boldsymbol{b_{i_{j+1}}}$, where $\lambda_1,\lambda_2$ are non-negative integers and $\lambda_1+\lambda_2\ge 2$ (because $b_{i_0}$ doesn't directly connect $b_3$). By Proposition \ref{p2.5.4}, we know that $$\frac{a_{i_0}(q)a_{3}(q)^{\lambda_1+\lambda_2-1}}{a_{i_j}(q)^{\lambda_1}a_{i_{j+1}}(q)^{\lambda_2}}=\frac{a_{i_0}(q)}{a_{i_j}(q)^{\lambda_1}a_{i_{j+1}}(q)^{\lambda_2}}$$ is a non-constant monomial of $q$ denoted by $f_{i_0}(q)$. 
Then we get \begin{align*}
    \frac{a_{i_0}(q)x_0^{m_{i_0}}y_0^{n_{i_0}}}{(a_{i_j}(q)x_0^{m_{i_j}}y_0^{n_{i_j}})^{\lambda_1}(a_{i_{j+1}}(q)x_0^{m_{i_{j+1}}}y_0^{n_{i_{j+1}}})^{\lambda_2}}=f_{i_0}(q).
\end{align*} Because near the large radius limit,  $|f_{i_0}(q)|<(1+c)^{1-\lambda_1-\lambda_2}$, we have \begin{align*}
    1+c&=|a_{i_0}(q)x_0^{m_{i_0}}y_0^{n_{i_0}}|\\&=|f_{i_0}(q)||a_{i_j}(q)x_0^{m_{i_j}}y_0^{n_{i_j}}|^{\lambda_1}|a_{i_{j+1}}(q)x_0^{m_{i_{j+1}}}y_0^{n_{i_{j+1}}}|^{\lambda_2}\\&\le|f_{i_0}(q)|(1+c)^{\lambda_1+\lambda_2}\\&<1+c.
\end{align*}
This is a contradiction. Thus $b_{i_0}$ must directly connect $b_3$. 

Next, we define $$\partial C_{b_3,--,j}=\{(x_0,y_0)\in (\mathbb{R}^-)^2|1+c=|a_{j}(q)x_0^{m_{j}}y_0^{n_{j}}|\ge|a_i(q)x_0^{m_i}y_0^{n_i}|,\forall i\ne j\}.$$ We have proved $\partial C_{b_3,--,j}= \emptyset$ if $b_j$ doesn't directly connect $b_3$ by 1-cell in $T_{\Sigma}$. So $\partial C_{b_3,--}=\bigcup_{j\in V}\partial C_{b_3,--,j}$, where $V$ is the index set of each $j$ which makes $b_j$ connect $b_3$ directly by a 1-cell in $T_{\Sigma}$.
 
It is enough to prove $H(x_0,y_0,q)<0$ for all points $(x_0,y_0)$ on $\partial C_{b_3,--,j}$ for a fixed $j\in V$ near the large radius limit. If $(x_0,y_0)\in \partial C_{b_3,--,j}$ for some fixed $j\in V$, then $m_j, n_j$ could not be even at the same time because $T_{\Sigma}$ is the finest triangulation. Then we have $|a_j(q)x_0^{m_j}y_0^{n_j}|=1+c$. Also, because $m_j,n_j$ could not be even at the same time, we have $(-1)^{m_j+n_j}a_j(q)<0$. One could prove this by enumerating the 3 cases. Then $a_j(q)x_0^{m_j}y_0^{n_j}=-1-c$. At this time, $$H(x_0,y_0,q)=-c+\sum_{i\ne 3,j}a_i(q)(1+c)^{m_i}y_0^{n_i}.$$ All the $b_j$ which connects $b_3$ directly correspond to negative value $a_j(q)x_0^{m_j}y_0^{n_j}<0$ because $m_j$ and $n_j$ couldn't be even at the same time. 

Among all the monomials in $H(x_0,y_0,q)$, any positive one $a_r(q)x_0^{m_r}y_0^{n_r}$ corresponds to a $b_r'$ which have two even coordinate components that $m_r$ and $n_r$ are both even. If $b_r'\ne (0,0)$, then there exist $b_{i_s}$ and $b_{i_{s+1}}$ which make $$\boldsymbol{b_r}=\lambda_1\boldsymbol{b_{i_s}}+\lambda_2\boldsymbol{b_{i_{s+1}}},$$ where $\lambda_1$ and $\lambda_2$ are non-negative integers, and $\lambda_1+\lambda_2\ge2$. By Proposition \ref{p2.5.4}, we know that \begin{align*}
    &\frac{a_r(q)x_0^{m_r}y_0^{n_r}}{(a_{i_s}(q)x_0^{m_{i_s}}y_0^{n_{i_s}})^{\lambda_1}(a_{i_{s+1}}(q)x_0^{m_{i_{s+1}}}y_0^{n_{i_{s+1}}})^{\lambda_2}}\\=&\frac{a_r(q)}{a_{i_s}(q)^{\lambda_1}a_{i_{s+1}}(q)^{\lambda_2}}
    \end{align*} is a non-constant monomial of $q$ denoted by $f_r(q)$. Because for any $i_s$, we have \begin{align*}
  &|a_r(q)x_0^{m_r}y_0^{n_r}|\\=&|f_r(q)||(a_{i_s}(q)x_0^{m_{i_s}}y_0^{n_{i_s}})^{\lambda_1}(a_{i_{s+1}}(q)x_0^{m_{i_{s+1}}}y_0^{n_{i_{s+1}}})^{\lambda_2}|\\\le&|f_r(q)|(1+c)^{\lambda_1+\lambda_2}\\\le&|f_r(q)|(1+c)^{2D}\\<&\frac{4}{3}|f_r(q)|     \end{align*} 
Therefore, near the large radius limit, $$H(x_0,y_0,q)<-c+\frac{4}{3}\sum_{r}|f_r(q)|<0.$$ Then we have found a domain connected component $W_{b_3}$ of $b_3(0,0)$. Generally, we could define $W_{b_i}$ by taking the affine coordinate change given by a fixed flag $(\tau,\sigma)$ which has $b_{i}$ as the origin, then under the coordinate $(X_{(\tau,\sigma)},Y_{(\tau,\sigma)})$, we find the connected component $W_{b_i}$ which is a domain component of $b_i$. By the affine equivalence, we could write down this connected component by the coordinate $(X,Y)$. It is easy to see the domain component of $b_i$ is irrelevant to the flag we choose.

We only need to prove that near the large radius limit, any connected component $W_{b_i}$ which we have given before could only be the domain component of a unique interior lattice point in $P_{\Sigma}$. It is a complicated proof. In the beginning, we state a lemma.

\begin{lemma}
\label{l5.2.2}
Near the large radius limit, any connected component $W_{b_i}$ which is the domain component of the interior lattice point $b_i(m_i,n_i)$ as we defined previously locates in one of the four components of $(\mathbb{R}^*)^2$: $(\mathbb{R}^-)^2$, $(\mathbb{R}^+)^2$, $(\mathbb{R}^+)*(\mathbb{R}^-)$, $(\mathbb{R}^-)*(\mathbb{R}^+)$ uniquely determined by the parity of $(m_i,n_i)$. Specifically, when $m_i$ and $n_i$ are both even, $W_{b_i}$ locates in $(\mathbb{R}^-)^2$. When $m_i$ and $n_i$ are both odd, $W_{b_i}$ locates in $(\mathbb{R}^+)^2$. When $m_i$ is even and $n_i$ is odd, $W_{b_i}$ locates in $(\mathbb{R}^+)*(\mathbb{R}^-)$. When $m_i$ is odd and $n_i$ is even, $W_{b_i}$ locates in $(\mathbb{R}^-)*(\mathbb{R}^+)$.
\end{lemma}

\begin{proof}
Fix an flag $(\tau,\sigma)$ with $I_{\sigma}'=\{i_1,i_2,i_3\}$ and $I_{\tau}'=\{i_2,i_3\}$. If \begin{align*}\boldsymbol{b_{i_1}}=&\boldsymbol{b_{i_3}}+a\boldsymbol{e_1}+b\boldsymbol{e_2}\\ \boldsymbol{b_{i_2}}=&\boldsymbol{b_{i_3}}+c\boldsymbol{e_1}+d\boldsymbol{e_2}.\end{align*} Then for any $(x_1,y_1)\in W_{b_{i_3}}$, because $(x_{1,(\tau,\sigma)},y_{1,(\tau,\sigma)})$ locates in $W_{b_3}'$, we have 
\begin{align*}
    x_{1,(\tau,\sigma)}=x_1^ay_1^b\frac{a_{i_1}(q)}{a_{i_3}(q)}<0\\
    y_{1,(\tau,\sigma)}=x_1^cy_1^d\frac{a_{i_2}(q)}{a_{i_3}(q)}<0,
\end{align*} where we denote the domain component of $b_3'(0,0)$ written under the coordinate $(X_{(\tau,\sigma)},Y_{(\tau,\sigma)})$ of our fixed flag by $W_{b_3}'$. 

Because $ad-bc=1$, then the signs of $x_1$ and $y_1$ are uniquely determined by $a,b,c,d$ and $a_{i_1}(q),a_{i_2}(q),a_{i_3}(q)$. Next, we finish the proof of the lemma in four cases.

\begin{enumerate}
\item[\textbf{Case 1}]: When $m_{i_3}$ and $ n_{i_3}$ are both even, then $a_{i_3}(q)>0$. Because $ad-bc=1$, $a$ and $b$ contain at least one odd integer. If $a$ is even and $b$ is odd, then $m_{i_1}$ is even and $n_{i_1}$ is odd. Thus $a_{i_1}(q)>0$. Then if $x_1<0$ and $y_1<0$, the sign of $$x_1^ay_1^b\frac{a_{i_1}(q)}{a_{i_3}(q)}$$ is the same as $(-1)^b$, then we get $$x_1^ay_1^b\frac{a_{i_1}(q)}{a_{i_3}(q)}<0.$$ Similarly, if $a$ is odd and $b$ is even, we have $a_{i_1}(q)>0$. When $x_1<0$ and $y_1<0$, we have $$x_1^ay_1^b\frac{a_{i_1}(q)}{a_{i_3}(q)}<0.$$ Besides, if $a$ and $b$ are both odd, we have $a_{i_1}(q)<0$. When $x_1<0$ and $y_1<0$, we have $$x_1^ay_1^b\frac{a_{i_1}(q)}{a_{i_3}(q)}<0.$$ Then if $m_{i_3}$ and $n_{i_3}$ are both even, we could prove $$x_1^ay_1^b\frac{a_{i_1}(q)}{a_{i_3}(q)}<0,$$ and similarly $$x_1^cy_1^d\frac{a_{i_2}(q)}{a_{i_3}(q)}<0.$$ Because the signs of $x_1,y_1$ uniquely determine the sign of $x_{1,(\tau,\sigma)},y_{1,(\tau,\sigma)}$, we know $(x_1,y_1)\in (\mathbb{R}^-)^2$. Thus we get $W_{b_{i_3}}\subset (\mathbb{R}^-)^2$

\item[\textbf{Case 2}]: When $m_{i_3}$ and $n_{i_3}$ are both odd, then $a_{i_3}(q)<0$. Because $ad-bc=1$, $a$ and $b$ could not be even at the same time, and $c$ and $d$ could not be even at the same time as well, we know that $a_{i_1}(q)>0$ and $a_{i_2}(q)>0$. If $x_1>0,y_1>0$, $$x_1^ay_1^b\frac{a_{i_1}(q)}{a_{i_3}(q)}<0$$ $$x_1^cy_1^b\frac{a_{i_1}(q)}{a_{i_3}(q)}<0$$ satisfy the signs of $x_{1,(\tau,\sigma)},y_{1,(\tau,\sigma)}$. Similarly, we have $W_{b_{i_3}}\subset (\mathbb{R}^+)^2$.

\item[\textbf{Case 3}]: $m_{i_3}$ is odd. $n_{i_3}$ is even. 

\item[\textbf{Case 4}]: $m_{i_3}$ is even. $n_{i_3}$ is odd.
\end{enumerate}

Cases 3 and 4 are just similar, we only prove Lemma \ref{l5.2.2} in case 3. Because $m_{i_3}$ is odd, and $n_{i_3}$ is even, $a_{i_3}(q)>0$, where $(a,b)$ also have three kinds of parity. When $a$ is odd, and $b$ is even, $a_{i_1}(q)>0$, then $x_1<0$ and $y_1>0$ make $$x_1^ay_1^b\frac{a_{i_1}(q)}{a_{i_3}(q)}<0.$$ When $a$ is even and $b$ is odd, $a_{i_1}(q)<0$. We have $$x_1^ay_1^b\frac{a_{i_1}(q)}{a_{i_3}(q)}<0.$$ When $a$ and $b$ are both odd, $a_{i_1}(q)>0$. We have $$x_1^ay_1^b\frac{a_{i_1}(q)}{a_{i_3}(q)}<0.$$ Thus we have proved $W_{b_{i_3}}\subset (\mathbb{R}^-)*(\mathbb{R}^+)$. 

\end{proof}

With Lemma \ref{l5.2.2}, we only need to show the domain connected components of $b_{r_1}$ and $b_{r_2}$ have an empty intersection when $(m_{r_1},n_{r_1})$ has the same parity with $(m_{r_2},n_{r_2})$. Without loss of generality, we may assume $r_1=3$ and $b_{r_1}(0,0)$, then $m_{r_2}$ and $n_{r_2}$ are both even. If $(x_0,y_0)\in W_{b_{r_1}}\cap W_{b_{r_2}}$, then we consider $f_1=1$, $f_2=a_{r_2}(q)x_0^{m_{r_2}}y_0^{n_{r_2}}$. Choose a flag $(\tau,\sigma)$ which contains $b_{r_2}$ as the origin. Under the coordinates of $(\tau,\sigma)$, we assume $b_3=b_e'$ has the coordinate $(m_e',n_e')$ under the flag $(\tau,\sigma)$, where $m_e',n_e'$ must be even because $\boldsymbol{b_{i_1}}-\boldsymbol{b_{i_3}}$ and $\boldsymbol{b_{i_2}}-\boldsymbol{b_{i_3}}$ form a $\mathbb{Z}$-basis of $\mathbb{Z}e_1\bigoplus\mathbb{Z}e_2$. We write that $$\boldsymbol{b_{i_1}}-\boldsymbol{b_{i_3}}=a\boldsymbol{e_1}+b\boldsymbol{e_2}$$ $$\boldsymbol{b_{i_2}}-\boldsymbol{b_{i_3}}=c\boldsymbol{e_1}+d\boldsymbol{e_2},$$ and consider $$f_1'=a_e'(q)X_{1,(\tau,\sigma)}^{m_e'}Y_{1,(\tau,\sigma)}^{n_e'}, f_2'=1.$$ Here the coordinate change is that \begin{align*}
    X_{1,(\tau,\sigma)}=x_1^ay_1^b\frac{a_{i_1}(q)}{a_{i_3}(q)}\\Y_{1,(\tau,\sigma)}=x_1^cy_1^d\frac{a_{i_2}(q)}{a_{i_3}(q)}.
\end{align*} Therefore, we know $$f_1=a_{i_3}(q)x_1^{m_{i_3}}y_1^{n_{i_3}}f_1'$$ $$f_2=a_{i_3}(q)x_1^{m_{i_3}}y_1^{n_{i_3}}f_2'.$$ Then we have $$\frac{|f_1|}{|f_2|}=\frac{|f_1'|}{|f_2'|}.$$

When $(x_1,y_1)\in W_{b_{3}}$, and $b_{r_2}$ does not connect $b_3$ directly by any 1-cell in $T_{\Sigma}$, we have proved $$|f_2|<\frac{4}{3}|f_{r_2}(q)|,$$ where $f_{r_2}(q)$ is a non-constant monomial of $q$. Then we have $$\frac{|f_1|}{|f_2|}>\frac{3}{4|f_{r_2}(q)|}>1$$ near the large radius limit. But we could also similarly prove near the large radius limit, $$\frac{|f_2'|}{|f_1'|}>1$$ because $(x_1,y_1)\in W_{b_{r_2}}$. Thus we have the following contradiction 
\begin{align*}
    1>\frac{|f_1'|}{|f_2'|}=\frac{|f_1|}{|f_2|}>1.
\end{align*}
\end{proof}
We have finished the proof of Theorem \ref{t5.2.1}. By Theorem \ref{t5.2.1}, under our choice of coefficients near the large radius limit, any interior point of $P_{\Sigma}$ is the domain term of a unique connected component on the real toric surface which does not intersect axes of $X_{P_{\Sigma}}$. Besides, because there is still one connected component intersecting axes of $X_{P_{\Sigma}}$, we get the following corollary.

\begin{corollary}
\label{c5.2.3}
Near the large radius limit, if $a_i(q)<0$ when $m_i$ and $n_i$ are both odd, and $a_i(q)>0$ when $m_i$ or $n_i$ is even, the mirror curve given by $H(X,Y,q)=0$ is an M-curve. 
\end{corollary}
\subsection{Cyclic Condition}
\label{5.3}
In this chapter, we show that near the large radius limit, under our choices of coefficients, the mirror curve satisfies the cyclic condition.

If $P_{\Sigma}$ contains sides $r_1,...,r_n$, and they correspond to axes $l_1,...,l_n$ in the cyclic order, first, we could show under the special coefficients we have fixed, for any $i$, $l_i$ intersects with a unique connected component of the mirror curve in $d_i$ points, where $d_i$ is the integer length of $r_i$. For the convention, we denote the only connected component intersecting with axes by $W$ and call it the unbounded component.  

\begin{theorem}
\label{t5.3.1}
Near the large radius limit, with $a_i(q)<0$ when $m_i$ and $n_i$ are both odd, and $a_i(q)>0$ when $m_i$ or $n_i$ is even, if $l_i$ is an axis of $\mathbb{R}X_{P_{\Sigma}}$ corresponding to side $r_i$ with integer length $d_i$, then the unbounded component $W$ of the mirror curve $H(X,Y,q)=0$ intersects with $l_i$ in $d_i$ points. \end{theorem}     

\begin{proof}
Given any flag $(\tau,\sigma)$, we consider the affine coordinate change 
\begin{align*}
X_{(\tau,\sigma)}=X^{a(\tau,\sigma)}Y^{b(\tau,\sigma)}\frac{a_{i_1}(q)}{a_{i_3}(q)}, 
\\
Y_{(\tau,\sigma)}=X^{c(\tau,\sigma)}Y^{d(\tau,\sigma)}\frac{a_{i_2}(q)}{a_{i_3}(q)}.  
\end{align*}
Under this coordinate change, the sides of $\Delta_H$ are sent to the sides of $\Delta_{H_{(\tau,\sigma)}}$, and any axis would be sent to a new axis in the new real toric surface $\mathbb{R}X_{\Delta_{H_{(\tau,\sigma)}}}$ defined by $\Delta_{H_{(\tau,\sigma)}}$. Then without loss of generality, we could assume that $l_1$ is the axis corresponding to the side $r_1$ of $P_{\Sigma}$ containing $b_3(0,0), b_{j_k}(0,k)$ as two vertices, and other lattice points in $\Sigma$ outside $r_1$ has the coordinate $(m,n)$ with $m>0$ as shown in Figure \ref{f21}. It is sufficient to show the mirror curve intersects with $l_1$ in $k$ different points.
\begin{figure}
    \centering
    \includegraphics[scale=0.3]{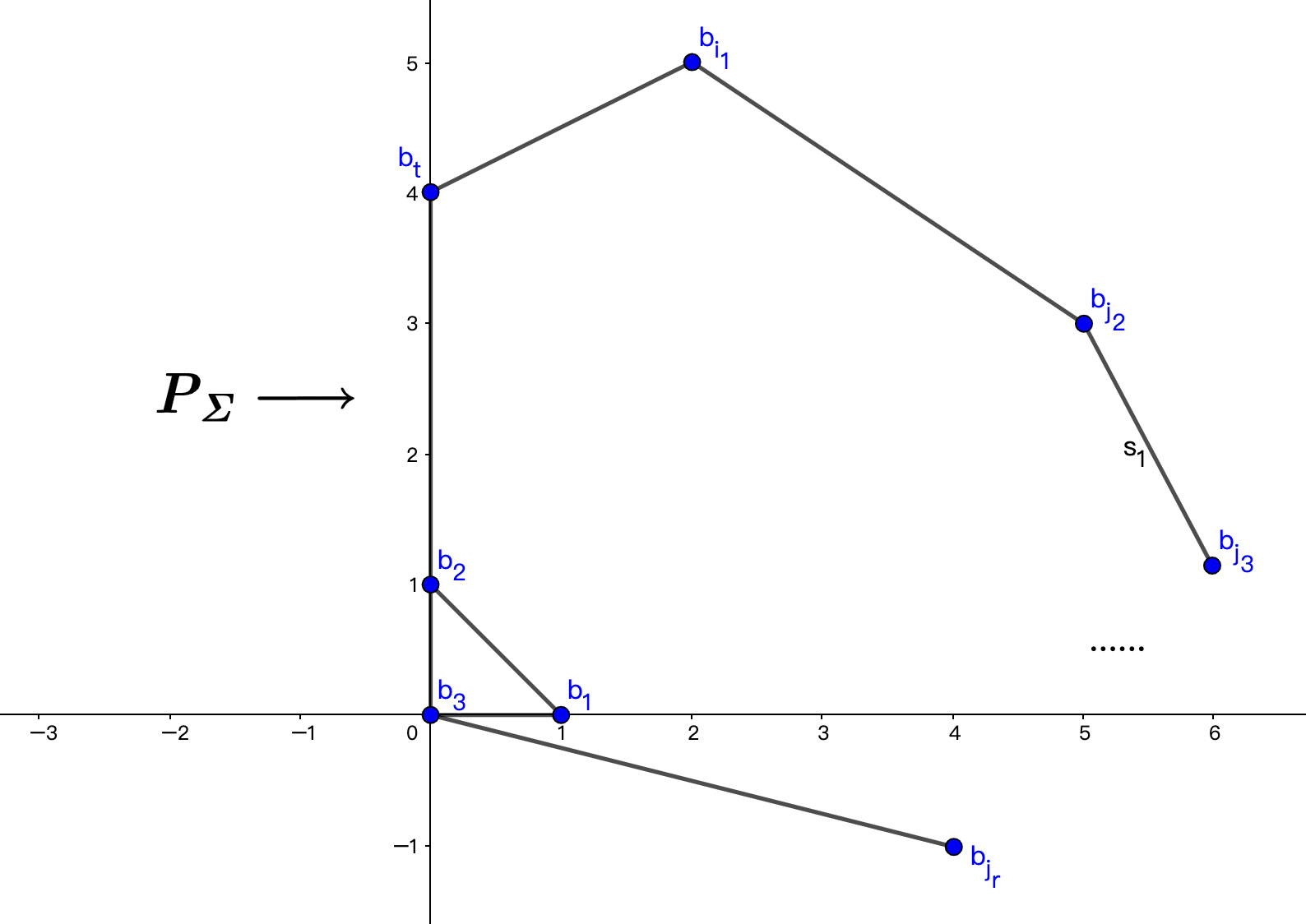}
    \caption{The polytope given by $\Sigma$}
    \label{f21}
\end{figure}

Then $l_1$ is the $y-axis$ at this time, by restricting $H(x,y,q)$ on $l_1$, we get \begin{align*}
    H|_{l_1}(y)=1+y+\sum\limits_{r=2}^ka_{j_r}(q)y^{r}=\sum\limits_{r=0}^ka_{j_r}(q)y^{r}
,\end{align*} where $j_r$ is the index which corresponds to vertex $(0,r)$. But by our choice of coefficients that all $a_{j_r}(q)>0$ with $m_{j_r}=0$, $H|_{l_1}$ only have negative zeroes.

Next, we show that $H|_{l_1}$ has exactly $k$ different negative zeroes near the large radius limit. For any fixed $0\le r\le k-1$, by Proposition \ref{p2.5.3},  $$b_r(q)=\frac{a_{j_{r+1}}(q)a_{j_{r-1}}(q)}{a_{j_{r}}(q)^2}$$ is a non-constant monomial of $q$. Moreover, if we choose any flag $(\tau,\sigma)$ which contains $(0,r)$ as the origin and $I_{\tau}=\{j_r,j_{r+1}\}$, then under such a flag, $b_{j_t}=(0,t)$ has the coordinate $(0,t-r)$. By Proposition \ref{p2.5.4}, we get $$b_{r,t}(q)=\frac{a_{j_t}(q)}{a_{j_r}(q)}*(\frac{a_{j_{r+1}}(q)}{a_{j_r}(q)})^{r-t}$$ is a non-constant monomial of $q$ when $t\ne r,r+1$.

Then we calculate \begin{align*}
    H|_{l_1}(-2\frac{a_{j_r}(q)}{a_{j_{r+1}}(q)})&=\sum\limits_{t=0}^k a_{j_t}(q)(-2\frac{a_{j_r}(q)}{a_{j_{r+1}}(q)})^t\\&=\frac{a_{j_r}(q)^{r+1}}{a_{j_{r+1}}(q)^{r}}(-(-2)^r+\sum\limits_{t\ne r,r+1}(-2)^tb_{r,t}(q)),
\end{align*} and \begin{align*}
    H|_{l_1}(-\frac{1}{2}\frac{a_{j_r}(q)}{a_{j_{r+1}}(q)})&=\sum\limits_{t=0}^k a_{j_t}(q)(-\frac{1}{2}\frac{a_{j_r}(q)}{a_{j_{r+1}}(q)})^t\\&=\frac{a_{j_r}(q)^{r+1}}{a_{j_{r+1}}(q)^{r}}(-(-\frac{1}{2})^{r+1}+\sum\limits_{t\ne r,r+1}(-\frac{1}{2})^tb_{r,t}(q)).
\end{align*} Near the large radius limit, we have $$|\sum\limits_{t\ne r,r+1}(-\frac{1}{2})^tb_{r,t}(q)|<\frac{1}{2^{r+1}}$$ $$|\sum\limits_{t\ne r,r+1}(-2)^tb_{r,t}(q)|<2^r.$$ So $H|_{l_1}(-2\frac{a_{j_r}(q)}{a_{j_{r+1}}(q)})$ has the same sign with $(-1)^{r+1}\frac{a_{j_r}(q)^{r+1}}{a_{j_{r+1}}(q)^{r}}$, and $H|_{l_1}(-\frac{1}{2}\frac{a_{j_r}(q)}{a_{j_{r+1}}(q)})$ has the same sign with $(-1)^{r}\frac{a_{j_r}(q)^{r+1}}{a_{j_{r+1}}(q)^{r}}$. We have shown that they have the opposite signs, by the intermediate value theorem, there exists at least a zero of $H|_{l_1}$ in $(-2\frac{a_{j_r}(q)}{a_{j_{r+1}}(q)},-\frac{1}{2}\frac{a_{j_r}(q)}{a_{j_{r+1}}(q)})$.

Because $b_r(q)=\frac{a_{j_{r+1}}(q)a_{j_{r-1}}(q)}{a_{j_r}(q)^2}$ is a non-constant monomial for any $r$, then near the large radius limit, we have $|b_r(q)|<\frac{1}{5}$. Therefore, $-\frac{1}{2}\frac{a_{j_r}(q)}{a_{j_{r+1}}(q)}<-2\frac{a_{j_{r-1}}(q)}{a_{j_r}(q)}$ for any $r$. We have seen that the $k$ intervals $$(-2\frac{a_{j_r}(q)}{a_{j_{r+1}}(q)},-\frac{1}{2}\frac{a_{j_r}(q)}{a_{j_{r+1}}(q)})\quad r=0,...,k-1$$ don't intersect with each other. Thus we have proved $H|_{l_1}$ has $k$ different zeroes.
\end{proof}

After we get Theorem \ref{t5.3.1}, we also hope that for any axis $l_i$, there exists an arc $u_i$ of the unbounded component $W$ intersects $l_i$ at $d_i$ points and doesn't intersect other axes as we show in Figure \ref{f22}. Actually, we have the following theorem.
\begin{figure}
    \centering
    \includegraphics[scale=0.4]{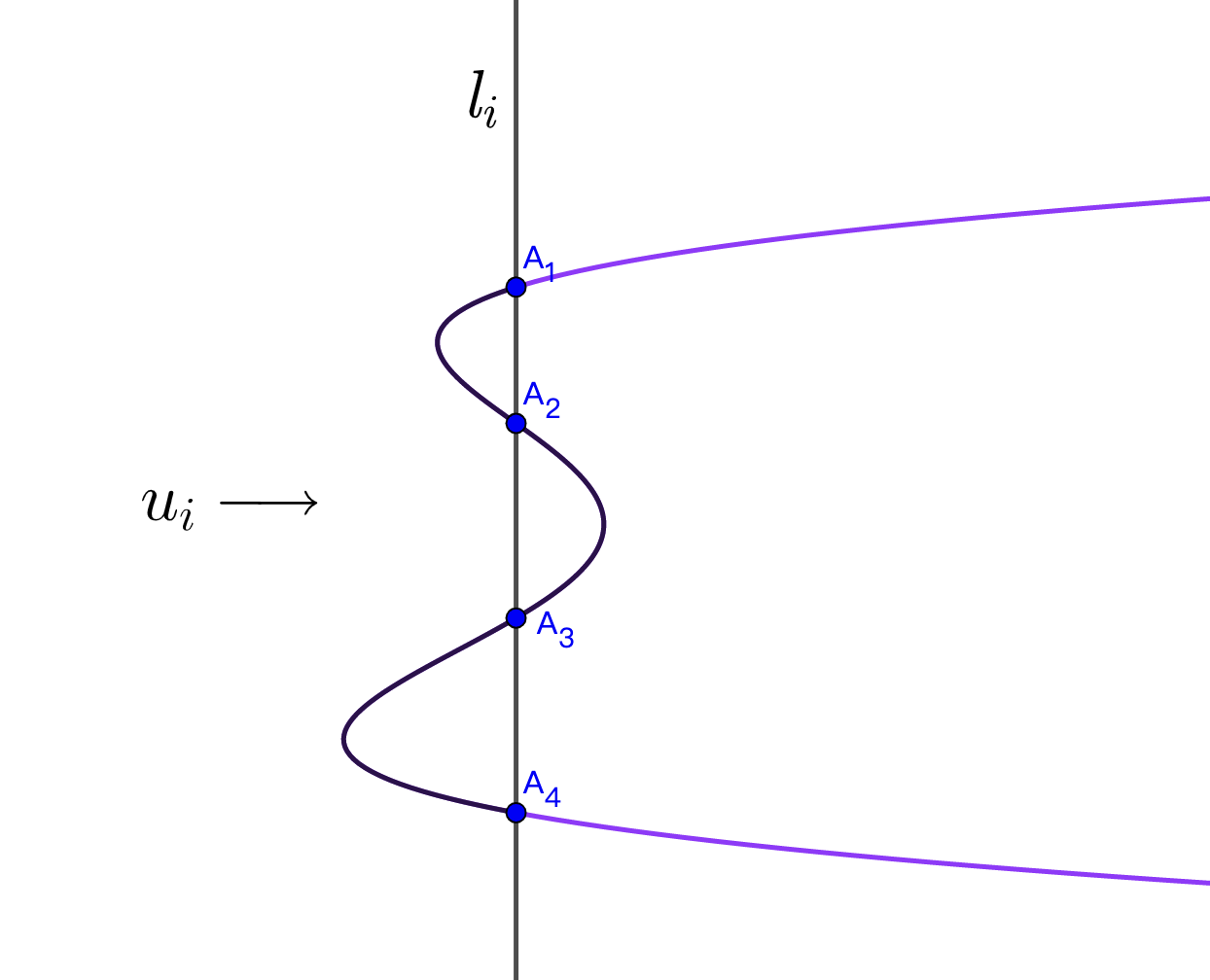}
    \caption{An example when $d_i=4$}
    \label{f22}
\end{figure}

\begin{theorem}
\label{t5.3.2}
Near the large radius limit, with $a_i(q)<0$ when $m_i,n_i$ are both odd and $a_i(q)>0$ when $m_i$ or $n_i$ is even, if $l_i$ is an axis of $\mathbb{R}X_{P_{\Sigma}}$ corresponding to side $r_i$ with integer length $d_i$, then their exists an arc of $W$ which intersecting with the axis $l_i$ in $d_i$ points and not intersecting with other axes.  
\end{theorem}
\begin{proof}
Similarly as the proof of Theorem \ref{t5.3.1}, without loss of generality, we assume $l_1$ is the axis corresponding to the side $r_1$ of $P_{\Sigma}$ containing $b_3(0,0), b_{j_k}(0,k)$ as the two vertices, and other lattice points in $P_{\Sigma}$ outside $r_1$ has coordinate $(m,n)$ with $m>0$. 

We denote $k$ negative zeroes of $H|_{l_1}$ by $s_1$,..., $s_k$, where $s_r$ is the zero lies in $(-2\frac{a_{j_r}(q)}{a_{j_{r+1}}(q)},-\frac{1}{2}\frac{a_{j_r}(q)}{a_{j_{r+1}}(q)})$, and we assume $A_r=(0,s_r)$. Then we hope to construct an arc that only intersects $l_1$ with the points $A_r$ and $A_{r+1}$ and doesn't intersect with the other axes. 

Now let us construct the arc connecting $A_r$ and $A_{r+1}$. We choose a flag $(\tau,\sigma)$ which contains $b_{j_{r}}$ as the origin and makes $I_{\tau}'=\{j_r,j_{r+1}\}$. Consider the affine coordinate change under this flag \begin{align*}
    &X_{(\tau,\sigma)}=X^{a}Y^{b}\frac{a_w(q)}{a_{j_{r}}(q)}
    \\&Y_{(\tau,\sigma)}=Y\frac{a_{j_{r+1}}(q)}{a_{j_{r}}(q)}.
\end{align*} Then for $e=1,...,k$, under the coordinate $(X_{(\tau,\sigma)},Y_{(\tau,\sigma)})$, $A_e=(0,\frac{a_{j_{r}}(q)}{a_{j_{r+1}}(q)}s_e)$. We denote $\frac{a_{j_{r}}(q)}{a_{j_{r+1}}(q)}s_e$ by $s_e'$. Then all zeroes of $H_{(\tau,\sigma)}|_{X_{(\tau,\sigma)}=0}$ are $\{s_e'|e=1,..,k\}$. Because under the flag $(\tau,\sigma)$, $b_{j_r}$ has the coordinate $(0,0)$, and $b_{j_{r+1}}$ has the coordinate $(0,1)$ with $a_{j_r}'(q)=a_{j_{r+1}}'(q)=1$, we know that $s_{r+1}'\in (-2a_{j_{r+2}}'(q),-\frac{1}{2}a_{j_{r+2}}'(q))$, and $s_{r}'\in (-2 ,-\frac{1}{2})$. Now we only need to construct the arc under the coordinate $(X_{(\tau,\sigma)},Y_{(\tau,\sigma)})$. 

We consider all lattice points $b_{l_1}$, ..., $b_{l_w}$ connecting $b_{j_{r+1}}$ directly. Under the flag $(\tau,\sigma)$, $b_{l_c}$ must have the coordinate $(m_{l_c}',n_{l_c}')$ where $m_{l_c}'\ge 0$, and $m_{l_c}',n_{l_c}'$ could not be even at the same time for $c=1,...,w$. Then for any $c=1,..,w$, we have $(-1)^{m_{l_c}'+n_{l_c}'}a_{l_c}^{\sigma}(q)<0$. Thus $a_{l_c}^{\sigma}(q)x^{m_{l_c}'}y^{n_{l_c}'}<0$ when $x<0$ and $y<0$.

For $y\in [s_{r}',s_{r+1}']$, we define $g(y)$ to be the unique negative $s$ which makes $$max\{|a_{l_c}^{\sigma}(q)s^{m_{l_c}'}y^{n_{l_c}'}|\}=3.$$ Here the existence and uniqueness of $s$ are due to the fact $m_{l_c}'\ge0$, and at least one $c$ makes $m_{l_c}'>0$ which makes $$max\{|a_{l_c}^{\sigma}(q)s^{m_{l_c}'}y^{n_{l_c}'}|\}$$ increase monotonically strictly with respect to $|s|$ without an upper bound.
Then we define $A$ to be the index set of the lattice points which do not connect $b_{j_r}$ directly. Similarly as the point $(-1,-1)$ in Figure \ref{f20}, for any point $b_i(m_i',n_i')$, there exists a flag $(\tau_1,\sigma_1)$ containing $b_{j_r}$ as the origin, and the two coordinate components of $b_i$, denoted by $m_i''$ and $n_i''$, are both non-negative. For $y\in[s_{r}',s_{r+1}']$, because $y<0$ and $g(y)<0$, we have $$a_{l_c}^{\sigma}(q)g(y)^{m_{l_c}'}y^{n_{l_c}'}<0 \quad\forall c=1,...,w.$$ Then \begin{align*} H_{(\tau,\sigma)}(g(y),y,q)&=\sum\limits_{i=1}^{p+3}a_i^{\sigma}(q)g(y)^{m_i'}y^{n_i'}\\&=1+\sum\limits_{c=1}^w a_{l_c}^{\sigma}(q)g(y)^{m_{l_c}'}y^{n_{l_c}'}+\sum\limits_{i\in A}a_i^{\sigma}(q)g(y)^{m_i'}y^{n_i'}\\&<1-3+\sum\limits_{i\in A}|a_i^{\sigma}(q)g(y)^{m_i'}y^{n_i'}|.
\end{align*}

By Proposition \ref{p2.5.4}, for any $i$, there exists a non-constant monomial $b_i(q)$ and $c_1,c_2$, makes \begin{align*}
|a_i^{\sigma}(q)g(y)^{m_i'}y^{n_i'}|&=|b_i(q)(a_{l_{c_1}}(q)g(y)^{m_{c_1}'}y^{n_{c_1}'})^{m_i''}(a_{l_{c_2}}(q)g(y)^{m_{c_2}'}y^{n_{c_2}'})^{n_i''}|\\&\le 3^{m_i''+n_i''}|b_i(q)|\\&<3^{2D}|b_i(q)|.
\end{align*} Finally, near the large radius limit, we get $$H_{(\tau,\sigma)}(g(y),y,q)<-2+3^{2D}\sum\limits_{i\in A}|b_i(q)|<0.$$
However, because near the large radius limit, for $-\frac{1}{2}\in [s_r',s_{r+1}']$, it is obvious that $H_{(\tau,\sigma)}(0,-\frac{1}{2},q)>0$, we have proved $$H_{(\tau,\sigma)}(0,y,q)>0\quad \forall y\in(s_{r}',s_{r+1}').$$ By the intermediate value theorem, for any fixed $y\in [s_{r}',s_{r+1}']$, there exists $z\in(g(y),0]$ making $H_{(\tau,\sigma)}(z,y,q)=0$. Denote the non-positive zero of $H_{(\tau,\sigma)}(\quad,y)$ with the least norm by $h(y)$. Then $g(y)\le h(y)\le 0$ for all $y\in [s_r',s_{r+1}']$. Besides, $h(y)=0$ iff $y=s_r'$ or $y=s_{r+1}'$. Finally, because $H_{(\tau,\sigma)}(X,Y,q)$ is a Laurent polynomial with respect to $X,Y$, $h(y)$ is a smooth function with respect to $y\in [s_r',s_{r+1}']$. We have constructed an arc connecting $A_r$ and $A_{r+1}$ with the parameterization given by $(h(y),y) ,\quad y\in [s_r',s_{r+1}']$.   
\end{proof}

Now, it suffices to show there exists an arc connecting two adjacent axes. We have the following theorem.

\begin{theorem}
\label{t5.3.3}
Near the large radius limit, with $a_i(q)<0$ when $m_i$ and $n_i$ are both odd, and $a_i(q)>0$ when $m_i$ or $n_i$ is even, if $l_i$ and $l_{i+1}$ are two axes of $\mathbb{R}X_{P_{\Sigma}}$ corresponding to two adjacent sides $r_i$, $r_{i+1}$ in $P_{\Sigma}$, then there exists an arc of $W$ only intersecting with these two axes.
\end{theorem}
\begin{proof}
Still without loss of generality, we assume $l_i$ and $l_{i+1}$ are the axes corresponding to side $r_i$, $r_{i+1}$, where $r_i$ is the segment having the two vertices $b_3(0,0)$ and $b_{j_k}(0,k)$, and $r_{i+1}$ having the two vertices $b_3(0,0)$ and $b_{l_e}(ew_1,ew_2)$. Here $e$ is an integer, $w_1$ and $w_2$ are coprime, and $w_1>0$, $w_2\le 0$, as we show in Figure \ref{f23}.
\begin{figure}
    \centering
    \includegraphics[scale=0.4]{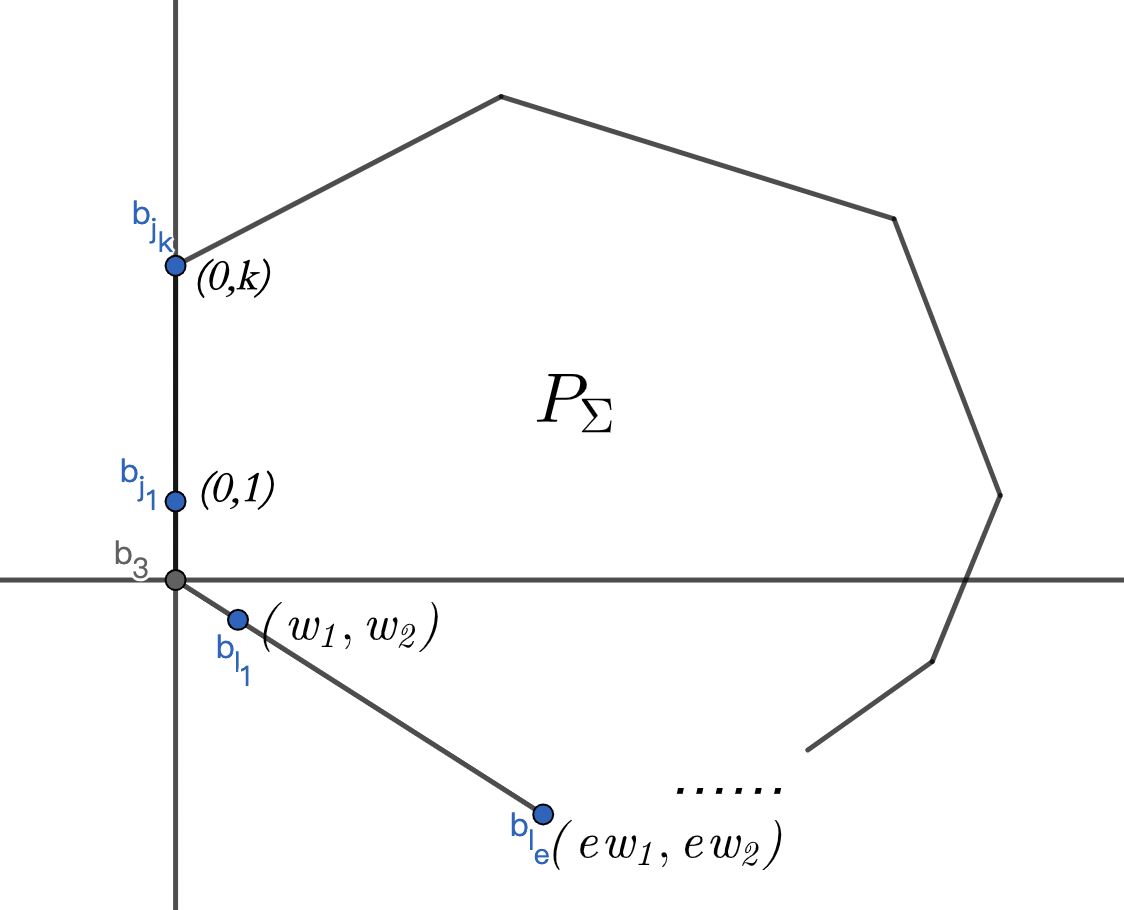}
    \caption{}
    \label{f23}
\end{figure}

Then $l_i$ is the y-axis, because $H|_{y-axis}=1+y+a_{j_2}(q)y^2+...+a_{j_k}(q)y^k$ has the zero with the least norm in $(-2,-\frac{1}{2})$ near the large radius limit. We denote this zero by $s$. Thus $H(0,t,q)>0$ for $s\le t<0$.

Next, we consider all lattice points $b_{l_1}$, ..., $b_{l_w}$ connecting $b_3$ directly, similarly, for any fixed $t\in [s,0)$, $$v(|x|)=max\{|a_{l_c}(q)x^{m_{l_c}}t^{n_{l_c}}|\}$$ increases monotonically strictly with $|x|$ increasing without an upper boundary. We still define $g(t)$ to be the unique negative $x$ which makes $$max\{|a_{l_c}(q)x^{m_{l_c}}t^{n_{l_c}}|\}=3.$$  
Similarly, if we define $A$ to be the index set that corresponds to all lattices which do not connect $b_3$ directly, we have \begin{align*} H(g(t),t,q)&=\sum\limits_{i=1}^{p+3}a_i(q)g(t)^{m_i}t^{n_i}\\&=1+\sum\limits_{c=1}^w a_{l_c}(q)g(t)^{m_{l_c}}t^{n_{l_c}}+\sum\limits_{i\in A}a_i(q)g(t)^{m_i}t^{n_i}\\&<1-3+\sum\limits_{i\in A}|a_i(q)g(t)^{m_i}t^{n_i}|
\end{align*}

However, similarly, by Proposition \ref{p2.5.4}, for any $i$, there exist a non-constant monomial $b_i(q)$ and $c_1,c_2$, making \begin{align*}
|a_i(q)g(t)^{m_i}t^{n_i}|&=|b_i(q)(a_{l_{c_1}}(q)g(t)^{m_{c_1}}t^{n_{c_1}})^{m_i'}(a_{l_{c_2}}(q)g(t)^{m_{c_2}}t^{n_{c_2}})^{n_i'}|\\&\le 3^{m_i'+n_i'}|b_i(q)|\\&<3^{2D}|b_i(q)|.
\end{align*} Finally, near the large radius limit, we get $$H(g(t),t,q)<-2+3^{2D}\sum\limits_{i\in A}|b_i(q)|<0.$$ Then by the intermediate value theorem, letting $h(t)$ be the non-positive zero of $H(\quad,t,q)$ with the least norm, we have $h(t)\in [g(t),0]$. Besides, $h(t)=0$ iff $t=s$.

Finally, we consider the arc $(h(t),t)$ with the parameter $t\in [s,0)$. Because $H(h(t),t,q)=0$ for all $t\in [s,0)$, then \begin{align*}
0&=\sum\limits_{i=1}^{p+3}a_i(q)h(t)^{m_i}t^{n_i}\\&=\sum\limits_{i=0}^{e}a_{l_i}(q)(h(t)^{w_1}t^{w_2})^i+\sum\limits_{j\ne l_i}a_j(q)h(t)^{m_j}t^{n_j}.
\end{align*}
Thus, for any $j\ne l_i$, there exist positive integer numbers $k_{1,j}$, $k_{2,j}$, $k_{3,j}$ which make $k_{1,j}(w_1,w_2)+k_{2,j}(0,1)=k_{3,j}(m_j,n_j)$. Then $$|a_j(q)h(t)^{m_j}t^{n_j}|=|a_j(q)||t|^{\frac{k_{2,j}}{k_{3,j}}}|h(t)^{w_1}t^{w_2}|^{\frac{k_{1,j}}{k_{3,j}}}.$$ 
Therefore, near the large radius limit, if we denote $e$ different zeroes of $$\sum\limits_{i=0}^{e}a_{l_i}(q)x^i$$ by $u_1$, ..., $u_e$ with $|u_1|<...<|u_e|$, where there exist $e$ different zeroes because we could see $$\sum\limits_{i=0}^{e}a_{l_i}(q)x^i$$ as $H$ restricted on axis $l_{i+1}$, we have \begin{align*}
0=H(h(t),t,q)&=\sum\limits_{i=0}^{e}a_{l_i}(q)(h(t)^{w_1}t^{w_2})^i+\sum\limits_{j\ne l_i}a_j(q)h(t)^{m_j}t^{n_j}\\&\le\sum\limits_{i=0}^{e}a_{l_i}(q)(h(t)^{w_1}t^{w_2})^i+\sum\limits_{j\ne l_i}|a_j(q)||t|^{\frac{k_{2,j}}{k_{3,j}}}|h(t)^{w_1}t^{w_2}|^{\frac{k_{1,j}}{k_{3,j}}}.
\end{align*}

Similarly, we could prove \begin{align*}-\sum\limits_{j\ne l_i}|a_j(q)||t|^{\frac{k_{2,j}}{k_{3,j}}}|h(t)^{w_1}t^{w_2}|^{\frac{k_{1,j}}{k_{3,j}}}&\le\sum\limits_{i=0}^{e}a_{l_i}(q)(h(t)^{w_1}t^{w_2})^i\\&\le\sum\limits_{j\ne l_i}|a_j(q)||t|^{\frac{k_{2,j}}{k_{3,j}}}|h(t)^{w_1}t^{w_2}|^{\frac{k_{1,j}}{k_{3,j}}}.
\end{align*}

Because $|h(t)^{w_1}t^{w_2}|\le \frac{3}{|a_{l_1}(q)|}$, take the limit of $t$ tending to $0$, we get $$lim_{t\to 0}\sum\limits_{i=0}^{e}a_{l_i}(q)(h(t)^{w_1}t^{w_2})^i=0.$$ Finally, because $h(t)$ corresponds to the zero with the least norm, we have $$lim_{t\to 0}h(t)^{w_1}t^{w_2}=u_1.$$ If we write such an arc by the homogeneous coordinate, for the parameter $t\in [s,0)$, the arc is given as $(1,1,....,h(t)*t^{\frac{w_2}{w_1}},t^{\frac{1}{w_1}},...,1)$. Here we choose a fixed branch of $x^{\frac{1}{w_1}}$. Then $$lim_{t\to 0}h(t)t^{\frac{w_2}{w_1}}=u_1^{\frac{1}{w_1}}.$$ Because for any branch of $x^{\frac{1}{w_1}}$, we have $(1,1,...,u_1^{\frac{1}{w_1}},0,1,...,1)$ corresponding to the same point in the toric surface (one could see this by considering the group action on $\mathbb{C}^{p+3}$), we could naturally extend the arc with the parameter $t\in[s,0]$. Besides, such an extended arc only intersects with the axes $l_i$, $l_{i+1}$.
\end{proof}

With the previous three theorems, we have shown the corollary given below.
\begin{corollary}
\label{cr5.4}
Near the large radius limit, with $a_i(q)<0$ when $m_i$ and $n_i$ are both odd, and $a_i(q)>0$ when $m_i$ or $n_i$ is even, the mirror curve is a cyclic M-curve.
\end{corollary}
With the corollary, we could give a refinement of the topology of the mirror curve near the large radius limit. Because we have already calculated the tropicalization of the mirror curve near the large radius limit which is dual to the subdivision $T_{\Sigma}$ given by $\Sigma$. For any bounded 2-cell $C$ in the subdivision given by the tropicalization, there exists an oval contains in $C$. Besides, for unbounded 2-cell $C'$, there exists an unbounded convex half circle contains in $C'$, as we show in Figure \ref{f24}.
\begin{figure}
    \centering
    \includegraphics[scale=0.4]{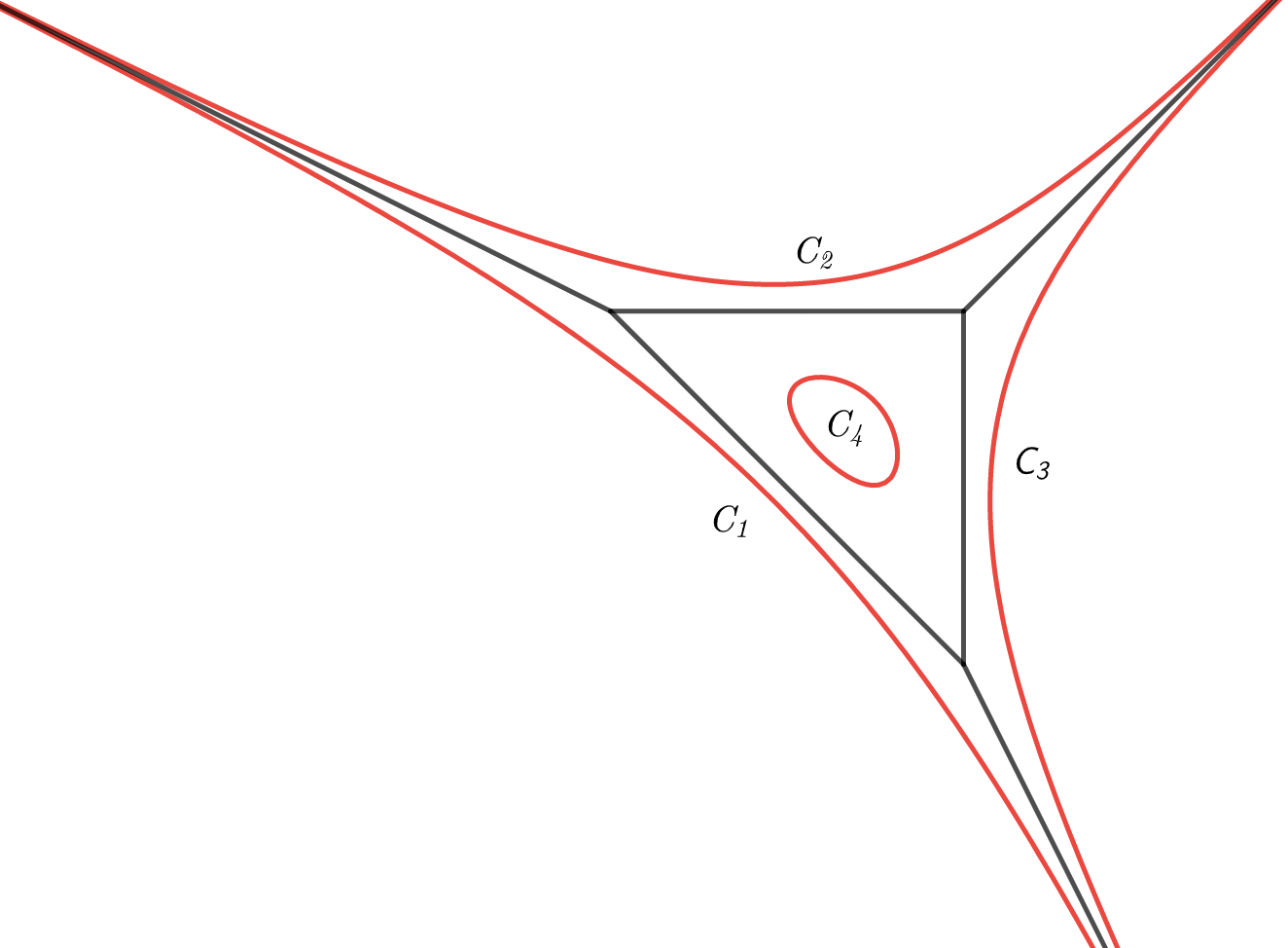}
    \caption{}
    \label{f24}
\end{figure}

Then because the amoeba map is an embedding on the boundary of the amoeba and 2-1 in the interior of the amoeba, we could show that the mirror curve under such special coefficients is glued by tubes and pairs of pants. Any vertex corresponds to a pair of pants, and any edge corresponds to a tube, as we show in Figure \ref{f25} and Figure \ref{f26}. Finally, because $q$ is a flat family of parameters\cite{fang2020remodeling} near the large radius limit, we have seen that, for a general $q$, the mirror curve is glued from tubes and pairs of pants, and these tubes and pairs of pants move smoothly with the change of the parameter $q\in (\mathbb{C^*})^p$ near the large radius limit. 

\begin{aremark}This result is only a refinement because the topological type of a non-singular algebraic curve on a toric surface is uniquely determined by its genus. However, near the large radius limit, we have shown how to glue these tubes and pairs of pants concretely. This gives us the local way to see the topology of the mirror curve.
\end{aremark}

\begin{figure}[htbp]
\centering
\begin{minipage}[t]{0.48\textwidth}
\centering
\includegraphics[width=6cm]{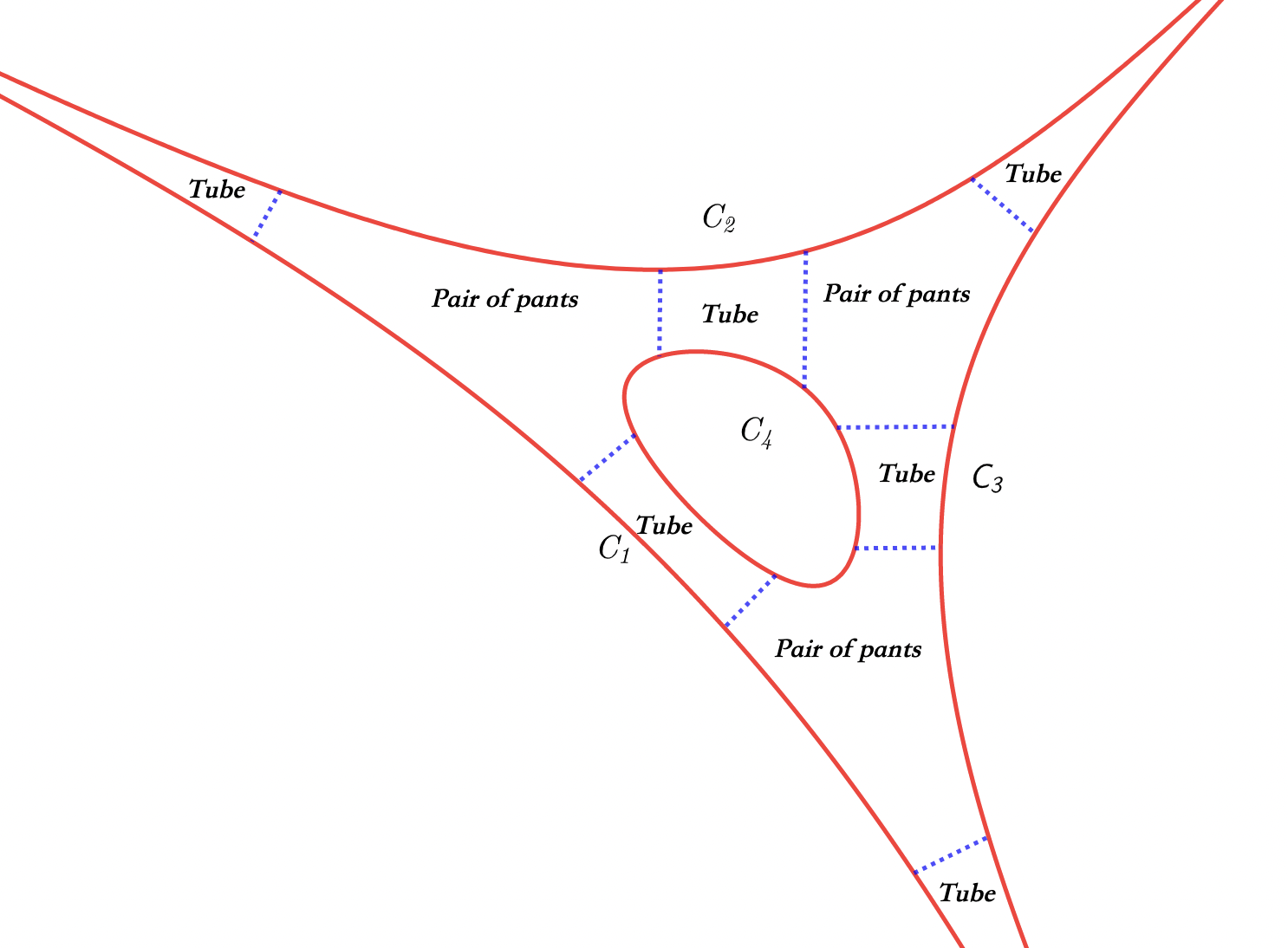}
\caption{$1+x+y-0.01x^{-1}y^{-1}=0$}
\label{f25}
\end{minipage}
\begin{minipage}[t]{0.48\textwidth}
\centering
\includegraphics[width=6cm]{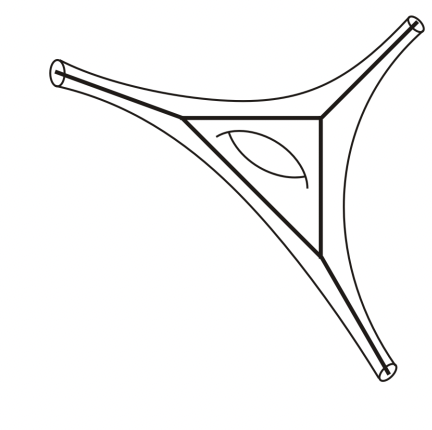}
\caption{Topology of Mirror curve}
\label{f26}
\end{minipage}
\end{figure}

\bibliography{ref}
\end{document}